\documentclass[12pt,psfig]{article}
\usepackage{graphics,epsfig,amssymb}

\usepackage{amsmath}
\usepackage{amsthm}
\usepackage{amsfonts}
\usepackage{amssymb,latexsym}
\usepackage{graphicx}
\usepackage[mathscr]{eucal}
\usepackage{verbatim}
\usepackage{hyperref}
\usepackage{color}

\theoremstyle{plain}
\newtheorem{thm}{Theorem}[section]
\newtheorem{prop}{Proposition}[section]
\newtheorem{lemma}{Lemma}[section]

\newtheorem{rem}{Remark}[section]

\numberwithin{equation}{section}

\usepackage{graphicx}

\begin{document}
\title{Spatio-temporal dynamics for non-monotone semiflows with
limiting systems having spreading speeds}
\author{Taishan Yi\\
School of Mathematics (Zhuhai)\\
Sun Yat-Sen University\\
Zhuhai, Guangdong 519082, China
\and
Xiao-Qiang Zhao\thanks{The corresponding author.
	E-mail: zhao@mun.ca}\\
Department of Mathematics and Statistics\\
Memorial University of Newfoundland\\
St. John's, NL A1C 5S7, Canada}

\date {}
\maketitle

\begin{abstract}
This paper is devoted to the study of propagation dynamics for a large class of non-monotone evolution systems.  In two directions of the
spatial variable, such a system has two limiting systems admitting the spatial translation invariance.  Under the assumption that each of these two
limiting systems has both leftward and rightward  spreading speeds, we establish the spreading properties of solutions
and the existence of  nontrivial fixed points, steady states, traveling waves for the original systems.  We
also apply  the developed theory to  a  time-delayed reaction-diffusion equation with a shifting habitat and
a class of asymptotically homogeneous reaction-diffusion systems.
\end{abstract}


\smallskip


%\tableofcontents

\baselineskip=18pt

%================================================

\section {Introduction}

We start with the following scalar reaction-diffusion equation:
\begin{equation}\label{fisher}
\frac{\partial u}{\partial t}=\Delta u+u(1-u),\quad  x\in\mathbb{R}, \,    \, t\ge 0.
\end{equation}
In 1937, Fisher \cite{f1937}  proved that equation (\ref{fisher}) has a nonnegative
traveling wave $U(x-ct)$ connecting $1$ and $0$  if and only if $c\geq c_{min}=2$.  In the same year,   Kolmogoroff, Petrowsky and Piscounoff \cite{kpp1937} obtained  the similar result on the minimum wave speed when
the reaction term $u(1-u)$ is replaced by a general  monostable function $f(u)$,  and also proved a convergence result.  Aroson and Weinberger \cite{aw1975, aw1978} further established the existence of the asymptotic speed of spread
(in short, spreading speed) in the sense that
any nontrivial and nonnegative solution $u(t,x)$ of equation (\ref{fisher}) with the continuous initial data $u(0,\cdot)$ having compact support admits the following spreading properties:
\begin{enumerate}
	\item[(a)] $\lim_{t\to\infty,|x|\ge ct}u(t,x)=0,\quad \forall c>c^*=2$;
	\item[(b)] $\lim_{t\to\infty,|x|\le ct}u(t,x)=1,\quad \forall
	c\in(0,c^*)$.
\end{enumerate}
Since these seminal works,  there have been extensive investigations on travelling wave solutions and propagation phenomena for various evolution equations, see, e.g., \cite{BHCPAM2012,Bramson1983,DuMatanoJEMS, fm1977, Shen2010,vvv1994, XinReview,yz2020-1} and references therein.

For any $\phi \in BC(\mathbb{R},\mathbb{R}_+)$, let $u(t,x,\phi)$ be the unique
solution of  equation (\ref{fisher}) with $u(0,\cdot,\phi)=\phi$.  Then the time-$t$ map $Q_t$ of
equation (\ref{fisher}) is defined as
$$Q_t[\phi]=u(t,\cdot,\phi),\,   \,  \forall  t\geq 0,  \, \phi \in BC(\mathbb{R},\mathbb{R}_{+}).
$$
For any $y\in \mathbb{R}$, we define  the translation operator $T_y$ by
$$T_y[\phi](x)=\phi(x-y),\, \, \forall x\in \mathbb{R}, \, \phi \in BC(\mathbb{R},\mathbb{R}_+).
$$
Note  that
if $u(t,x)$ is a solution of  equation (\ref{fisher}), then so is
$u(t,x-y)$ for any given $y\in \mathbb{R}$. Moreover,  equation (\ref{fisher}) admits the comparison principle. It then easily follows that each map $Q_t$ has the following properties:
\begin{enumerate}
	\item[(P1)]  (Translation invariance) \, $Q_t\circ T_y=T_y\circ Q_t$, $\forall y\in \mathbb{R}$.
	
	\item[(P2)]  (Monotonicity) \, $Q_t[\phi]\geq Q_t[\psi]$ whenever $\phi\geq \psi$.
	\end{enumerate}

From the viewpoint of dynamical systems,
Weinberger \cite{w1982} established the theory of traveling waves and spreading speeds for monotone discrete-time systems with spatial translation invariance. This theory has been greatly developed
for {\it monotone} discrete and continuous-time semiflows
in \cite{FYZ2017,fz2014,LYZ,lz2007,lz2010,l1989,w2002}
and some {\it non-monotone} systems in \cite{ycw2013,yz2015}
so that it can be applied to a variety of evolution systems in homogeneous or periodic media.
By employing the Harnack inequality up to boundary and the strict positivity of solutions, Berestycki \textit {et al.}~\cite{bhn2005,bhn2010} investigated  the asymptotic spreading speed for KPP equations in  periodic or  non-periodic spatial domains.  A quite different method  was used in \cite{yc2017} to study
the spreading speed and asymptotic propagation for the Dirichlet problem of monostable  reaction-diffusion equations on the half line, see also
 \cite[Section 6.3]{yz2020} for an abstract  approach to such a
problem with time delay.

Motivated by mathematical modeling of  impacts of climate changes
and biological invasions (see, e.g., \cite{bbltc2012,Scheel,k2008,wpcmpbfhb2002}), there have been numerous works on
traveling waves and long-term behavior of solutions for evolution equations with  a shifting habitat, see, e.g.,  \cite{bdnz2009, bf2018, br2008,BR2009, BG2019, DWZ2018, flw2016,FangPengZhao2018,hyz2019,hu2020free,lsz2020,
lbsf2014, lwz2018, pl2004, wz2018,WuWangZou2019, ZhangZhao2019, zhang2017persistence} and references therein.
Another class of evolution equations consists of those being small perturbations of  homogeneous systems or in locally spatially inhomogeneous media (see \cite{Hamel1997I, Hamel1997II, ks2011}). To be more specific,  we consider a scalar reaction-diffusion equation with a shifting habitat:
\begin{equation}\label{shift}
u_t=d \Delta u+f(x-ct,u), \quad  x\in\mathbb{R}, \,    \, t\ge 0,
\end{equation}
where $d>0$, $c \in\mathbb{R}$ is the shifting speed, and $f$ is $C^1$ such that $f(z,0)=0, \forall z\in\mathbb{R}$.
Letting $z=x-ct$ and $v(t,z)=u(t,z+ct)$,  we change equation \eqref{shift} into
the following spatially inhomogeneous reaction-diffusion-advection equation:
\begin{equation}\label{shift2}
v_t=d \Delta v +cv_z+f(z,v), \quad  z\in\mathbb{R}, \,    \, t\ge 0.
\end{equation}
Let $\Phi_t$ be the time-$t$ map of equation \eqref{shift2}. Since
$f(z,v)$ depends on $z$,  it is easy to see that the map $\Phi_t$ does not
admit the translation invariance (P1).

Assume that $\lim_{z\to \pm \infty} f(z,v)=f_{\pm}^{\infty}(v)$ uniformly for $v$ in any bounded subset of
$\mathbb{R}_{+}$.  Then equation \eqref{shift2}  has the following two limiting equations:
\begin{equation}\label{limitE}
v_t=d \Delta v +cv_z+f_{\pm}^{\infty}(v), \quad  z\in\mathbb{R}, \,    \, t\ge 0.
\end{equation}
Let  $\Phi_t^{\pm}$ be the time-$t$ map of equation \eqref{limitE}.  Clearly,
$\Phi_t^{\pm}$ admits the translation invariance (P1).
If, in addition, $f(z,v)$ is nondecreasing in $z\in\mathbb{R}$, then
we easily see that each $\Phi_t$ has the following property:
\begin{enumerate}
	\item[(P3)]  $T_{-y}\circ  \Phi_t[\varphi]\geq \Phi_t\circ T_{-y}[\varphi],
	 \quad     \forall \varphi\in BC(\mathbb{R},\mathbb{R}_+),\,    \,
	 y\in  \mathbb{R}_+$.
\end{enumerate}

Note that systems \eqref{shift2} and \eqref{limitE} admit the comparison principle, and hence, their time-$t$ maps $\Phi_t$ and  $\Phi_t^{\pm}$ still satisfy the monotonicity (P2). In applications, however, we often meet
non-monotone evolution systems.  As a simple example, we consider
a scalar reaction-diffusion equation with time delay:
	\begin {equation}	\label{timedelay}
	u_t=d \Delta u+f(x-ct,u(t-\tau,x)), \quad  x\in\mathbb{R}, \,    \, t\ge 0,
	\end {equation}
where $\tau>0$, $c \in\mathbb{R}$, and  $f$ is $C^1$ such that $f(z,0)=0, \forall z\in\mathbb{R}$. It is well known that system \eqref{timedelay}
does not admit the comparison principle  unless the function $f(z,u)$ is nondecreasing  in $u\in\mathbb{R}_+$ for each $z\in \mathbb{R}$.

Under an abstract setting, recently we  studied the  propagation dynamics for a large class of monotone evolution systems without translation invariance in \cite{yz2020}. More precisely, we established the spreading properties
and the existence of  nontrivial steady states for monotone semiflows
with the property (P3) under the assumption that one limiting system has
both leftward and rightward  spreading speeds (good property G) and the
other one has the uniform asymptotic annihilation (bad property B).

The purpose of our current paper is to investigate the propagation dynamics
for non-monotone semiflows  without the property (P3) in the case where
 two limiting systems have good property  G (i.e., the GG combination).
To overcome the difficulty induced by the lack of (P3), we give
 unilateral estimates of the discrete orbits for a special class of  initial data having compact supports in the case where one limiting system admits the upward convergence property (see Propositions \ref{prop2.1} and \ref{prop2.2}). Another innovation in this
project is the establishment of spreading properties for  nonlinear
maps under the assumption of the global asymptotic stability of a positive fixed point (see Theorem \ref{thm3.1-bil-gas}).  We note that 
the BB combination for two limiting systems has been addressed 
more recently in \cite{yz2023}, and the  GB combination for two limiting systems will be investigated in a forthcoming  paper.

The rest of this paper is organized as follows.  In Section 2,  we  present notations and preliminary results.  In particular, we introduce the
(NM) condition for non-monotone maps and show that it is satisfied automatically whenever  the map is monotone.  And we also establish the links between the system with asymptotic translation invariance and its  limiting systems.
 In Section 3, we prove the existence of fixed points and asymptotic propagation properties for discrete-time semiflows in both unilateral and bilateral
 limit cases, respectively.  In Section 4,
we extend these results to continuous-time semiflows and  a class of
nonautonomous  evolution systems with asymptotic translation invariance, respectively.
In Section 5, we apply the developed  theory to
two prototypical evolution systems in the bilateral limit  case:
 a time-delayed reaction-diffusion equation with a shifting habitat  and a class of asymptotically homogeneous reaction-diffusion systems. Here we leave the study of the  unilateral limit case to readers who are interested in these illustrative examples.
We expect that our developed theory may find more applications in other
types of evolution systems having spatio-temporal heterogeneity such as those
population models with seasonality or nonlocal dispersal.

%======================================================
\section {Preliminaries}
 Let $X=BC(\mathbb{R},\mathbb{R}^N)$ be the normed
vector space of all bounded and continuous functions from
$\mathbb{R}$ to $\mathbb{R}^N$ with the norm $||\phi||_{X}\triangleq
\sum \limits_{n=1}^{\infty}2^{-n}\sup\limits_{|x|\leq n}\{||\phi(x)||_{\mathbb{R}^N}\}$. Let $X_+=\{\phi\in
X:  \,   \phi(x)\in \mathbb{R}_+^N, \,  \forall x\in \mathbb{R}\}$ and
$X_+^\circ=\{\phi\in X:   \, \phi(x)\in Int(\mathbb{R}_+^N),
\,  \forall  x\in \mathbb{R}\}$.
For a given compact metric space $M$,  let
$Y=C(M,\mathbb{R}^N)$ be the Banach space space of all
continuous functions from $M$ into $\mathbb{R}^N$ with the
norm $||\beta||_{Y}\triangleq  \sup\limits_{\theta\in
	M} \{||\beta(\theta)||_{\mathbb{R}^N}\}$ and
$Y_{+}=C(M,\mathbb{R}_+^N)$. Clearly,  $Int(Y_+)=\{\beta\in Y: \, \beta (\theta)\in Int(\mathbb{R}_+^N), \,  \forall  \theta\in M\}$.

Let $C=C(M,X)$ be the
normed vector space of all continuous functions from $M$
into $X$ with the norm $||\varphi||_{C}\triangleq
\sup\limits_{\theta\in M}\{||\varphi(\theta)||_{X}\}$,  $C_+=C(M,X_+)$ and
$C_{+}^{\circ}=C(M,X_{+}^\circ)$. It follows that $C_+$ is a closed
cone in the normed vector space $C$. Note that $C_+^\circ\neq Int(C_+)$ due to
the non-compactness of the spatial domain $\mathbb{R}$.
For the sake of convenience, we identify an element $\varphi\in C$
with a bounded and continuous function from $M\times
\mathbb{R}$ into $\mathbb{R}^N$. Accordingly,  we can also regard $X$ and $Y$ as subspaces of $C$.

 For any $\xi$, $\eta \in \mathbb{R}^N$,  we write $\xi\geq_{\mathbb{R}^N} \eta$ if $\xi-\eta
 \in \mathbb{R}^N_+$; $\xi >_{\mathbb{R}^N} \eta$ if $\xi\geq_{\mathbb{R}^N} \eta$ and $\xi\neq \eta$;
 $\xi \gg_{\mathbb{R}^N} \eta$ if $\xi- \eta \in Int(\mathbb{R}^N_+)$.
 Similarly, we use $Y_+$ and $Int(Y_+)$ to define $\geq_Y$, $>_Y$ and $\gg_Y$ for the  space $Y$.
 For any $\xi$, $\eta \in X$,  we write $\xi\geq_X \eta$ if $\xi-\eta
 \in X_+$; $\xi >_X \eta$ if $\xi\geq_X \eta$ and $\xi\neq \eta$;
 $\xi \gg_X \eta$ if $\xi- \eta \in X_+^\circ$.
We also employ  $C_+$ and $C_+^\circ$ to define $\geq_C$ $>_C$, $\gg_C$ for the space $C$ in a similar way. For simplicity, we  write
$\geq$, $>$, $\gg$, and $||\cdot||$, respectively, for $\geq_*$,
$>_*$, $\gg_*$, and $||\cdot||_*$, where $*$ stands  for one of ${\mathbb{R}^N}$, $X$,
$Y$ and $C$.

For any given $s,r\in Int(\mathbb{R}_+^N)$ with $s\geq r$, define $C_r=\{\varphi\in
C:0\leq\varphi\leq r\}$ and $C_{r,s}=\{\varphi\in C:r\leq\varphi\leq
s \}$. For any given $\varphi\in C_+$, define $C_{\varphi}=\{\psi\in C:
0\le \psi\le \varphi\}$.
We also  define $[\varphi,\psi]_*=\{\xi\in *:\varphi\leq_*\xi\leq_*\psi\}$ and  $[[\varphi,\psi]]_*=\{\xi\in *:\varphi\ll_*\xi\ll_*\psi\}$ for $\varphi,\psi\in *$ with $\varphi\leq_*\psi$, where $*$ stands  for one of ${\mathbb{R}^N}$, $X$,
$Y$ and $C$.  Let $\check{1}:=(1,1,\cdots,1)^{T}\in \mathbb{R}^N$ and $\check{\bf 1}\in \mathbb{R}^{N\times N}$ with $(\check{\bf 1})_{ij}=1, \forall i,j\in \{1,2,\cdots,N\}$.
For any given $y\in\mathbb{R}$, we define the spatial translation operator
$T_{y}$ on $C$ by
$$T_{y}[\varphi](\theta,x)=\varphi(\theta,x-y),\quad \forall \varphi\in C,\, \theta\in M, \, x\in\mathbb{R}.
$$

As a convention of this paper, we say a  map ${Q}:  C_+\to C_+$ is continuous
if for any given  $r\in Int(\mathbb{R}_+^N)$, $Q|_{C_r}:C_r\to C_{+}$ is continuous in the usual sense.  The map ${Q}:  C_+\to C_+$ is said to be monotone if ${Q}[\phi]\leq {Q}[\psi]$ whenever $\phi,\psi\in  C_{+}$ with
$\phi\leq \psi$.

Throughout the whole paper,  unless specified otherwise, we fix a continuous map $Q: C_+\to C_+$, and always assume that 
\begin{enumerate}
\item [{\bf (A)}] There exist $r^*\in Int(Y_+)$ and a continuous map $Q_+:  C_+\to C_+$ such that $Q_+[r^*]=r^*$ and 
$\lim\limits_{y\to\infty}T_{-y}\circ Q^n \circ T_y[\varphi]=Q_+^n[\varphi]$  in $C$ for all $\varphi\in C_+$ and $n\in \mathbb{N}$.
\end{enumerate}

The following observation is a straightforward  consequence of  assumption {\bf (A)}.
\begin{lemma}\label{lemm2.1}
The map  $Q_+:C_+\to C_+$ admits the following properties:
\begin{enumerate}
\item [(i)] $Q_+[\varphi](\theta,x)=\lim\limits_{y\to\infty}T_{-y}\circ Q \circ T_y[\varphi](\theta,x), \forall  \varphi\in C_+, (\theta,x)\in
M\times  \mathbb{R}$.

\item  [(ii)] $T_y[Q_+[\varphi]]=Q_+[T_y[\varphi]]$ for all $(y,\varphi)\in
\mathbb{R}\times C_{+}$.

\item   [(iii)] $Q_+$ is monotone whenever $Q$ is monotone.
\end{enumerate}
\end{lemma}

The following upward convergence was requested  for 
the map $Q_+$  in \cite{yz2020}.  In the main results of the current paper, however, we do not directly  assume that $Q_+$ has property ({\bf UC}).

\begin{enumerate}
\item [{\bf (UC)}]  There exist $c_-^*,c_+^*\in \mathbb{R}$   such that $c_+^*+c_-^*>0$, and
$$
\lim\limits_{n\rightarrow \infty}
\max\limits_{x\in \mathcal{A}_{\varepsilon,n}^+}
||Q_+^n[\varphi](\cdot,x)-r^*(\cdot)||=0, \, \forall \varepsilon\in (0,\frac{c_+^*+c_-^*}{2}), \, \varphi\in
C_{+}\setminus \{0\},
$$
 where $\mathcal{A}_{\varepsilon,n}^+=n{[-c_-^*+\varepsilon,c_+^*-\varepsilon]}$.

\end{enumerate}

%{\color{blue} If $Q_+$ satisfies  {\bf (UC)} and (AA) with  $c_-^*=\bar{c}_-$ and $c_+^*=\bar{c}_+$, then $c_-^*$ and $c_+^*$
%are called the leftward and rightward spreading speeds, respectively, for the discrete-time monotone system $\{Q_+^n\}_{n\geq 0}$. For the general results on the existence
%of spreading speeds for monotone semiflows, we refer to \cite{w1982,lz2007, LYZ,lz2010,fz2014}.}

Define a continuous function $h:M\times \mathbb{R}\rightarrow \mathbb{R}$ by
%$$h(\theta,x)=\max\{0,\min\{1,2-|x|\}\},  \quad  \forall  (\theta,x)\in M \times  \mathbb{R}.$$
\[
		h(\theta,x)=\left\{
		\begin{array}{ll}
		1, & x\in [-1,1],
		\\
		0, \qquad & x\in (-\infty,-2]\cup [2,\infty),
		\\
		2-x, \qquad & x\in (1,2),
		\\
		 2+x, & x\in (-2,-1).
		\end {array}
		\right.
		\]

By the arguments similar to those for  Proposition 2.1 in \cite{yz2020},  we  can prove the first part of the following result, while  the second part easily follows from the first one.

\begin{prop} \label{prop2.1}
Assume  that $Q_+$ satisfies {\bf (UC)}.
Then for any $\varepsilon\in (0,\frac{3c_+^*+3c_-^*}{4})$, there exist $n_0:=n_0({\varepsilon})\in \mathbb{N} \cap [\frac{6}{\varepsilon},\infty)$ and $y_0:=y_0({\varepsilon})>2$ such that
$$
T_{-n c}\circ T_{-y} \circ Q^{n}\circ T_{y}[\frac{r^*}{16}h]\geq \frac{r^*}{4}h,
 \quad  \forall
 (c,n,y)\in [-c_-^*+\frac{2\varepsilon}{3},c_+^*-\frac{2\varepsilon}{3}]\times [n_0,2n_0] \times [y_0,\infty).
$$
If,  in addition, $Q$ is monotone, then
$$
T_{-y}\circ (T_{-n_0 c}\circ Q^{n_0})^n\circ T_{y}[\frac{r^*}{16}h]\geq \frac{r^*}{4}h,  \quad  \forall
(c,n,y)\in [-c_-^*+\frac{2\varepsilon}{3},c_+^*-\frac{2\varepsilon}{3}]\times \mathbb{N}\times [y_0,\infty).
$$
\end{prop}

Following \cite{Zhaobook}, we say $Q$  is a subhomogeneous map on $[0,r^*]_C$ if $Q[\kappa \phi]\geq \kappa Q[\phi]$ for all $(\kappa,\phi)\in [0,1]\times [0,r^*]_C$.

By modifying the proof of \cite[Proposition 2.2]{yz2020}, we can
obtain the following result without assuming 
the translation monotone property (P3).

\begin{prop} \label{prop2.2}
Assume  that $Q$ is monotone and  $Q_+$ satisfies {\bf (UC)}. Let $c_+^*>0$, $\varepsilon\in (0,\frac{1}{2}\min\{c_+^*,c_+^*+c_-^*\})$, and let  $n_0:=n_0({\varepsilon})$ and  $y_0:=y_0({\varepsilon})$ be  defined as in Proposition~\ref{prop2.1}. Then
$$
T_{-n c}\circ T_{-y} \circ Q^{n}\circ T_{y}[\frac{r^*}{16}h]\geq \frac{r^*}{4}h,
\quad  \forall
c\in [\max\{0,-c_-^*+\frac{2\varepsilon}{3}\},\, c_+^*-\frac{2\varepsilon}{3}], \, n\geq n_0, y\geq y_0.
$$
 If, in addition,   $Q$ is subhomogeneous on $[0,r^*]_C$, then
for any $\delta\in [0,1]$,  there holds
$$
T_{-n c}\circ T_{-y} \circ Q^{n}\circ T_{y}[\frac{\delta r^*}{16}h]\geq \frac{ \delta r^*}{4}h,
\quad  \forall
c\in [\max\{0,-c_-^*+\frac{2\varepsilon}{3}\},\, c_+^*-\frac{2\varepsilon}{3}], \, n\geq n_0, \, y\geq y_0.
$$

% \item [{\rm (ii)}]  {\color{blue} If $Q$ satisfies {\bf (SP)} and is  subhomogeneous on $[0,r^*]_C$, then there exist $\kappa_0=\kappa_0(\varepsilon)\in (0,1]$ and $N_0(\varepsilon)>0$ such that for any $\delta\in (0,1)$,  there holds $T_{-n c} \circ Q^{n}\circ T_\alpha [\frac{\delta  r^*}{16}h]\geq \frac{\delta \kappa_0 r^*}{4}h$ for all
%$c\in [\max\{\varepsilon,-c_-^*+\varepsilon\},c_+^*-\varepsilon]$,{ $\alpha\in [-3,3]$, }and $n\geq N_0(\varepsilon)$.}%oš®??2šŠš®?š¢??a???š¢???e¡ê?

\end{prop}

\noindent
\begin{proof} %Since $Q$ is monotone, so is  $Q_+$ due to Lemma~\ref{lemm2.1}-(iii).
Define
\begin{eqnarray*}
&&n^*:=\sup\left\{k\in [n_0,\infty)\cap \mathbb{N}: \quad T_{-n c-y} \circ Q^{n}\circ T_{y}[\frac{r^*}{16}h]\geq \frac{r^*}{4}h, \right.\\
&&\qquad \qquad \quad  \left.
\forall  (c,n,y)\in [\max\{0,-c_-^*+\frac{2\varepsilon}{3}\},c_+^*-\frac{2\varepsilon}{3}]\times [n_0,k]\times [y_0,\infty)\right\}.
\end{eqnarray*}
Clearly, $n^*\geq 2n_0$ due to Proposition~\ref{prop2.1}. It suffices to prove $n^*=\infty$. Otherwise, we have $n^*<\infty$.  By Proposition~\ref{prop2.1} and the choices  of $y_0,n_0$ and $n^*$, it follows that for any $
c\in [\max\{0,-c_-^*+\frac{2\varepsilon}{3}\},c_+^*-\frac{2\varepsilon}{3}]$ and $y\geq y_0$, we have
\begin{eqnarray*}
&&T_{-(n^*+1) c-y} \circ Q^{n^*+1 }\circ T_{y}[\frac{r^*}{16}h]
\\
&&=T_{-n_0c-[(n^*+1-n_0)c+y]}\circ Q^{n_0}\circ T_{(n^*+1-n_0)c+y}[ T_{-(n^*+1-n_0)c-y}\circ Q^{n^*+1-n_0}\circ T_{y}[\frac{r^*}{16}h]]
\\
&&\geq T_{-n_0c-[(n^*+1-n_0)c+y]}\circ Q^{n_0}\circ T_{(n^*+1-n_0)c+y}[ \frac{r^*}{16}h]\geq \frac{r^*}{4}h,
\end{eqnarray*}
which  contradicts the choice of  $n^*$.
For any $\delta\in [0,1]$,  the subhomogeneity  of $Q$ implies that
$$
T_{-n c}\circ T_{-y} \circ Q^{n}\circ T_{y}[\frac{\delta r^*}{16}h]
\geq \delta T_{-n c}\circ T_{-y} \circ Q^{n}\circ T_{y}[\frac{ r^*}{16}h]
\geq \delta \frac{  r^*}{4}h,
$$
for all
$c\in [\max\{0,-c_-^*+\frac{2\varepsilon}{3}\},c_+^*-\frac{2\varepsilon}{3}]$, $n\geq n_0$, and  $y\geq y_0$.
 \end{proof}

 %$\frac{nc-y_0^*}{n-N_1^*}\in  [\max\{0,-c_-^*+\frac{2\varepsilon}{3}\},c_+^*-\frac{2\varepsilon}{3}]$ $\frac{nc-y_0^*}{n-N_1^*}\geq -c_-^*+\frac{2\varepsilon}{3}$ iff $nc-y_0^*\geq (n-N_1^*)[-c_-^*+\frac{2\varepsilon}{3}]=-n c_-^*+N_1^* c_-^*+\frac{2n \varepsilon}{3}-\frac{2N_1^*\varepsilon}{3}$ $nc-y_0^*\geq -n c_-^*+N_1^* c_-^*+\frac{2n \varepsilon}{3}-\frac{2N_1^*\varepsilon}{3}$ iff $n[c+c_-^*-\varepsilon]-y_0^*\geq N_1^* c_-^*-\frac{n \varepsilon}{3}-\frac{2N_1^*\varepsilon}{3}$!$\frac{nc-y_0^*}{n-N_1^*}\leq c_+^*-\frac{2\varepsilon}{3}$ iff ${nc-y_0^*}\leq [c_+^*-\frac{2\varepsilon}{3}]({n-N_1^*})=n c_+^*-N_1^* c_+^*-\frac{2n\varepsilon}{3} +\frac{2 N_1^* \varepsilon}{3} $ $\frac{nc-y_0^*}{n-N_1^*}\leq c_+^*-\frac{2\varepsilon}{3}$ iff $n(c-c_+^*+\varepsilon)-y_0^*\leq -N_1^* c_+^*+\frac{n\varepsilon}{3} +\frac{2 N_1^* \varepsilon}{3} $!

It is easy to verify the  following properties for $Q_+$.
\begin{prop} \label{prop2.30000}
Assume that $Q_+$ satisfies {\bf (UC)}.
Then the following statements are valid:
\begin{itemize}
\item [{\rm (i)}]  $Q_+^n[\alpha r^*](\theta,x)=Q_+^n[\alpha r^*](\theta,0),
\,  \, \forall
(n,\alpha,\theta,x)\in \mathbb{N} \times  \mathbb{R}_+\times M \times \mathbb{R}.$

\item [{\rm (ii)}]  $\lim\limits_{n\to \infty}Q_+^n[\alpha r^*]= r^*$ in $L^\infty(M\times \mathbb{R},\mathbb{R}^N), \,  \, \forall \alpha\in (0,\infty)$.

%\item [{\rm (iii)}]  $Q^n[r^*]$   is nonincreasing in $n\in  \mathbb{N}$.

\end{itemize}
\end{prop}

The following assumption is quite different from the assumption {\bf (SP)}  in \cite{yz2020}, although we still denoted it by {\bf (SP)}
for the convenience in applications.
	 
\begin{enumerate}
\item [{\bf (SP)}]  There exist $(\rho^*,N^*)\in (0,\infty)\times \mathbb{N}$ and  $\varrho^*\in (\rho^*,\infty)$  such that for any  $a\in \mathbb{R}$  and $\varphi\in
 C_{+}\setminus \{0\}$ with   $\varphi(\cdot, a)\subseteq Y_+\setminus \{0\}$,
we have  $Q^n[\varphi](\cdot,x)\in Int(Y_+)$ for all  $n\geq N^*$ and $x-a\in [n\rho^*,n\varrho^*]$.
\end{enumerate}

\begin{lemma}\label{lemm2.1-3.1-4.1} %Let $r^*\in Int(Y+)$, and 
If  the map  $Q:C_+\to C_+$ satisfies {\bf (SP)}  and $K$ is a compact subset of $C_{+}\setminus\{0\}$, then for any $\mathfrak{y}>0$, there exist $\delta_0=\delta_0(K,\mathfrak{y})\in (0,1)$, $\mathfrak{y}_0=\mathfrak{y}_0(K,\mathfrak{y})>2+\mathfrak{y}$, and $\mathfrak{N}_0=\mathfrak{N}_0(K,\mathfrak{y})\in \mathbb{N}$ such that
$T_{-\mathfrak{y}_0}Q^{\mathfrak{N}_0}[\varphi]\geq \delta_0 r^* h$ for all $\varphi\in K$.
\end{lemma}

\begin{proof}
Since $K$ is a compact subset of $C_{+}\setminus\{0\}$, there exist $j_0\in \mathbb{N}$ and  a finite sequence $\{(a_j,O_j)\}$ with
$a_j\in \mathbb{R}$ and $O_j\subset K$  for all $1\leq j\leq j_0$ such that
$K=\bigcup\limits_{j=1}^{j_0} O_j$ and
$$
\varphi(\cdot,a_j)\subseteq Y_+\setminus \{0\},
 \,  \forall \varphi\in O_j,    \,  1\leq j\leq j_0.
 $$
It follows from  {\bf (SP)} that $Q^n[\varphi](\cdot,x)\subseteq Int(Y_+)$ for all  $n\in \mathbb{N}\cap [N^*,\infty)$,
$x\in a_j+ [n\rho^*,n\varrho^*]$, and  $\varphi\in O_j$ with $j\in  \mathbb{N}\cap [1,j_0]$. Let $I_n= \bigcap\limits_{j=1}^{j_0}[a_j+n\rho^*,a_j+n\varrho^*]$.
Then we have
$$
I_n\supseteq\left [\sum\limits_{j=1}^{j_0}|a_j|+n\rho^*,-\sum\limits_{j=1}^{j_0}|a_j|+n\varrho^* \right]\neq \emptyset,
 \,  \forall
n>\mathfrak{n}^*:=\frac{4+\mathfrak{y}+2\sum\limits_{j=1}^{j_0}|a_j|}{\varrho^*-\rho^*}.
$$
Consequently,  for any $n\geq \mathfrak{N}_0:=1+N^*+\mathfrak{n}^*$,
there holds
$$
Q^n[\varphi](\cdot,x)\subseteq Int(Y_+),
\,  \forall  x\in\left [\sum\limits_{j=1}^{j_0}|a_j|+n\rho^*,-\sum\limits_{j=1}^{j_0}|a_j|+n\varrho^*\right],\,  \varphi\in K.
$$
In particular,  there exists  $\mathfrak{y}_0=\mathfrak{y}_0(K)>2+\mathfrak{y}$ such that
\[
Q^{\mathfrak{N}_0}[\varphi](\cdot,x)\subseteq Int(Y_+),
\,  \forall  x\in [\mathfrak{y}_0-2,\mathfrak{y}_0+2],\,  \varphi\in K.
\]
By virtue of  the compactness of $K$, there exists $\delta_0\in (0,1)$ such that $$Q^{\mathfrak{N}_0}[\varphi](\cdot,x)\geq \delta_0 r^*,
\,  \forall  x\in [\mathfrak{y}_0-2,\mathfrak{y}_0+2],\,  \varphi\in K.$$
This,  together with the definition of $h$,  completes the proof.
\end{proof}

In Lemma \ref{lemm2.1-3.1-4.1}, for any $\mathfrak{y}>0$ we emphasize  the importance of $\mathfrak{y}_0>2+\mathfrak{y}$, not just its existence, which is helpful for later discussions such as the proof of Theorem~\ref{thm3.1}.

For any $(d,\varrho)\in (0,\infty)\times [1,\infty)$, we define two functions:
$$
\xi_{d}(\theta,x):=\max\{0,\min\{1,d+1-|x|\}\},\quad
 \forall (\theta,x)\in M\times \mathbb{R},
 $$
and
$$
\tilde{\xi}_{d,\varrho}(\theta,x):=\min\{\varrho,\max\{1,(\varrho-1) |x| -\varrho d+d+1\}\},\quad \forall (\theta,x)\in M\times \mathbb{R}.
$$
Let $r^*\in Int(Y_+)$ be given.
For any $(d,r,s,z,n,\varrho)\in Int(\mathbb{R}_+^6)$ with $s\geq r$, $n\in \mathbb{N}$ and $\varrho\geq 1$, we
%let $\sum_{d,r,s,\varrho}:=[r\xi_{d}r^{*},s\tilde{\xi}_{d,\varrho}r^{*}]_C$, and
introduce two numbers:
 $$I_{d,r,s,z,n,\varrho}:=
\sup\{\beta\geq 0:  \, T_{-y}\circ Q^{n} \circ T_{y}[\varphi](\cdot,0)\geq \beta r^*, \forall \varphi \in [r\xi_{d}r^{*},s\tilde{\xi}_{d,\varrho}r^{*}]_C, y\in [z,\infty)\}$$
and
$$
S_{d,r,s,z,n,\varrho}:=\inf\{\beta\geq 0:  \, T_{-y}\circ Q^{n} \circ T_{y}[\varphi](\cdot,0)\leq \beta r^*, \forall \varphi \in [r\xi_{d}r^{*},s\tilde{\xi}_{d,\varrho}r^{*}]_C,  y\in [z,\infty)\}.
$$

In order to  study  the nonlinear map $Q$  with $Q_+[r^*]=r^*$ 	for some $r^*\in Int(Y_+)$, we need
the following assumption.

\begin{enumerate}
\item [{\bf (NM)}]
If $t\geq s\geq r>0$ and $\{r,s\}\neq \{1\}$, then there exist $n(r,s)\in \mathbb{N}$, $\gamma(r,s)\in (0,r) $ and $(d(r,s,t),z(r,s,t))\in Int(\mathbb{R}_+^2)$
such that either  $$I_{d(r,s,t),r-\gamma(r,s),s+\gamma(r,s),z(r,s,t),n(r,s),\frac{t+\gamma(r,s)}{r-\gamma(r,s)}}\geq r+\gamma(r,s)$$ or $$S_{d(r,s,t),r-\gamma(r,s),s+\gamma(r,s),z(r,s,t),n(r,s),\frac{t+\gamma(r,s)}{r-\gamma(r,s)}}\leq s-\gamma(r,s).$$
\end{enumerate}

\begin{lemma}\label{lemm4.1} If $Q$ is monotone and $Q_{+}$ satisfies {\bf (UC)}, then {\bf (NM)} holds true.
\end{lemma}

\begin{proof}  Fix $r,s,t\in (0,\infty)$ with $t\geq s \geq r$ and $\{r,s\}\neq \{1\}$. By Proposition~\ref{prop2.30000}-(ii), we easily see that there exist  $n(r,s)\in \mathbb{N}$ and $\gamma(r,s)\in (0,r)$ such that in the
case  $r<1$,
$$
Q_+^{n(r,s)}[(r-\gamma(r,s))r^*]>(r+3\gamma(r,s))r^*\in C_{r^*};
$$
and in the case where $s>1$,
$$
Q_+^{n(r,s)}[(s+\gamma(r,s))r^*]<(s-3\gamma(r,s))r^* \in r^*+C_+.
$$
By the continuity of $Q_+$, it then follows that there exists
 $d(r,s,t)\in (0,\infty)$ such that
 $$
 Q_+^{n(r,s)}[(r-\gamma(r,s))\xi_{d(r,s,t)}r^*](\cdot,0)>(r+2\gamma(r,s))r^*
 $$
 when  $r<1$;  and
  $$
  Q_+^{n(r,s)}[(s+\gamma(r,s))\tilde{\xi}_{d(r,s,t),\frac{t+\gamma(r,s)}{r-\gamma(r,s)}}r^*](\cdot,0)<(s-2\gamma(r,s))r^*
  $$
  when $s>1$.  In view of {\bf (A)},  there exists $z(r,s,t)\in (0,\infty)$ such that $$
  T_{-y}\circ Q^{n(r,s)} \circ T_{y}[(r-\gamma(r,s))\xi_{d(r,s,t)}r^*](\cdot,0)>(r+\gamma(r,s))r^*,  \quad
   \forall y\in [z(r,s,t),\infty)
   $$
 in the case  $r<1$;  and
 $$
 T_{-y}\circ Q^{n(r,s)} \circ T_{y}[(s+\gamma(r,s))\tilde{\xi}_{d(r,s,t),\frac{t+\gamma(r,s)}{r-\gamma(r,s)}}r^*](\cdot,0)<(s-\gamma(r,s))r^*, \quad \forall y\in [z(r,s,t),\infty)
 $$
 in the case $s>1$.
Thus, {\bf (NM)} follows from the monotonicity of $Q$.
\end{proof}

In the current paper, we use the non-monotonicity ({\bf NM}) 
to replace the monotonicity for the map $Q$. Throughout the whole
paper, we make the following assumption on the uniform boundedness for the map $Q$.

\begin{enumerate}
\item [{\bf (UB)}]  There exists a sequence $\{\phi_k^*\}_{k\in \mathbb{N}}$ in $Int(Y_+)$ such that $\phi_{k+1}^*>\phi_k^*$, $Y_+=\bigcup\limits_{k\in \mathbb{N}}[0,\phi_k^*]_{Y}$, and $Q[C_{\phi_k^*}]\subseteq C_{\phi_k^*}$  for all $k\in \mathbb{N}$.
\end{enumerate}

Instead of the subhomogeneity and {\bf (UC)} assumed   in Proposition~\ref{prop2.2}, we introduce the following asymptotic monotonicity, subhomogeneity,  and {\bf (UC)}  hypothesis.
\begin{enumerate}
\item [{\bf(ACH)}] There exist sequences $\{c_{l,\pm}^*\}_{l=1}^\infty$ in $\mathbb{R}$, $\{r_l^*\}_{l=1}^\infty$ in $Int(Y_+)$, $\{Q_{l}\}_{l=1}^\infty$ and $\{Q_{l,+}\}_{l=1}^\infty$ such that for any positive integer $l$, there hold
\begin{enumerate}
\item [(i)]   $c_+^*:=\lim\limits_{l\to \infty}c_{l,+}^*$,  $c_-^*:=\lim\limits_{l\to \infty}c_{l,-}^*$, and $r^*_l\leq \phi_1^*$.

\item [(ii)]   $Q_{l},Q_{l,+}:C_+ \to C_+$ are continuous and monotone maps with
$Q\geq Q_l$ in $C_{\phi_l^*}$, where $\phi_l^*$ is defined as in {\bf (UB)}.

\item [(iii)]  $Q_l$ is subhomogeneous on $[0,r_l^*]_ C$.

\item [(iv)]    $Q_{l,+}$  satisfies {\bf (UC)} with  $(c_{l,-}^*,c_{l,+}^*,r^*_l)$.

\item [(v)]   $(Q_{l},Q_{l,+},r^*_l)$  satisfies {\bf (A)}.
\end{enumerate}
\end{enumerate}

%======================================================
\section {Discrete-time  semiflows} \label{3sec}
In this section, we study the upward convergence, asymptotic annihilation,
and the existence of fixed points for discrete-time systems according to the
unilateral and bilateral limit cases, respectively.

\subsection{The unilateral limit case}
\label{sec3.1}
We first consider discrete-time systems under the unilateral limit assumptions.  We start with
the upward convergence in the unilateral {\bf (UC)}  case.
\begin{thm} \label{thm3.1}
Assume that $Q$ satisfies  {\bf (ACH)}, {\bf (NM)}, and {\bf (SP)}.
Let $c_+^*>0$ and $c\in (-c_-^*,c_+^*)\cap \mathbb{R}_+$.  Then
the following statements are valid:
\begin{itemize}
\item [{\rm (i)}]  $\lim\limits_{\alpha\rightarrow \infty}
\sup\Big\{|| Q^n[\varphi](\cdot, x)-r^*||:n\geq \alpha, \alpha+nc\leq x \leq n(c_+^*-\varepsilon),\mbox{ and } \varphi\in K\Big\}= 0$ for any $\varepsilon \in (0,c_+^*-c)$ and for any compact, uniformly bounded subset $K$ of  $C_{+}\setminus \{0\}$.

\item [{\rm (ii)}]  $\lim\limits_{\alpha\rightarrow \infty}
\sup\Big\{|| Q^n[\varphi](\cdot, x)-r^*||:n\geq \alpha, \alpha+nc\leq x \leq n(c_+^*-\varepsilon)\Big\}= 0$ for all $(\varepsilon,\varphi)\in (0,c_+^*-c)\times   C_{+}\setminus \{0\}$.

\end{itemize}
\end{thm}

\noindent
\begin{proof}  (i) Fix a  compact and uniformly bounded subset $K$ of  $C_{+}\setminus \{0\}$. For any given $(\varepsilon,\alpha)\in (0,c_+^*-c)\times (0,\infty)$, we define
$$
\mathcal{I}_{\varepsilon,\alpha}=\{(n,x)\in \mathbb{N}\times \mathbb{R}: \, n\geq \alpha,  \,  \alpha+nc\leq x \leq n(c_+^*-\varepsilon)\},
$$
$$
U_-(\varepsilon)=\liminf\limits_{\alpha\rightarrow \infty}\Big[\sup\Big\{\beta\in \mathbb{R}_+: \,
 Q^n[\varphi](\cdot,x)\geq \beta r^*,\,  \,  \forall  (n,x)\in \mathcal{I}_{\varepsilon,\alpha},   \, \varphi\in K\Big\}\Big],$$
and
$$
U_+(\varepsilon)=\limsup\limits_{\alpha\rightarrow \infty}\Big[\inf\{\beta\in \mathbb{R}_+: \,
 Q^n[\varphi](\cdot,x)\leq \beta r^*,\,  \,  \forall  (n,x) \in \mathcal{I}_{\varepsilon,\alpha}, \, \varphi\in K\}\Big].
$$
By  the definition of $U_-(\varepsilon)$ and  $U_+(\varepsilon)$, it easily follows  that $0\leq U_-(\varepsilon) \leq U_+(\varepsilon)$,  $U_-(\varepsilon)$ is non-decreasing in $\varepsilon\in (0,c_+^*-c)$, and
$U_+(\varepsilon)$ is nonincreasing in $\varepsilon\in  (0,c_+^*-c)$.
Since $\emptyset\neq\mathcal{I}_{\varepsilon_1,\alpha}\subseteq \mathcal{I}_{\varepsilon_2,\alpha}$ for all $0<\varepsilon_2<\varepsilon_1<c_+^*-c$,
without loss of generality, we assume $\varepsilon<\frac{1}{2}\min\{c_+^*-c,c+c_-^*,c_+^*+c_-^*\}$.
 By {\bf (ACH)}-(i) and  {\bf (UB)}, there exist $l_0\in \mathbb{N}$ and $\tau_{0}\in (0,1)$ such that $|c_{l_0,+}^*-c^*_+|<\frac{\varepsilon}{3}$, $|c_{l_0,-}^*-c^*_-|<\frac{\varepsilon}{3}$, $\frac{1}{\tau_0} r^{*} \geq r^*_{l_0}\geq \tau_{0}r^*$, and $K\subseteq  C_{\phi^*_{l_0}}$. Thus, $\frac{\varepsilon}{2}\in (0,\frac{1}{2}\min\{c_{l_0,+}^*,c_{l_0,-}^*+c_{l_0,+}^*\})$. Applying Proposition~\ref{prop2.2} to $Q_{l_0}$, we easily see  that for any
$\mathfrak{c}\in [\max\{0,-c_{l_0,-}^*+\frac{\varepsilon}{3}\},c_{l_0,+}^*-\frac{\varepsilon}{3}]$, there holds
 $$
 T_{-n \mathfrak{c}}\circ T_{-y} \circ Q_{l_0}^{n}\circ T_{y}[\frac{\delta r^*_{l_0}}{16}h]\geq \frac{ \delta r^*_{l_0}}{4}h,\,  \forall
n\geq n_0(\frac{\varepsilon}{2}),  y\geq y_0(\frac{\varepsilon}{2}),  \delta\in [0,1],
$$
 where $n_0(\frac{\varepsilon}{2})$  and $y_0(\frac{\varepsilon}{2})$ are defined as in Proposition \ref{prop2.2} with $Q_{l_0}$ and $\frac{\varepsilon}{2}$.

 By {\bf (SP)} and Lemma~\ref{lemm2.1-3.1-4.1}, it follows that there exist $\delta_0\in (0,1)$, $\mathfrak{y}_0>y_0(\frac{\varepsilon}{2})$ and $\mathfrak{N}_0\in \mathbb{N}$ such that
$T_{-\mathfrak{y}_0}Q^{\mathfrak{N}_0}[\varphi]\geq \delta_0 r^* h> \frac{\delta_0 \tau_0 r^*_{l_0}}{16}h$ for all $\varphi\in K$, and hence,
 $$
 Q^n[\varphi]=Q^{n-\mathfrak{N}_0}\circ  T_{\mathfrak{y}_0}[ T_{-\mathfrak{y}_0}\circ Q^{\mathfrak{N}_0}[\varphi]]\geq Q_{l_0}^{n-\mathfrak{N}_0} \circ  T_{\mathfrak{y}_0}[ \frac{\delta_0 r^*_{l_0}}{16}h],\,
 \,  \forall  n\geq \mathfrak{N}_0, \, \varphi\in K.
 $$
  Thus,  for any $\mathfrak{c}\in [\max\{0,-c_{l_0,-}^*+\frac{\varepsilon}{3}\},c_{l_0,+}^*-\frac{\varepsilon}{3}], n\geq n^{**}:=n_0(\frac{\varepsilon}{2})+ \mathfrak{N}_0$ and  $\varphi\in K$,  we have
 $$
  T_{-(n-\mathfrak{N}_0)\mathfrak{c}}\circ T_{-\mathfrak{y}_0} \circ Q^n[\varphi]\geq  T_{-(n-\mathfrak{N}_0)\mathfrak{c}}\circ T_{-\mathfrak{y}_0} \circ Q_{l_0}^{n-\mathfrak{N}_0} \circ  T_{\mathfrak{y}_0}[ \frac{\delta_0 r^*_{l_0}}{16}h]
 \geq \frac{ \delta_0 r^*_{l_0}}{4}h>\frac{\delta_0 \tau_{0} r^*}{4}h.
 $$
 In other words, $Q^n[\varphi](\cdot,x)\geq \frac{ \delta_0 \tau_{0} r^*}{4}$ for all $(n,x,\varphi)\in \mathcal{J}\times K$, where
 $$
\mathcal{J}=\left\{(n,x)\in \mathbb{N}\times \mathbb{R}: n\geq  n^{**}\mbox{ and } \max\{0,-c_{l_0,-}^*+\frac{\varepsilon}{3}\}\leq \frac{x-\mathfrak{y}_0}{n-\mathfrak{N}_0}\leq c_{l_0,+}^*-\frac{\varepsilon}{3}\right\}.
$$
By the choices of $l_0$ and $\varepsilon$, we can verify that for any   $ \alpha\geq n^{**}+\mathfrak{y}_0+\frac{3 \mathfrak{N}_0c_{+}^*}{\varepsilon}+\mathfrak{N}_0|c_-^*|$, there holds
$\mathcal{I}_{\varepsilon,\alpha}\subseteq \mathcal{J}$,  and hence
$$
Q^n[\varphi](\cdot,x)\geq \frac{ \delta_0 \tau_{0} r^*}{4},
\,  \forall  (n,x,\varphi)\in \mathcal{I}_{\varepsilon,\alpha}\times K.
 $$
This, together with the definitions of $U_\pm(\varepsilon)$ and the arbitrariness of $\varepsilon$, implies that
$$
0< \frac{ \delta_0 \tau_{0} }{4}\leq U_-(\varepsilon) \leq U_+(\varepsilon)< U^*:=1+\inf\{\beta\in \mathbb{R}_+:
 \phi_{l_0}^*\leq \beta r^*\}<\infty,$$
where $\varepsilon\in (0,\frac{1}{2}\min\{c_+^*-c,c_-^*+c,c_+^*+c_-^*\}).
$

%==========?????????
%{$\mathcal{I}_{\varepsilon,\alpha}=\{(n,x)\in \mathbb{N}\times \mathbb{R}:n\geq \alpha \mbox{ and } \alpha+nc\leq x \leq n(c_+^*-\varepsilon)\},$$\alpha+nc\leq x \leq n(c_+^*-\varepsilon)$ iff $\frac{\alpha+nc-y^*}{n-n^*}\leq \frac{x-y^*}{n-n^*} \leq \frac{n(c_+^*-\varepsilon)-y^*}{n-n^*}$;$\max\{0,-c_{l_0,-}^*+\frac{\varepsilon}{3}\}\leq \frac{\alpha+nc-y^*}{n-n^*}$ and $  \frac{n(c_+^*-\varepsilon)-y^*}{n-n^*}\leq c_{l_0,+}^*-\frac{\varepsilon}{3}$;$\max\{0,-c_{-}^*+\frac{2\varepsilon}{3}\}\leq \frac{\alpha+nc-y^*}{n-n^*}$ and $  \frac{n(c_+^*-\varepsilon)-y^*}{n-n^*}\leq c_{+}^*-\frac{2\varepsilon}{3}$;$-c_{-}^*+\frac{2\varepsilon}{3}\leq \frac{\alpha+nc-y^*}{n-n^*}$ and $  \frac{n(c_+^*-\varepsilon)-y^*}{n-n^*}\leq c_{+}^*-\frac{2\varepsilon}{3}$;$[-c_{-}^*+\frac{2\varepsilon}{3}](n-n^*)\leq \alpha+nc-y^*$ and $ n(c_+^*-\varepsilon)-y^*\leq [c_{+}^*-\frac{2\varepsilon}{3}](n-n^*)$;$-n c_{-}^*+\frac{2n\varepsilon}{3}+n^*c_{-}^*-\frac{2 n^*\varepsilon}{3}\leq \alpha+nc-y^*$ and $ n(c_+^*-\varepsilon)-y^*\leq nc_{+}^*-\frac{2n\varepsilon}{3}-n^*c_{+}^*+\frac{2n^*\varepsilon}{3}$;$-n c_{-}^*+\frac{2n\varepsilon}{3}+n^*c_{-}^*-\frac{2 n^*\varepsilon}{3}\leq \alpha+nc-y^*$ and $ -n\varepsilon-y^*\leq -\frac{2n\varepsilon}{3}-n^*c_{+}^*+\frac{2n^*\varepsilon}{3}$;$-n c_{-}^*+\frac{2n\varepsilon}{3}+n^*c_{-}^*-\frac{2 n^*\varepsilon}{3}\leq \alpha+nc-y^*$ and $ -y^*\leq \frac{n\varepsilon}{3}-n^*c_{+}^*+\frac{2n^*\varepsilon}{3}$ %?????šŠ?alpha ??šš??šššº1I¡ê?J¡ãš¹o?1??ŠÌ}
%==========?????????
To finish the proof, it suffices to prove $U_\pm(\varepsilon)=1$ for all $\varepsilon\in (0,\frac{1}{2}\min\{c_+^*-c,c_-^*+c,c_+^*+c_-^*\})$. Otherwise, there exists $\varepsilon_0\in (0,\frac{1}{2}\min\{c_+^*-c,c_-^*+c,c_+^*+c_-^*\})$ such that $\{U_-(\varepsilon_0),U_+(\varepsilon_0)\}\neq \{1\}$. Thus, $\{U_-(\varepsilon),U_+(\varepsilon)\}\neq \{1\}$ for all $\varepsilon\in (0,\varepsilon_0]$.
 Due to the
monotonicity of $U_\pm$, we
may assume,  without loss of generality, that for some
$\varepsilon_1\in (0,\varepsilon_0)$, $U_\pm$ is continuous at
$\varepsilon_1$. Note that  $\{U_-(\varepsilon_1),U_+(\varepsilon_1)\}\neq \{1\}$ due to $\varepsilon_1\in (0,\varepsilon_0)$.

 In view of  {\bf (NM)},   there exist  $\gamma_0\in (0,U_-(\varepsilon_1))$, $d_0,z_0\in (0,\infty)$, and $N_0\in \mathbb{N}$ such that  one
 of the following properties holds true:
 \begin{itemize}
 \item [(a)]
 $T_{-z}\circ Q^{N_0} \circ T_{z} [\varphi](\cdot,0)\geq (U_-(\varepsilon_1)+\gamma_0)r^{*}, \, \forall (z,\varphi)\in [z_0,\infty)\times [\psi_1, \psi_2]_C$,
  \item [(b)] $T_{-z}\circ Q^{N_0} \circ T_{z}[\varphi](\cdot,0) \leq (U_+(\varepsilon_1)-\gamma_0)r^{*}, \, \forall (z,\varphi)\in [z_0,\infty)\times [\psi_1, \psi_2]_C$,
 \end{itemize}
 where $\psi_1=(U_-(\varepsilon_1)-\gamma_0){\xi}_{d_0} r^{*}$
 and $\psi_2=(U_+(\varepsilon_1)+\gamma_0)\tilde{\xi}_{d_0,\frac{U^*+\gamma_0}{U_-(\varepsilon_1)-\gamma_0}} r^{*}$.
 In the following, we only give the details to get a contradiction for
the case (a) since the case (b) can be addressed in a similar way.

According to the definition of $U_-(\tau)$, for any given $\tau\in
(\varepsilon_1,\varepsilon_0)$, there exist  sequences $\{\alpha_k\}_{k\in \mathbb{N}}$ in $(0,\infty)$ and
$\{(\theta_k,n_k,x_k,\varphi_k)\}_{k\in \mathbb{N}}$ in $M\times \mathcal{I}_{\tau,\alpha_k}\times K$ such that $\lim\limits_{k\rightarrow \infty}r^*(\theta_k)\in Int(\mathbb{R}_+^N)$, $\lim\limits_{k\rightarrow \infty}\alpha_k=\lim\limits_{k\rightarrow \infty}n_k=\infty$, and
$\lim\limits_{k\rightarrow \infty}Q^{n_k}[\varphi_k](\theta_k,x_k)-U_-(\tau)r^*(\theta_k)\in \partial(\mathbb{R}^N_+)$. Since
 for any bounded subset $\mathcal{B}$ of $\mathbb{R}$,
$\{n_k-N_0\}\times (x_k+\mathcal{B})\subseteq  \mathcal{I}_{\varepsilon_1,\alpha_k-N_0-\sup\limits_{x\in\mathcal{B}}|x|}$ for all
large $k$ and  bounded subset $\mathcal{B}$ of $\mathbb{R}$, we obtain
%3.on  $\mathcal{B}$: $x_k+b\leq n_k(c_+^*-\tau)+b\leq??? (n_k-N_0)(c_+^*-\varepsilon_1)$; : $x_k+b\geq \alpha_k+n_kc+b\geq???\alpha_k-N_0-\sup\limits_{x\in\mathcal{B}}|x|+(n_k-N_0)c $$$\mathcal{I}_{\varepsilon,\alpha}=\{(n,x)\in \mathbb{N}\times \mathbb{R}:n\geq \alpha \mbox{ and } \alpha+nc\leq x \leq n(c_+^*-\varepsilon)\},$$
%?????aš¢??šŠ??¡ã??¡ãš¹o?1??ŠÌ¡ê?
$$
\liminf\limits_{k\rightarrow \infty}\Big[
\sup\{\beta \in \mathbb{R}_+:T_{-x_k}\circ Q^{n_k-N_0}[\varphi_k](\cdot,x)\geq \beta r^* \mbox{ for all } x\in \mathcal{B}\}\Big]\in [U_-(\varepsilon_1),U_+(\varepsilon_1)]
$$
 and
$$
\limsup\limits_{k\rightarrow \infty}
\Big[\inf\{\beta\in \mathbb{R}_+:T_{-x_k}\circ Q^{n_k-N_0}[\varphi_k](\cdot,x)\leq \beta r^* \mbox{ for all }  x\in \mathcal{B}\}\Big] \in
[U_-(\varepsilon_1),U_+(\varepsilon_1)].
$$
 It then follows  that there exists $k_0\in \mathbb{N}$ such that $x_{k}\geq z_0$ and
 $$
 T_{-x_{k}}\circ Q^{n_{k}-N_0}[\varphi_k]\in [(U_-(\varepsilon_1)-\gamma_0)\xi_{d_0} r^{*},(U_+(\varepsilon_1)+\gamma_0)\tilde{\xi}_{d_0,\frac{U^*+\gamma_0}{U_-(\varepsilon_1)-\gamma_0}} r^{*}]_C,
 \,  \forall  k\geq k_0.
 $$

 In the case (a), we easily check that
 $$
 T_{-x_k}\circ Q^{n_k}[\varphi_k](\cdot,0)=T_{-x_k}\circ Q^{N_0} \circ T_{x_k} [T_{-x_k}\circ Q^{n_k-N_0}[\varphi_k]](\cdot,0),
 $$
 and hence,
 $$T_{-x_k}\circ Q^{n_k}[\varphi_k](\cdot,0)\geq (U_-(\varepsilon_1)+\gamma_0)r^*,  \,  \forall  k\geq k_0.
 $$
This implies that
$$
\liminf\limits_{k\rightarrow \infty}\Big[Q^{n_k}[\varphi_k](\theta_k,x_k)-(\gamma_0+U_-(\varepsilon_1))r^*(\theta_k)\Big]\in \mathbb{R}^N_+.
$$
By the choices of $\theta_k,n_k,x_k$, and $r^*(\cdot)$,  there holds
$$
U_-(\tau)\geq U_-(\varepsilon_1)+\gamma_0> U_-(\varepsilon_1).
$$
Letting
$\tau\rightarrow \varepsilon_1$, we  obtain $U_-(\varepsilon_1)\geq
U_-(\varepsilon_1)+\gamma_0>U_-(\varepsilon_1)$, a
contradiction.

Consequently, $U_-(\varepsilon)=U_+(\varepsilon)=1$ for all $\varepsilon\in
(0,\frac{1}{2}\min\{c_+^*-c,c_-^*+c,c_+^*+c_-^*\})$. This shows that  statement (i) holds true.

(ii)  follows from (i) by taking $K=\{\varphi\}$.
\end{proof}

\begin{rem} 
From the proof of Theorem \ref{thm3.1}, it is easy to see that
if each $Q_l$ in {\bf (ACH)} satisfies {\bf (SP)}, then the set $K$ can be replaced by an order interval $[u_0,v_0]_C$ with $u_0>0$.
\end{rem} 

\begin{thm} \label{thm3.1-0}
Assume that all the assumptions of Theorem~\ref{thm3.1} hold and
$c_+^*>0$.  Then for any $\varepsilon\in (0,\frac{1}{2}\min\{c_+^*,c_+^*+c_-^*\})$, the following statements are valid:
\begin{itemize}
\item [{\rm (i)}]  $\lim\limits_{n\rightarrow \infty}
\max\left\{||Q^n[\varphi](\cdot, x)-r^*||:n\max\{\varepsilon,-c_-^*+\varepsilon\}\leq x \leq n(c_+^*-\varepsilon)\mbox{ and } \varphi\in K\right\}= 0$ for any
 compact{,  uniformly bounded} subset $K$ of  $C_{+}\setminus \{0\}$.

\item [{\rm (ii)}]  $\lim\limits_{n\rightarrow \infty}
\max\left\{||Q^n[\varphi](\cdot, x)-r^*||:n\max\{\varepsilon,-c_-^*+\varepsilon\}\leq x \leq n(c_+^*-\varepsilon)\right\}= 0$ for all $\varphi\in   C_{+}\setminus \{0\}$.

\end{itemize}
\end{thm}

\noindent
\begin{proof}  Let $J_{\varepsilon,n}=\{n\}\times [n\max\{\varepsilon,-c_-^*+\varepsilon\}, n(c_+^*-\varepsilon)]$ and
$$
\mathcal{J}_{\varepsilon,\alpha,c}=\{(n,x)\in \mathbb{N}\times \mathbb{R}:  \,  n\geq \alpha,  \, \alpha+nc\leq x \leq n(c_+^*-\varepsilon)\},
\quad  \forall  (n,\varepsilon,\alpha,c)\in \mathbb{N}\times \mathbb{R}_+^3.
$$
Since $J_{\varepsilon,n}\subseteq \mathcal{J}_{\frac{\varepsilon}{3},\frac{n\min\{1,\varepsilon\}}{3},\frac{\varepsilon}{3}-\min\{0,c_-^*\}}$ for all $\varepsilon\in (0,\frac{1}{2}\min\{c_+^*,c_+^*+c_-^*\})$ and $n \in \mathbb{N}$,
Theorem~\ref{thm3.1} yield all the desired  conclusions.
\end{proof}

For any given $(\varphi,\mu)\in L^{\infty}(M\times \mathbb{R},\mathbb{R}^{N})\times \mathbb{R}$, we define a function
$$
e_{\varphi,\mu}(\theta,x)=\varphi(\theta,x) e^{-\mu x}, \,  \forall (\theta,x)\in  M\times \mathbb{R}.
$$
Let   $ L,L_+$ and $L_-$ be monotone  linear operators on the space of functions
$$
\tilde{C}:=\left\{\sum\limits_{k=1}^l e_{\varphi_k,\mu_k}:\, l\geq 1,
(\varphi_k,\mu_k)\in L^{\infty}(M\times \mathbb{R},\mathbb{R}^{N})
\times \mathbb{R}, \,  \forall 1\leq k\leq l\right\}
$$
equipped with the pointwise ordering. Here and after, let us denote the set of all bounded maps from $M\times \mathbb{R}$ into $\mathbb{R}^{N}$ by $L^{\infty}(M\times \mathbb{R},\mathbb{R}^{N})$ with the pointwise ordering induced by $L^{\infty}(M\times \mathbb{R},\mathbb{R}^{N}_+)$.
For any  $\mu\in \mathbb{R}$ with $L_+[e_{\check{1},\mu}](\cdot,0)\in Y_+$, we  define $L_{+,\mu}:Y\to Y$ by
$$
L_{+,\mu}[\phi](\theta)=L_+[e_{\phi,\mu}](\theta,0),  \,   \forall \theta\in  M,\,
\phi\in  Y.
$$
Assume that $I_{L_+}:=\{\mu\in (0,\infty):L_+[e_{\check{1},\mu}](\cdot,0)\in Y_+\}\neq \emptyset$,   $T_{-z}\circ L_+ \circ T_z[e_{\zeta,\mu}]=L_+[e_{\zeta,\mu}] $ for all $(z,\mu,\zeta)\in \mathbb{R}\times I_{L_+}\times Y_+$, and   there exists
$l_0\in \mathbb{N}$ such that
$(L_{+,\mu})^{l_0}$ is a compact and strongly positive operator on $Y$ for all  $\mu\in I_{L_+}$. It follows that
$L_{+,\mu}$ has the principal eigenvalue $\lambda(\mu)$ and  a unique  strongly
positive eigenfunction $\zeta_\mu$ associated with  $\lambda(\mu)$ such that
$\inf\{\zeta_\mu^i(\theta):\theta\in M, 1\leq i\leq N
\}=1$.  Let
$${c^*_{L_+}}:=\inf\limits_{\mu\in I_{L_+}}\frac{1}{\mu}\log\lambda(\mu)$$
and define
$$
\underline{L}[\varphi;\alpha,\mu](\theta,x)=\min\{\alpha \zeta_\mu(\theta),L[\varphi](\theta,x)\},
 \,  \forall (\varphi,\theta,x)\in \tilde{C}\times M\times \mathbb{R},  \,   (\alpha,\mu)\in \mathbb{R}_+^2.
$$

To obtain the asymptotic annihilation properties, we make  the following  assumption.

%=========================+++++++++=============
%=========================+++++++++=============
%=========================+++++++++=============
\begin{description}
\item [{\bf (LC)}]  For any $\epsilon>0$ and $\mu\in I_{L_{+}}$, there exists $ \mathfrak{x}_{\epsilon,\mu}>0$ such that
$$
L[e_{\zeta_\mu,\mu}](\theta,x)\leq (1+\mu \epsilon) L_+[e_{\zeta_\mu,\mu}](\theta,x),  \,  \forall  \theta\in M,  \,  x\geq \mathfrak{x}_{\epsilon,\mu}.
$$
\end{description}

\begin{lemma} \label{lem3.2-000000} Let {\bf (LC)} hold. Then for any $(\epsilon,\alpha)\in (0,\infty)^2$, there exist $\mu:=\mu_\epsilon\in I_{L_+}$ and $\beta_{\epsilon,\alpha}>\alpha$
such that
$$
\underline{L}^n[v_{\alpha,\beta,\mu};\alpha,\mu](\theta,x)\leq  \beta e^{\mu(-x+n\max\{0,c^*_{L_+}+\epsilon\})}\zeta_\mu(\theta), \,  \forall (\theta,x,\beta,n)\in  M \times\mathbb{R}\times [\beta_{\epsilon,\alpha},\infty)\times \mathbb{N},
$$
where
$v_{\alpha,\beta,\mu}(\theta,x)=\min\{\alpha,\beta  e^{-\mu x}\}\zeta_\mu(\theta)$.
\end{lemma}

\begin{proof} Fix $(\epsilon,\alpha)\in  Int(\mathbb{R}_+^2)$.  If $c^*_{L_+}<-\infty$, then by the definition of ${c^*_{L_+}}$, there exists $\mu:=\mu_\epsilon\in I_{L_+}$ such that $\lambda(\mu)<e^{\mu(c^*_{L_+}+\frac{\epsilon}{3})}$.
It follows from {\bf (LC)} that   for any $\theta\in M$ and  $x\geq \mathfrak{x}_\epsilon:=\mathfrak{x}_{\frac{\epsilon}{3},\mu_\epsilon}$, there holds
\begin{eqnarray*}
L[e_{\zeta_\mu,\mu}](\theta,x)&\leq& (1+\frac{\mu \epsilon}{3}) L_+[e_{\zeta_\mu,\mu}](\theta,x)
\\
&=&(1+\frac{\mu \epsilon}{3}) \lambda(\mu)\zeta_\mu(\theta) e^{-\mu x}
\\
&\leq& \zeta_\mu(\theta) e^{-\mu x+\mu(c^*_{L_+}+\epsilon)}.
\end{eqnarray*}
Taking  $\beta_{\epsilon,\alpha}=\alpha e^{\mu (|c^*_{L_+}+\epsilon|+\mathfrak{x}_{\epsilon})}$, we  have
$$\alpha \zeta_\mu(\theta)\leq \beta \zeta_\mu(\theta) e^{-\mu x+\mu(c^*_{L_+}+\epsilon)}, \, \forall (\theta,x,\beta)\in M\times (-\infty,\mathfrak{x}_{\epsilon}]\times [\beta_{\epsilon,\alpha},\infty),$$
$$\underline{L}[v_{\alpha,\beta,\mu};\alpha,\mu](\theta,x)\leq \beta \zeta_\mu(\theta) e^{-\mu x+\mu(c^*_{L_+}+\epsilon)}, \, \forall (\theta,x,\beta)\in M\times [\mathfrak{x}_{\epsilon},\infty)\times [\beta_{\epsilon,\alpha},\infty),$$ and
$$\underline{L}[v_{\alpha,\beta,\mu};\alpha,\mu](\theta,x)\leq \alpha \zeta_\mu(\theta),\, \forall (\theta,x,\beta)\in M\times \mathbb{R}\times [\beta_{\epsilon,\alpha},\infty). $$
It then follows that
\begin{eqnarray*}
\underline{L}[v_{\alpha,\beta,\mu};\alpha,\mu](\theta,x)&\leq& \min\{\alpha \zeta_\mu(\theta),\beta \zeta_\mu(\theta) e^{-\mu x+\mu(c^*_{L_+}+\epsilon)}\}
\\
&\leq& e^{\mu\max\{0,c^*_{L_+}+\epsilon\}}v_{\alpha,\beta,\mu}(\theta,x)
\end{eqnarray*}
 for all $(\theta,x,\beta)\in M\times \mathbb{R}\times [\beta_{\epsilon,\alpha},\infty)$,
and hence,
$$
\underline{L}^n[v_{\alpha,\beta,\mu};\alpha,\mu](\theta,x)\leq  \beta e^{\mu(-x+n\max\{0,c^*_{L_+}+\epsilon\})}\zeta_\mu(\theta),
\,  \forall
(\theta,x,\beta,n)\in M\times \mathbb{R}\times [\beta_{\epsilon,\alpha},\infty)\times \mathbb{N}.
$$

Next, if $c^*_{L_+}=-\infty$, then by  {\bf (LC)}, there exist $\mu:=\mu_\epsilon\in I_{L_+}$ and $\mathfrak{x}_\epsilon:=\mathfrak{x}_{\frac{\epsilon}{3},\mu_\epsilon}$ such that $\lambda(\mu)<\frac{1}{1+\frac{\mu \epsilon}{3}}$ and hence, $L[e_{\zeta_\mu,\mu}](\theta,x)\leq \zeta_\mu(\theta) e^{-\mu x}$ for any $\theta\in M$ and  $x\geq \mathfrak{x}_\epsilon$.
By the definitions of $\underline{L}, v_{\alpha,\beta,\mu}$ , we easily check that
$\underline{L}[v_{\alpha,\beta,\mu};\alpha,\mu]\leq v_{\alpha,\beta,\mu}$
 for all $\beta\geq \beta_{\epsilon,\alpha}:=\alpha e^{\mu \mathfrak{x}_{\epsilon}}$. Thus, we have
$$
\underline{L}^n[v_{\alpha,\beta,\mu};\alpha,\mu](\theta,x)\leq  \beta e^{-\mu x}\zeta_\mu(\theta),
\,  \forall
(\theta,x,\beta,n)\in M\times \mathbb{R}\times [\beta_{\epsilon,\alpha},\infty)\times \mathbb{N}.
$$
This completes the proof.
\end{proof}

In the following, we write  $\lim\limits_{l\to \infty}\psi_l= \psi$ in $L^\infty_{loc}(M\times \mathbb{R},\mathbb{R}^{N})$  provided that
$$\lim\limits_{l\to \infty} \sup\left\{||\psi_l(\theta,x)-\psi(\theta,x)||:(\theta,x)\in M\times [-k,k]\right\}=0, \quad \forall k\in \mathbb{N}.$$
In order to verify {\bf (LC)} in applications,  we introduce the
following assumption on  $L$ and $L_\pm$.
\begin{enumerate}
\item [{\bf (SLC)}]  For any real interval $I$ and  $ (\mu,\nu,\zeta)\in (0,\infty)^2\times Y_+$ with $\zeta\geq \check{1}:=(1,1,\cdots,1)^{T}\in \mathbb{R}^N$ and $\mathfrak{J}_{\mu}^\pm:=\{\theta\in M:L_\pm[e_{\check{1},\mu}](\theta,0)\in \mathbb{R}^N\}\neq \emptyset$,  there hold
\begin{enumerate}
\item [(i)]
the map $M\times \mathbb{R}\ni (\theta,z)\mapsto L_\pm[e_{\zeta\cdot {\bf 1}_{[z,\infty)},\mu}](\theta,0)\in \mathbb{R}^N$ is bounded and continuous on $M\times \mathbb{R}$ with $ L_+[e_{\check{1},\mu}](\cdot,0)\in Int(Y_+)$, $L_\pm[e_{\zeta,\mu}](\cdot,0)\in Y_+$, $\lim\limits_{z\to -\infty}L_\pm[e_{\zeta\cdot {\bf 1}_{[z,\infty)},\mu}](\theta,0)=L_\pm[e_{\zeta,\mu}](\theta,0)$ and $\lim\limits_{z\to \infty}L_\pm[e_{\zeta\cdot {\bf 1}_{[z,\infty)},\mu}](\theta,0)=0$ for all $\theta\in M$.

\item [(ii)]     $\lim\limits_{z\to \pm\infty}T_{-z}\circ L\circ T_{z}[e_{\zeta\cdot {\bf 1}_{I},\mu}]= L_\pm[e_{\zeta\cdot {\bf 1}_{I},\mu}]$  in $L^\infty_{loc}(M\times \mathbb{R},\mathbb{R}^{N})$,  where
$e_{\zeta\cdot {\bf 1}_{I},\mu}(\theta,x)=\zeta(\theta)e^{-\mu x}{\bf 1}_{I}(x)$ for all  $(\theta,x)\in M\times \mathbb{R}$.

\item [(iii)]   $T_{-z}\circ L\circ T_{z}[e_{\zeta\cdot {\bf 1}_{I},\mu}]$ is nonincreasing in $z\in \mathbb{R}$.

\item [(iv)]   There exists $\kappa=\kappa_\nu>1$ such that
$\frac{1}{\kappa} L_+[e_{\check{1},\nu}]\leq L_-[e_{\check{1},\nu}] \leq \kappa L_+[e_{\check{1},\nu}]$.

\end{enumerate}
\end{enumerate}

\begin{lemma} \label{lem3.2-000} Let {\bf (SLC)} hold. Then $\mathfrak{I}_\mu^+=\mathfrak{I}_\mu^-$ and $\mathfrak{I}_\mu:=\mathfrak{I}_\mu^\pm \in \{\emptyset, M\}$ for all $\mu\in (0,\infty)$.
Moreover, if $(\mu,\zeta)\in (0,\infty)\times Y_+$ with $\mathfrak{I}_\mu=M$ and $\zeta\geq \check{1}$,  then  the following statements hold:
\begin {itemize}
\item [{\rm (i)}]   $T_{-z}\circ L_\pm \circ T_z[e_{\zeta\cdot {\bf 1}_{I},\mu}]=L_\pm[e_{\zeta\cdot {\bf 1}_{I},\mu}] $ for all $z\in \mathbb{R}$ and any interval $I$ in $\mathbb{R}$. Hence,  for any $(\theta,x,z)\in M\times \mathbb{R}^2$, there hold $$e^{\mu x}L_\pm[ e_{\zeta\cdot {\bf 1}_{[x-z,\infty)},\mu}](\theta,x)=L_\pm[ e_{\zeta\cdot {\bf 1}_{[-z,\infty)},\mu}](\theta,0)$$ and $$e^{\mu x}L_\pm[ e_{\zeta\cdot {\bf 1}_{(-\infty,x-z]},\mu}](\theta,x)=L_\pm[ e_{\zeta\cdot {\bf 1}_{(-\infty,-z]},\mu}](\theta,0).$$

\item [{\rm (ii)}]   $ L_-[e_{\zeta\cdot {\bf 1}_{I},\mu}]\geq L[e_{\zeta\cdot {\bf 1}_{I},\mu}]\geq L_+[e_{\zeta\cdot {\bf 1}_{I},\mu}]$ for any interval $I$ in $\mathbb{R}$.

\item [{\rm (iii)}] For any $\gamma\in (0,\infty)$, there exists $\varrho_{\gamma,\mu,\zeta}>0$ such that
$$
L[e_{\zeta\cdot {\bf 1}_{(-\infty,x-\varrho]},\mu}](\theta,x)\leq\gamma
L[e_{\zeta\cdot {\bf 1}_{[x-\varrho,\infty)},\mu}](\theta,x), \, \forall (\theta,x,\varrho)\in M\times  \mathbb{R} \times [\varrho_{\gamma,\mu,\zeta},\infty).
$$
\item [{\rm (iv)}] For any $(\delta,\varrho)\in (0,\infty)^2$, there exist $\varrho_0:=\varrho_0(\delta,\varrho,\mu,\zeta)>\varrho$ and $\rho_{\delta,\varrho,\mu,\zeta}>0$ such that
$$
L[e_{\zeta\cdot {\bf 1}_{[x-\varrho_0,\infty)},\mu}](\theta,x)\leq (1+\delta)L_+[e_{\zeta\cdot {\bf 1}_{[x-\varrho_0,\infty)},\mu}](\theta,x),
\, \forall \theta\in M, x\geq \rho_{\delta,\varrho,\mu,\zeta}.
$$

\item [{\rm (v)}] For any $\epsilon>0$, there exists  $ \mathfrak{x}_{\epsilon,\mu,\zeta}>0$ such that $$
L[e_{\zeta,\mu}](\cdot,x)\leq  (1+\mu \epsilon) L_+[e_{\zeta,\mu}](\cdot,x),
\, \forall x\in  [ \mathfrak{x}_{\epsilon,\mu,\zeta},\infty).
$$
In particular, the assumption {\bf (LC)} holds.
\end {itemize}
\end{lemma}

\begin{proof}  Clearly, by  the former statement of {\bf (SLC)}-(iv), we may check  that $\mathfrak{I}_\mu^+=\mathfrak{I}_\mu^-$ for all $\mu\in (0,\infty)$. Then $\mathfrak{I}_\mu^+=\mathfrak{I}_\mu^-=\emptyset$ or $\mathfrak{I}_\mu^+=\mathfrak{I}_\mu^-\neq\emptyset$.
If  $\mathfrak{I}_\mu^+=\mathfrak{I}_\mu^-\neq\emptyset$, then {\bf (SLC)}-(i) implies that $\mathfrak{I}_\mu^+=\mathfrak{I}_\mu^-=M$.
Now, let us fix $(\mu,\zeta)\in (0,\infty)^2\times Y_+$ with $\mathfrak{I}_\mu=M$ and $\zeta\geq \check{1}$.

(i) follows from  {\bf (SLC)}-(ii) and  the definition of $e_{\zeta\cdot {\bf 1}_{I},\mu}$.

 (ii) follows from {\bf (SLC)}-(ii,iii).

(iii) Fix $\gamma\in (0,\infty)$ and define $F:M\times\mathbb{R}\to \mathbb{R}^N$ by $$F(\theta,\varrho)=\gamma
L_+[e_{\zeta\cdot {\bf 1}_{[-\varrho,\infty)},\mu}](\theta,0)-L_-[e_{\zeta\cdot {\bf 1}_{(-\infty,-\varrho]},\mu}](\theta,0), \forall (\theta,\varrho)\in M\times\mathbb{R}.$$ In view of {\bf (SLC)}-(i,iii) and the definition of $F$, we easily see that $F\in C(M\times\mathbb{R},\mathbb{R}^N)$, $\lim\limits_{\varrho\to \infty} F(\theta, \varrho)\in Int(\mathbb{R}^N_+)$, and $ F(\theta, \varrho)$ is nondecreasing  in $\varrho$.  These, together with the compactness of $M$, imply that there exists $\varrho_{\gamma,\mu,\zeta}>0$ such that $F(\theta,\varrho)\in \mathbb{R}^N_+$ for all $\varrho\geq\varrho_{\gamma,\mu,\zeta}>0$ and $\theta\in M$.  Hence, by  (i), we obtain that $$L_-[e_{\zeta\cdot {\bf 1}_{(-\infty,x-\varrho]},\mu}](\theta,x)\leq\gamma
L_+[e_{\zeta\cdot {\bf 1}_{[x-\varrho,\infty)},\mu}](\theta,x),\, \forall
(\theta,x,\varrho)\in M\times  \mathbb{R} \times [\varrho_{\gamma,\mu,\zeta},\infty).
$$
Thus, (iii) follows from (ii).

(iv) Suppose that  $(\delta,\varrho)\in  Int(\mathbb{R}_+^2)$. In view of {\bf (SLC)}-(i), we easily see that  for any given $\theta_0\in M$, there exist $\varrho_{\theta_0}:=\varrho_{\theta_0,\delta,\varrho,\mu,\zeta}>\varrho>-\varrho>\hat{\varrho}_{\theta_0,\delta,\varrho,\mu,\zeta}=:\hat{\varrho}_{\theta_0}$ such that
$$L_-[e_{\zeta\cdot {\bf 1}_{[-\hat{\varrho}_{\theta_0},\infty)},\mu}](\theta_0,0)\ll \frac{\delta}{3}L_+[e_{\zeta\cdot {\bf 1}_{[-\varrho_{\theta_0},\infty)},\mu}](\theta_0,0).
$$
This, together with  {\bf (SLC)}-(i,iii) and the compactness of $M$, implies that there exist $\sigma_0\in(0,1)$ and $\varrho_0:=\varrho_0(\delta,\varrho,\mu,\zeta)>\varrho>-\varrho>\hat{\varrho}_0:=\hat{\varrho}_0(\delta,\varrho,\mu,\zeta)$ such that   $$
\frac{\delta}{3}L_+[e_{\zeta\cdot {\bf 1}_{[-\varrho_0,\infty)},\mu}](\theta,0)-L_-[e_{\zeta\cdot {\bf 1}_{[-\hat{\varrho}_0,\infty)},\mu}](\theta,0)\geq \sigma_0 \check{1}, \,   \,
 \, L_+[e_{\zeta \cdot {\bf 1}_{[-\varrho_0,\varrho_0],\mu}}](\theta,0)\ge \sigma_0  \check{1}, \,	 \forall \theta\in M.
 $$
Hence, by (i) and (ii), we have $$L[e_{\zeta\cdot {\bf 1}_{[x-\hat{\varrho}_0,\infty)},\mu}](\theta,x)\leq L_-[e_{\zeta\cdot {\bf 1}_{[x-\hat{\varrho}_0,\infty)},\mu}](\theta,x)\ll \frac{\delta}{3}L_+[e_{\zeta\cdot {\bf 1}_{[x-\varrho_0,\infty)},\mu}](\theta,x), \forall (\theta,x)\in M\times \mathbb{R}.$$

By {\bf (SLC)}-(ii), we easily check that $$\lim\limits_{x\to\infty}||T_{-x}\circ L \circ T_x[e_{\zeta\cdot {\bf 1}_{[-\varrho_0,-\hat{\varrho}_0]},\mu}](\cdot,0)- L_+[e_{\zeta\cdot {\bf 1}_{[-\varrho_0,-\hat{\varrho}_0]},\mu}](\cdot,0)||=0.$$ Thus, there exists  $\rho_{\delta,\varrho,\mu,\zeta}>0$
such that $$T_{-x}\circ L \circ T_x[e_{\zeta\cdot {\bf 1}_{[-\varrho_0,-\hat{\varrho_0})},\mu}](\cdot,0) \ll (1+\frac{\delta}{3}) L_+[e_{\zeta\cdot {\bf 1}_{[-\varrho_0,\infty)},\mu}](\cdot,0), \forall x\geq \rho_{\delta,\varrho,\mu,\zeta}.$$
It then follows from statement (i) that
$$
L[e_{\zeta\cdot {\bf 1}_{[x-\varrho_0,x-\hat{\varrho_0})},\mu}](\cdot,x) \ll (1+\frac{\delta}{3}) L_+[e_{\zeta\cdot {\bf 1}_{[x-\varrho_0,\infty)},\mu}](\cdot,x),
\, \forall x\geq \rho_{\delta,\varrho,\mu,\zeta}.
$$
This shows that $L[e_{\zeta\cdot {\bf 1}_{[x-\varrho_0,\infty)},\mu}](\cdot,x)\leq (1+\delta)L_+[e_{\zeta\cdot {\bf 1}_{[x-\varrho_0,\infty)},\mu}](\cdot,x)$ for all  $x\geq \rho_{\delta,\varrho,\mu,\zeta}$.

(v) Fix $\epsilon\in (0,\infty)$.  Let $\gamma=\delta:=-1+\sqrt{1+\mu \epsilon}$, $\varrho_1=\varrho_0({\delta,\varrho_{\gamma,\mu,\zeta},\mu,\zeta})$, and
$\mathfrak{x}_{\epsilon,\mu,\zeta}:=\rho_{\delta,\varrho_{\gamma,\mu,\zeta},\mu,\zeta}>0$. By applying (iii) and (iv), we obtain
\begin{eqnarray*}
L[e_{\zeta,\mu}](\cdot,x)&=&L[e_{\zeta\cdot{\bf 1}_{[x-\varrho_1,\infty)},\mu}](\cdot,x)+L[e_{\zeta\cdot{\bf 1}_{(-\infty,x-\varrho_1]},\mu}](\cdot,x)
\\
&\leq & (1+\gamma) L[e_{\zeta\cdot{\bf 1}_{[x-\varrho_1,\infty)},\mu}](\cdot,x)
\\
&\leq & (1+\delta)(1+\gamma)L_+[e_{\zeta\cdot{\bf 1}_{[x-\varrho_1,\infty)},\mu}](\cdot,x)
\\
&\leq&  (1+\mu \epsilon)L_+[e_{\zeta,\mu}](\cdot,x), \, \forall x\geq \mathfrak{x}_{\epsilon,\mu,\zeta}.
\end{eqnarray*}
This completes the proof.
\end{proof}

Now we are ready to prove  the asymptotic annihilation for  discrete-time systems  under the assumption {\bf (LC)},
which is highly nontrivial since we do not assume (P3) for the map $Q$.

\begin{thm} \label{thm3.4} Let  {\bf (LC)} hold. If there exists $k\in \mathbb{N}$ such that $Q[\varphi]\leq L[\varphi]$ for all $\varphi\in C_{\phi_k^*}$, then the following statements are valid:
\begin{itemize}
\item [{\rm (i)}] If $c_{L}^*:=c_{L_+}^*\geq 0$ and $\varphi\in C_{\phi_k^*}$ is such that $\varphi(\cdot,x) $ is zero for all sufficiently positive $x$, then $\lim\limits_{n\rightarrow \infty}
\sup\left\{
||Q^n[\varphi](\cdot,x)||:x\geq n(c_{L_+}^*+\varepsilon)\right\}=0$ for all $\varepsilon>0$.

\item [{\rm (ii)}] If $c_{L}^*:=c_{L_+}^*< 0$ and $\varphi\in C_{\phi_k^*}$ satisfies that $\varphi(\cdot,x) $ is zero for all sufficiently positive $x$, then $\lim\limits_{\sigma\rightarrow \infty}\sup\left\{||Q^n[\varphi](\cdot,x)||:x\geq \sigma \mbox{ with } n \in \mathbb{N}\right\}=0$, and $\lim\limits_{n\rightarrow \infty}
\sup\left\{||Q^n[\varphi](\cdot,x)||:x\geq n\varepsilon\right\}=0$ for all $\varepsilon>0$.

\end{itemize}
\end{thm}

\noindent
\begin{proof}  Fix $\varepsilon>0$ and $\varphi\in C_{\phi_k^*}$ with $\varphi(\cdot,x) $ being zero for all sufficiently positive $x$. By {\bf(UB)} and $\varphi\leq \phi_{k}^*$, we have $Q^n[\varphi]\leq \phi_{k}^*$ for all $n\in \mathbb{N}$. Take $\alpha\geq ||\phi_{k}^*||$ and $$
\epsilon=\left\{
\begin{array}{ll}
\frac{\varepsilon}{3},  \ \ \ \ \ \ c_{L_+}^*\geq 0 \\
 -c_{L_+}^*, \ \ \ \  c_{L_+}^*<0.
\end{array}
\right.
$$
According to the choices of $\varphi$ and $\alpha$, there is $\beta>\beta_{\epsilon,\alpha}$ such that $\varphi(\theta,x)\leq \beta e^{-\mu_\epsilon x} \zeta_{\mu_\epsilon}(\theta)$, and hence, $\varphi(\theta,x)\leq v_{\alpha,\beta,\mu_\epsilon}(\theta,x)$ for all $(\theta,x)\in M\times \mathbb{R}$, where $\beta_{\epsilon,\alpha}$ and $\mu_\epsilon$ are defined as in Lemma~\ref{lem3.2-000000}. By the  method of induction, we then easily see that  $Q^n[\varphi]\leq \underline{L}^n[v_{\alpha,\beta,\mu_\epsilon}]$ for all $n\in \mathbb{N}$.
It follows from Lemma~\ref{lem3.2-000000} that
$$
Q^n[\varphi](\theta,x)\leq \underline{L}^n[v_{\alpha,\beta,\mu_\epsilon};\alpha,\mu](\theta,x)\leq  \beta e^{\mu_\epsilon(-x+n\max\{0,c_{L_+}^*+\epsilon\})}, \, \forall  (n,\theta,x)\in \mathbb{N}\times M\times \mathbb{R}.
$$
This, together with the choice of $\epsilon$, yields the desired results. \end{proof}

Theorem~\ref{thm3.4} suggests that $\{Q^n\}_{n\in \mathbb{N}}$ may exhibit some  asymptotic annihilation property. It is worth noting that when $Q$ admits the  translation invariance, Theorems~\ref{thm3.4} also provides a new proof for  the asymptotic annihilation property, which is quite different from  the approaches developed in \cite{w1982} and  \cite{bhn2010}.

%=========================+++++++++=============
%=========================+++++++++=============
%=========================+++++++++=============

Next we turn to  the existence and nonexistence of fixed points
for the  unilateral {\bf (UC)} type discrete-time systems.
Recall that $\phi$ is a  nontrivial  fixed point  of the map $Q$
 if $\phi\in C_{+}\setminus\{0\}$  and
$Q[\phi]=\phi$. 
%We say that $\phi(\theta,x)$ connects  $0$ to $r^*$ if $\phi(\cdot,-\infty):=\lim\limits_{s\to-\infty}\phi(\cdot,s)=0$ and $\phi(\cdot,\infty):=\lim\limits_{s\to\infty}\phi(\cdot,s)=r^*$.
%It easily follows  from (A2) and {\bf (UB)} that for any $(\theta,x)\in M\times \mathbb{R}$, $W(\theta,x):=\lim\limits_{n\to \infty}Q^n[\phi_1^*](\theta,x)$ is well defined.

 The following result is a straightforward consequence of Theorem \ref{thm3.1}-(ii) with $\varphi=W$.
\begin{lemma}\label{lemm3.1} Assume that all the assumptions of Theorem~\ref{thm3.1} hold. If $c_+^*>0 $ and  $W$ is  a  nontrivial  fixed point  of the map $Q$ in $C_+\setminus \{0\}$, then $W(\cdot,\infty)=r^*$.
\end{lemma}

\begin{lemma}\label{lemm3.1-999000}  Let {\bf (NM)} hold. If $a>b>0$ and $W$ is  a  nontrivial  fixed point  of the map $Q$ in $C_{a r^*}\setminus \{0\}$ with $W(\cdot,x)\geq b r^*$ for all large $x$, then $W(\cdot,\infty)=r^*$.
\end{lemma}

\begin{proof}
Note that there exists $x_0>0$ such that $W(\cdot,x)\geq b r^*$ for all $x\geq x_0$. Let $t=a+1$, $r=\liminf\limits_{y\to \infty}\Big[\sup\{\beta\in \mathbb{R}_+:W(\cdot,y)\geq \beta r^*\}\Big]$,
and $s=\limsup\limits_{y\to \infty}\Big[\inf\{\beta\in \mathbb{R}_+:W(\cdot,y)\leq\beta r^*\}\Big]$. Then $W(\cdot,x)\leq t r^*$ for all $x\in \mathbb{R}$ and $t>s\geq r\geq b>0$ due to the choices of $x_0,t,r$ and $s$.   Now, it suffices to prove $r=s=1$. Otherwise, $\{r,s\}\neq \{1\}$.  In view of  {\bf (NM)}, there exist $N_1:=n(r,s)\in \mathbb{N}$, $\gamma_1:=\gamma(r,s)\in (0,r) $, and $(d_1,z_1):=(d(r,s,t),z(r,s,t))\in
 Int(\mathbb{R}_+^2)$
such that either  $I_{d_1,r-\gamma_1,s+\gamma_1,z_1,N_1,\frac{t+\gamma_1}{r-\gamma_1}}>r+\gamma_1$,  or
$S_{d_1,r-\gamma_1,s+\gamma_1,z_1,N_1,\frac{t+\gamma_1}{r-\gamma_1}}<s-\gamma_1$.  By the choices of $r,s$ and $t$, there exists $y_1> z_1$ such that $$
T_{-y}[W]\in [(r-\gamma_1)\xi_{d_1}r^{*},(s+\gamma_1)\tilde{\xi}_{d_1,\frac{t+\gamma_1}{r-\gamma_1}}r^{*}]_C, \,  \forall y\geq y_1.
$$
It then follows that either
 $$
 W(\cdot,y)=T_{-y}\circ Q^{N_1} \circ T_{y}[T_{-y}[W]](\cdot,0)\geq (r+\gamma_1)r^*,  \,  \forall y\geq y_1,
 $$
 or
 $$
 W(\cdot,y)=T_{-y}\circ Q^{N_1} \circ T_{y}[T_{-y}[W]](\cdot,0)\leq (s-\gamma_1)r^*,  \,  \forall y\geq y_1.
 $$
 This leads to a contradiction to the choices of $r$ and $s$. Thus, the desired conclusion holds true.
\end{proof}

\begin{lemma} \label{lemma-prop3.5}
 Assume that $Q_+$ satisfies {\bf (UC)}.  If there exist $r^{**}\in Int(Y_{+})$ and $x_0\in (-c_+^*,c_-^*)\cap \mathbb{R}_+$ such that $Q[r^{**}]\leq r^{**}$,  $Q|_{C_{{r^{**}}}}$ is monotone,  and $\{(T_{x_0}\circ Q)^n [r^{**}]:n\in \mathbb{N}\}$ is precompact in $C$ for some $x_0\in (-c_+^*,c_-^*)\cap \mathbb{R}_+$, then the following statements are valid:

\begin {itemize}
\item [{\rm (i)}]  If $x_0=0$, then $Q$ has a  nontrivial  fixed point $W$ in $C_{r^{**}}$ such that
$W(\cdot,\infty)=r^*$. Moreover,  $W(\cdot,x)$ is nondecreasing in $x$ provided that $ T_{-y}\circ Q[\varphi]\geq Q\circ T_{-y}[\varphi]$ for all $(y,\varphi)\in \mathbb{R}_+\times C_{r^{**}}$.

 \item [{\rm (ii)}] If $x_0>0$ and  $ T_{-y}\circ Q[\varphi]\geq Q\circ T_{-y}[\varphi]$ for all $(y,\varphi)\in \mathbb{R}_+\times C_{r^{**}}$, then $T_{x_0}\circ Q$ has a  nontrivial  fixed point $W$ in $C_{r^{**}}$ such that
$W(\cdot,\infty)=r^*$ and  $W(\cdot,x)$ is nondecreasing in $x$.
\end {itemize}
\end{lemma}

\begin{proof}
Note that $\{c_+^*+x_0,c_-^*-x_0\}>0$. By Proposition~\ref{prop2.30000} and the conditions of $Q,Q_{+}$,  we  easily see $r^{**}\geq r^{*}$. Let $\mathcal{Q}=T_{x_0}\circ Q$ and
$$
W_n(\cdot,x)=\mathcal{Q}^n[r^{**}](\cdot,x),   \,  \forall (x,n)\in \mathbb{R}\times \mathbb{N}.
$$
Then $0\leq W_{n+1}\leq W_n\leq r^{**}$ for all $n\in \mathbb{N}$. By the compactness of $\{\mathcal{Q}^n [r^{**}]:n\in \mathbb{N}\}$ in $C$,   it follows that there exists $W\in C_{r^{**}}$ such that $\lim\limits_{n\to \infty}||W_n-W||=0$, and hence, $T_{x_0}\circ Q[W]=W$.

 By applying Proposition~\ref{prop2.1} to $Q$ with $c=-x_0$ and $\varepsilon=\frac{\min\{c_-^*-x_0,c_+^*+x_0,c_+^*+c_-^*\}}{3}$, we see that there exist $n_0,y_0>0$ such that
 $$
  T_{-y}\circ (T_{n_0 x_0}\circ Q^{n_0})^n\circ T_y [\frac{r^{**}}{16}h](\cdot,0)\geq T_{-y}\circ (T_{n_0 x_0}\circ Q^{n_0})^n\circ T_y [\frac{r^*}{16}h](\cdot,0)\geq \frac{r^*}{4}>0
 $$
 and
 $$T_{-y}\circ (T_{n_0 x_0}\circ Q^{n_0})^n [r^{**}](\cdot,0)\geq  T_{-y}\circ (T_{n_0 x_0}\circ Q^{n_0})^n [r^*](\cdot,0)\geq \frac{r^*}{4}>0
 $$
 for all  $(n,y)\in \mathbb{N}\times[y_0,\infty)$.  Then we have the following
 claim.

{\bf Claim.} $W_{nn_0}(\cdot,y)\geq \frac{r^*}{4}$ for all $(n,y)\in \mathbb{N}\times[y_0,\infty)$.

Indeed, if $x_0=0$, then we have
$$
W_{nn_0}(\cdot,y)=T_{-y}\circ  (Q^{n_0})^n [r^{**}](\cdot,0)\geq \frac{r^*}{4}>0,    \,  \forall (n,y)\in \mathbb{N}\times[y_0,\infty).
$$
If $x_0>0$ and $ T_{-y}\circ Q[\varphi]\geq Q\circ T_{-y}[\varphi]$ for all $(y,\varphi)\in \mathbb{R}_+\times C_{r^{**}}$, then $  Q\circ T_{x_0}[\varphi] \geq T_{x_0}\circ Q [\varphi]$ for all $\varphi\in  C_{r^{**}}$.
It follows that $(T_{x_0}\circ Q)^{n_0}\geq T_{n_0x_0}\circ Q^{n_0}$, and hence, for any $(n,y)\in \mathbb{N}\times[y_0,\infty)$, there holds
\begin{eqnarray*}
W_{nn_0}(\cdot,y)&=&(T_{x_0}\circ Q)^{nn_0}[r^{**}](\cdot,y)
\\
&\geq& T_{-y}\circ  (T_{n_0x_0}\circ Q^{n_0})^n[r^{**}](0,\cdot)\geq \frac{r^*}{4}>0.
\end{eqnarray*}
This proves the claim.  Thus, $W(\cdot,y)\geq \frac{r^*}{4}$ for all $y\geq y_0$.
In view of  Lemmas~\ref{lemm4.1} and \ref{lemm3.1}, we obtain $W(\cdot,\infty)=r^*$.   Additionally,  we easily verify that $W(\cdot,x)$ is nondecreasing in $x$ provided that $ T_{-y}\circ Q[\varphi]\geq Q\circ T_{-y}[\varphi]$ for all $(y,\varphi)\in \mathbb{R}_+\times C_{r^{**}}$.
\end{proof}

%and $\mathcal{Q}_+=T_{x_0}\circ Q_+$. It is easy to verify that $\mathcal{Q},\mathcal{Q}_+$ satisfy {\bf (A)}, {\bf (B)}, and {\bf (UC)} with $(r^*,c_\pm^*)$ replaced by $(r^*,c_\pm^*\pm x_0)$. Define

\begin{lemma}\label{lemm3.2} Assume that there exist $r^{**}\in Int(Y_{+})$ and $W\in C_{r^{**}}\setminus \{0\}$ such that $Q[r^{**}]\leq r^{**}$, $Q|_{C_{r^{**}}}$ is monotone, and $W$ is  a  nontrivial  fixed point  of the map $Q$ in $C_{r^{**}}\setminus \{0\}$ with $W(\cdot,s)$ being nondecreasing in $s\in \mathbb{R}$. If $\{(T_{-z}\circ Q)^n [r^{**}]:n\in \mathbb{N}\}$ is
 precompact in $C$ for some  $z\in (0,\infty)$, then  $T_{-z}\circ Q$ has a  nontrivial  fixed point $W_z$ in $C_+$. Moreover, if  $W(\cdot,\infty)=r^*$ and $ T_{-y}\circ Q[\varphi]\geq Q\circ T_{-y}[\varphi]$ for all $(y,\varphi)\in \mathbb{R}_+\times C_{r^{**}}$, then $W_z(\cdot,\infty)\geq r^*$ and $W_z(\cdot,x)$ is nondecreasing in $x$.
\end{lemma}

\begin{proof}  By  the conditions on $W$ and $z$, it follows that $r^{**}\geq T_{-z}\circ Q[W]=T_{-z}[W]\geq W>0$.  We can define $W_z:=\lim\limits_{n\to \infty} (T_{-z}\circ Q)^n [r^{**}]$ due to the compactness and monotonicity. Then $T_{-z}\circ Q[W_z]=W_z$ and $W\leq W_z\leq r^{**}$.  Moreover, if the additional conditions hold, then we easily see that $W_z(\cdot,\infty)\geq r^*$  and $W_z(\cdot,x)$ is nondecreasing in $x$. \end{proof}

With the help of  Lemmas~\ref{lemma-prop3.5} and \ref{lemm3.2}, we can easily prove the following result.
\begin{prop} \label{prop3.5}
 Assume that $c_-^*>0$  and $Q_+$ satisfies {\bf (UC)}.  If   there exists $r^{**}\in Int(Y_{+})$  such that $Q[r^{**}]\leq r^{**}$, $Q|_{C_{r^{**}}}$ is monotone, and  $\{Q^n [r^{**}]:n\in \mathbb{N}\}$ is precompact in $C$,
   then the following statements are valid:

\begin {itemize}
\item [{\rm (i)}] If $c_+^*>0$, then
 $Q$ has a  nontrivial  fixed point $W$ in $C_{r^{**}}$ such that
$W(\cdot,\infty)=r^*$ and $\lim\limits_{n\to \infty}||Q^n[r^{**}]-W||=0$. Moreover,  $W(\cdot,s)$ is nondecreasing in $s\in \mathbb{R}$ provided that $ T_{-y}\circ Q[\varphi]\geq Q\circ T_{-y}[\varphi]$ for all $(y,\varphi)\in \mathbb{R}_+\times C_{r^{**}}$.

\item [{\rm (ii)}] If $c_+^*\leq 0$, $\{(T_{x_0}\circ Q)^n [r^{**}]:n\in \mathbb{N}\}$ is precompact in $C$ for some $x_0\in (-c_+^*,c_-^*)\cap \mathbb{R}_+$, and $ T_{-y}\circ Q[\varphi]\geq Q\circ T_{-y}[\varphi]$ for all $(y,\varphi)\in \mathbb{R}_+\times C_{r^{**}}$,
then $Q$ has a  nontrivial  fixed point $W$ in $C_{r^{**}}$ such that
$W(\cdot,\infty)\geq r^*$ and $\lim\limits_{n\to \infty}||Q^n[r^{**}]-W||=0$, and $W(\cdot,s)$ is nondecreasing in $s\in \mathbb{R}$.
\end {itemize}
\end{prop}

\begin{thm} \label{thm3.6}
Assume that $Q$ satisfies  {\bf (ACH)}, {\bf (NM)}, and  $Q$ has at least one fixed point in $\mathcal{K}$ for any closed, convex, and positively invariant set $\mathcal{K}$ of $Q$. Let $c_-^*>0$  and $T_{-y}\circ Q_l[\varphi]\geq Q_l\circ T_{-y}[\varphi]$ for all $l\in \mathbb{N}$ and $(y,\varphi)\in \mathbb{R}_+\times C_{\phi_{1}^{*}}$. If  for any $l\in \mathbb{N}$, there exists   $z:=z_l\in (-c_{l,+}^*,c_{l,-}^*)\cap \mathbb{R}_+$ such that $\{( Q_l)^n [\phi_{1}^{*}]:n\in \mathbb{N}\}$ and $\{(T_z\circ Q_l)^n [\phi_{1}^{*}]:n\in \mathbb{N}\}$ are precompact in $C$,  then $Q$ has a  nontrivial  fixed point $W$ in $C_{\phi_1^*}$ such that
$W(\cdot,\infty)=r^*$.
\end{thm}

\begin{proof}  Clearly, there is $l_0\in \mathbb{N}$ such that $c_{l_0,-}^*>0$.  In view of { \bf (ACH)}-(i), we have $r_{l_0}^*\leq  \phi_1^*$. By applying Proposition~\ref{prop3.5} to
$Q_{l_0}|_{C_{ \phi_{1}^{*}}  }$, we obtain that $Q_{l_0}[W^*]=W^*$ for some $W^*\in C_{\phi_{1}^{*} }$ with $W^*(\cdot,\infty)\geq r_{l_0}^*$. Let $\mathcal{K}=[W^*,\phi_1^*]_C$. Then $Q[{\mathcal{K}}]\subseteq{\mathcal{K}}\neq \emptyset$  and $\mathcal{K}$ is a closed and convex subset of $C$. It follows that we obtain that $Q$ has a  nontrivial  fixed point $W$ in $[W^*,\phi_1^*]_C$, and hence, there exists $x_0>0$ such that $\frac{r_{l_0}^*}{2}\leq W(\cdot,x)\leq \phi_1^*$ for all $x\geq x_0$. Thus, the desired conclusion follows from  Lemma~\ref{lemm3.1-999000}.
\end{proof}

\begin{rem}  \label{rem3.1-sp-monoty}
  In view of  Proposition~\ref{prop3.5}-(i),   in Therorem~\ref{thm3.6} we can use $c_+^*>0$ to replace  the assumption that  $T_{-y}\circ Q_l[\varphi]\geq Q_l\circ T_{-y}[\varphi]$ for all $(l, y,\varphi)\in \mathbb{N}\times
 \mathbb{R}_+\times C_{r^*_l}$.
  \end{rem}
\subsection{The bilateral limit case}
Next we study the  propagation dynamics for the discrete-time
systems in  the bilateral limit case.  For convenience,  we define the reflection operator $\mathcal{S}$ on
$C$ by
$$[\mathcal{S}(\varphi)](\theta,x)= \varphi(\theta,-x),   \, \forall
(\theta,x) \in M \times  \mathbb{R},  \, \varphi\in C.
$$
We need  a series of basic assumptions.
\begin{enumerate}
\item [{\bf (A$_-$)}] There exist  $r_{-}^*\in Int(Y_+)$ and a continuous map $Q_-:  C_+\to C_+$ such that $Q_-[r_{-}^*]=r_{-}^*$ and  $\lim\limits_{y\to -\infty}T_{-y}\circ Q^n \circ T_y[\varphi]= Q_-^n[\varphi]$  in $C$ for all $\varphi\in C_+$ and $n\in \mathbb{N}$.

\item [{\bf (SP$_-$)}]  There exist $(\rho^*_-,N^*_-)\in (0,\infty)\times \mathbb{N}$ and  $\varrho^*_-\in (\rho^*_-,\infty)$  such that for any  $a\in \mathbb{R}$  and $\varphi\in
C_{+}\setminus\{0\}$ with   $\varphi(\cdot, a)\subseteq Y_+\setminus \{0\}$,
we have  $Q^n[\varphi](\cdot,x)\in Int(Y_+)$ for all  $n\geq N^*_-$ and $x-a\in [-n\varrho^*_-,-n\rho^*_-]$.

%\item [{\bf(UB$_-$)}]  There exists a sequence $\{\phi_{-k}^*\}_{k\in \mathbb{N}}$ in $Int(Y_+)$ such that $\phi_{-k-1}^*>\phi_{-k}^*$, $Y_+\subseteq\bigcup\limits_{k\in %\mathbb{N}}\phi_{-k}^*-Y_+$, and $Q[\phi_{-k}^*]\leq \phi_{-k}^*$  for all $k\in \mathbb{N}$.

\item [{\bf(UC$_-$)}]  There exist $c_-^*(-\infty),c_+^*(-\infty)\in \mathbb{R}$   such that $c_+^*(-\infty)+c_-^*(-\infty)>0$, and
$$
\lim\limits_{n\rightarrow \infty}
\max\limits_{x\in \mathcal{A}_{\varepsilon,n}^-}
||Q_-^n[\varphi](\cdot,x)-r_{-}^*(\cdot)||=0, \, \forall \varepsilon\in (0,\frac{c_+^*(-\infty)+c_-^*(-\infty)}{2}), \, \varphi\in
C_{+}\setminus\{0\},
$$
where $\mathcal{A}_{\varepsilon,n}^-= n{[-c_-^*(-\infty)+\varepsilon,c_+^*(-\infty)-\varepsilon]}$.

\item [{\bf(ACH$_-$)}] There exist sequences $\{c_{-l,\pm}^*\}_{l=1}^\infty$ in $\mathbb{R}$, $\{r_{-l}^*\}_{l=1}^\infty$ in $Int(Y_+)$, $\{Q_{-l}\}_{l=1}^\infty$, and $\{Q_{-l,-}\}_{l=1}^\infty$ such that for any positive integer  $l$, there hold
\begin{enumerate}
\item [(i)]   $c_+^*(-\infty):=\lim\limits_{j\to \infty}c_{-j,+}^*$,  $c_-^*(-\infty):=\lim\limits_{j\to \infty}c_{-j,-}^*$, and $r_{-l}^*\leq \phi_1^*$.

\item [(ii)]   $Q_{-l},Q_{-l,-}:C_+ \to C_+$ are continuous and monotone maps with $Q\geq Q_{-l}$ in $C_{\phi_l^*}$,  where $\phi_l^*$ is defined as in {\bf (UB)}.

\item [(iii)]  $Q_{-l}$ is subhomogeneous on $[0,r_{-l}^*]_ C$.

\item [(iv)]      $Q_{-l,-}$  satisfies {\bf (UC$_-$)} with  $(c_{-l,-}^*,c_{-l,+}^*,r^*_{-l})$.

\item [(v)]   $(Q_{-l},Q_{-l,-},r^*_{-l})$  satisfies {\bf (A$_-$)}.
\end{enumerate}

\item [{\bf(NM$_-$)}]  $\mathcal{S}\circ Q \circ \mathcal{S} $ satisfies {\bf(NM)} with $r^*$ replaced by $r_-^*$.

\item [{\bf(LC$_-$)}]  $(\mathcal{S}\circ L \circ \mathcal{S}, \mathcal{S}\circ L_- \circ \mathcal{S}) $ satisfies {\bf(LC)} with $(L,L_+)$ replaced by $(\mathcal{S}\circ L \circ \mathcal{S}, \mathcal{S}\circ L_- \circ \mathcal{S})$.

\item [{\bf (GLC)}]  There exists a sequence of linear operators $\{(L_k,L_{k,+})\}_{k\in \mathbb{N}}$   with  {\bf (LC)}
such that
$$
Q[\varphi]\leq L_k[\varphi],  \, \forall \varphi\in C_{\phi_k^*},  \, k\in \mathbb{N}.
$$

\item [{\bf (GLC$_-$)} ] There exists a sequence of linear operators $\{(\hat{L}_k,\hat{L}_{k,-})\}_{k\in \mathbb{N}}$   with  {\bf (LC$_-$)}
such that
$$
Q [\varphi]\leq
\hat{L}_k[\varphi],\, \forall \varphi\in C_{\phi_k^*},  \, k\in \mathbb{N}.
$$
\end{enumerate}
Define
$$
 \tilde{Q}[\varphi](\theta,x)=Q[\varphi(\cdot,-\cdot)](\theta,-x), \,
 \tilde{Q}_+[\varphi](\theta,x)=Q_-[\varphi(\cdot,-\cdot)](\theta,-x), \,
 \forall (\varphi,\theta,x)\in C_+\times M\times \mathbb{R}.
 $$
 Clearly, $\tilde{Q}=\mathcal{S}\circ Q \circ \mathcal{S}$, and $\tilde{Q}_+=\mathcal{S}\circ Q_- \circ \mathcal{S}$.
 It then follows that $Q$ and $Q_-$ satisfy {\bf (A$_-$)}, {\bf(SP$_-$)}, {\bf(UC$_-$)}, {\bf(NM$_-$)}, {\bf(ACH$_-$)}, {\bf (LC$_-$)}, {\bf (GLC$_-$)}  if and only if $\tilde{Q}$ and $\tilde{Q}_+$ satisfy {\bf (A)}, {\bf(SP)}, {\bf(UC)}, {\bf(NM)},{\bf(ACH)}, {\bf (LC)}, {\bf (GLC)}, respectively, with different parameters.

%In the following, we always assume that $Q$ satisfies {\bf (A)} whenever $Q$ satisfies {\bf (UC)}, and  $Q$ satisfies {\bf (A$_-$)} whenever  $Q$ satisfies {\bf (UC$_-$)},   unless specified otherwise.

Let  ${\bf {1}}_{\mathbb{R}_+}:\mathbb{R} \to \mathbb{R}$ be defined by ${\bf {1}}_{\mathbb{R}_+}(x)=1, \forall x\in \mathbb{R}_+$, and   ${\bf {1}}_{\mathbb{R}_+}(x)=0, \forall x\in (-\infty,0)$.
Our first result is on  the upward convergence under  the bilateral {\bf(UC)}/{\bf(UC)}-type assumption.

\begin{thm} \label{thm3.1-bil-uc/uc}
Assume that $Q$ satisfies  {\bf (ACH)}, {\bf (NM)}, {\bf (SP)},  {\bf (ACH$_-$)}, {\bf (NM$_-$)}, and {\bf (SP$_-$)}.
Let $\varphi\in C_{+}\setminus\{0\}$. Then the following statements are valid:
\begin{itemize}
\item [{\rm (i)}]
If $\min\{c_+^*,c_-^*(-\infty)\}>0$, then $$\lim\limits_{n\rightarrow \infty}
\max\left\{||Q^n[\varphi](\cdot, x)-r^*\cdot {\bf {1}}_{\mathbb{R}_+}(x)-r^*_-\cdot (1-{\bf {1}}_{\mathbb{R}_+}(x))||:x\in\mathcal{B}_{n,\varepsilon}\right\}= 0,$$ where
$\varepsilon\in (0,\frac{1}{2}\min\{c_+^*,c_-^*(-\infty),c_+^*+c_-^*,c_+^*(-\infty)+c_-^*(-\infty)\})$ and $\mathcal{B}_{n,\varepsilon}=[n\max\{\varepsilon,-c_-^*+\varepsilon\}, n(c_+^*-\varepsilon)]\bigcup [n(-c_-^*(-\infty)+\varepsilon),n\min\{-\varepsilon,c_+^*(-\infty)-\varepsilon\}]$.

\item [{\rm (ii)}] If $\min\{c_-^*,c_+^*,c_-^*(-\infty),c_+^*(-\infty)\}>0$, then $$\lim\limits_{\alpha \rightarrow \infty}
\sup\left\{||Q^n[\varphi](\cdot, x)-r^*\cdot {\bf {1}}_{\mathbb{R}_+}(x)-r^*_-(1-{\bf {1}}_{\mathbb{R}_+}(x))||:(n,x)\in \mathcal{C}_{\alpha,\varepsilon}\right\}= 0,$$
where $\varepsilon\in (0,\min\{c_+^*,c_-^*(-\infty)\})$ and $\mathcal{C}_{\alpha,\varepsilon}=\{(n,x)\in \mathbb{N}\times \mathbb{R}:n\geq \alpha \mbox{ and } x \in [\alpha,  n(c_+^*-\varepsilon)] \bigcup [n(-c_-^*(-\infty)+\varepsilon),-\alpha]\}$.
\end{itemize}
\end{thm}

\begin{proof}
By applying Theorem~\ref{thm3.1-0}-(ii) to  $\{Q, \varphi\}$ and $\{\tilde{Q}, \mathcal{S}[\varphi]\}$, we have
$$
\lim\limits_{n\rightarrow \infty}
\Big[\max\{||Q^n[\varphi](\cdot, x)-r^*||:n\max\{\varepsilon,-c_-^*+\varepsilon\}\leq x \leq n(c_+^*-\varepsilon)\}\Big]= 0
$$
for all $\varepsilon\in (0,\frac{1}{2}\min\{c_+^*,c_+^*+c_-^*\})$,
and
$$\lim\limits_{n\rightarrow \infty}
\Big[\max\{||\tilde{Q}^n[\mathcal{S}[\varphi]](\cdot, x)-r^*_-||:n\max\{\varepsilon,-c_+^*(-\infty)+\varepsilon\}\leq x \leq n(c_-^*(-\infty)-\varepsilon)\}\Big]= 0$$
 for all $\varepsilon\in (0,\frac{1}{2}\min\{c_-^*(-\infty),c_+^*(-\infty)+c_-^*(-\infty)\})$. The latter implies that
 $$\lim\limits_{n\rightarrow \infty}
\Big[\max\{||\mathcal{S}\circ {Q}^n[\varphi](\cdot, x)-r^*_-||:n\max\{\varepsilon,-c_+^*(-\infty)+\varepsilon\}\leq x \leq n(c_-^*(-\infty)-\varepsilon)\}\Big]= 0,
$$
and hence,
$$\lim\limits_{n\rightarrow \infty}\Big[
\max\{|| {Q}^n[\varphi](\cdot, x)-r^*_-||: n(-c_-^*(-\infty)+\varepsilon)\leq x \leq n\min\{-\varepsilon,c_+^*(-\infty)-\varepsilon\}\}\Big]= 0
$$
for all $\varepsilon\in (0,\frac{1}{2}\min\{c_-^*(-\infty),c_+^*(-\infty)+c_-^*(-\infty)\})$.
Thus, statement (i) holds true.

(ii) By applying Theorem~\ref{thm3.1}-(ii) with $c=0$ to $\{Q, \varphi\}$ and $\{\tilde{Q}, \mathcal{S}[\varphi]\}$, we obtain
 $$
 \lim\limits_{\alpha \rightarrow \infty}
\Big[\sup\{||Q^n[\varphi](\cdot, x)-r^*||:n\geq \alpha \mbox{ and } \alpha \leq x \leq n(c_+^*-\varepsilon)\}\Big]= 0
$$
 for all $\varepsilon\in (0,c_+^*)$,
and
$$\lim\limits_{\alpha \rightarrow \infty}
\Big[\sup\{||\tilde{Q}^n[\mathcal{S}[\varphi]](\cdot, x)-r^*_-||:n\geq \alpha \mbox{ and } \alpha \leq x \leq n(c_-^*(-\infty)-\varepsilon)\}\Big]= 0
$$
for all $\varepsilon\in (0,c_-^*(-\infty))$. It then follows that
$$\lim\limits_{\alpha \rightarrow \infty}
\Big[\sup\{||{Q}^n[\varphi](\cdot, -x)-r^*_-||:n\geq \alpha \mbox{ and } \alpha \leq x \leq n(c_-^*(-\infty)-\varepsilon)\}\Big]= 0
$$
for all $ \varepsilon\in (0,c_-^*(-\infty))$, and hence,
$$\lim\limits_{\alpha \rightarrow \infty}
\Big[\sup\{||{Q}^n[\varphi](\cdot, x)-r^*_-||:n\geq \alpha \mbox{ and } -\alpha \geq x \geq n(-c_-^*(-\infty)+\varepsilon)\}\Big]= 0
$$
for all $\varepsilon\in (0,c_-^*(-\infty))$.  This shows that statement (ii) holds
true. \end{proof}

Applying Theorem~\ref{thm3.4} to $\{(Q, \varphi),(\tilde{Q}, \mathcal{S}[\varphi])\}$, we have the following
result on the asymptotic annihilation under the bilateral {\bf (GLC)} /{\bf (GLC)}-type assumption.

\begin{thm} \label{thm3.1-bil-meal/meal}
Assume that   {\bf (GLC)},  and {\bf (GLC$_-$)}  hold. Let $\varphi\in { C_{+}}$ with the compact support,
$c^*_1:=\liminf\limits_{k\to \infty}c^*_{L_k}$, and
$c^*_2:=\liminf\limits_{k\to \infty}c^*_{\mathcal{S}\circ\hat{L}_k\circ \mathcal{S}}$.  Then the following statements are valid:
\begin{itemize}
\item [{\rm (i)}] If $c^*_1\geq 0$ and  $c^*_2\geq 0$, then $$\lim\limits_{n\rightarrow \infty}
\sup\left\{
||Q^n[\varphi](\cdot,x)||: x\geq  n(c^*_1+\varepsilon) \mbox{ or } x\leq - n(c^*_2+\varepsilon)\right\}=0,  \forall \varepsilon>0.$$

\item [{\rm (ii)}] If $c^*_1< 0$ and  $c^*_2\geq 0$, then  $$\lim\limits_{n\rightarrow \infty}\sup\left\{
||Q^n[\varphi](\cdot,x)||:x\geq n\varepsilon \mbox{ or } x\leq - n(c^*_2+\varepsilon)\right\}=0,  \forall \varepsilon>0.$$

\item [{\rm (iii)}] If $c^*_1\geq 0$ and  $c^*_2< 0$, then  $$\lim\limits_{n\rightarrow \infty}\sup\left\{
||Q^n[\varphi](\cdot,x)||:x\leq -n\varepsilon \mbox{ or } x\geq  n(c^*_1+\varepsilon)\right\}=0,
\forall \varepsilon>0.$$

\item [{\rm (iv)}] If $c^*_1< 0$ and  $c^*_2< 0$, then  $$\lim\limits_{\alpha\rightarrow \infty}\sup\left\{||Q^n[\varphi](\cdot,x)||:|x|,n\geq \alpha \mbox{ with } n \in \mathbb{N}\right \}=0$$
 and $$\lim\limits_{n\rightarrow \infty}\sup\left\{
||Q^n[\varphi](\cdot,x)||:|x|\geq n\varepsilon\right\}=0,
\forall \varepsilon>0.
$$

%$\lim\limits_{\sigma\rightarrow \infty}\Big[\sup\{||Q^n[\varphi](\cdot,x)||:x,n\geq \sigma \mbox{ with } n \in \mathbb{N} \}\Big]=0$
\end{itemize}
\end{thm}

\begin{proof}   We only prove (i) since the other statements can be dealt with some similar arguments. Fix $\varepsilon>0$ and $\varphi\in C_+$ with  $\varphi$ having the compact support. In view of  {\bf (UB)}, {\bf (GLC)},  and {\bf (GLC$_-$)}, we have $\varphi\leq \inf\{\phi_{k_0}^*, \phi_{k_1}^*\}$, $|c^*_{L_{k_0}}-c^*_1|<\frac{\varepsilon}{6}$,  and $|c^*_{\mathcal{S}\circ\hat{L}_{k_1}\circ \mathcal{S}}-c^*_2|<\frac{\varepsilon}{6}$  for some $(k_0,k_1)\in \mathbb{N}^2$. %, $Q_n[C_{\phi_{k_0}^*}]\subseteq C_{\phi_{k_0}^*}$,
By applying Theorem~\ref{thm3.4}-(i) to $\{(Q, L_{k_0},\varphi,\frac{\varepsilon}{3}),(\tilde{Q},\mathcal{S}\circ\hat{L}_{k_1}\circ \mathcal{S}, \mathcal{S}[\varphi],\frac{\varepsilon}{3})\}$,  we obtain
\begin{equation}\label{4.3}
\lim\limits_{n\rightarrow \infty}
\sup\left\{
||Q^n[\varphi](\cdot,x)||:x\geq n(c^*_{L_{k_0}}+\frac{\varepsilon}{3}) \mbox{ or } x\leq -n(c^*_{\mathcal{S}\circ\hat{L}_{k_0}\circ \mathcal{S}}+\frac{\varepsilon}{3})\right \}=0.
\end{equation}
It then follows from~\eqref{4.3} and choices of $k_0$ and $k_1$ that $$\lim\limits_{n\rightarrow \infty}
\sup\left\{
||Q^n[\varphi](\cdot,x)||: x\geq  n(c^*_1+\varepsilon) \mbox{ or } x\leq - n(c^*_2+\varepsilon)\right\}=0,  \forall \varepsilon>0.$$
This completes the proof.
\end{proof}

Recall that $W$ is a  nontrivial  fixed point  of the map $Q$
 if $W\in C_{+}\setminus \{0\}$  and
$Q[W]=W$. 
%We say that $\phi(\theta,x)$ connects  $\xi \in Y_+$ to $\eta\in Y_+$ if$W(\cdot,-\infty):=\lim\limits_{s\to-\infty} W(\cdot,s)=\xi$ and $W(\cdot,\infty):=\lim\limits_{s\to\infty} W(\cdot,s)=\eta$.
%It easily follows  from (A2) and {\bf (UB)} that for any $(\theta,x)\in M\times \mathbb{R}$, $W(\theta,x):=\lim\limits_{n\to \infty}Q^n[\phi_1^*](\theta,x)$ is well defined
The following result is on the existence of fixed points  under
the bilateral {\bf(UC)}/{\bf(UC)}-type assumption.

\begin{thm} \label{thm3.6-bil-uc/uc-fixp}
Assume that   {\bf (ACH)}, {\bf (NM)}, {\bf (SP)},  {\bf (ACH$_-$)}, {\bf (NM$_-$)}, and {\bf (SP$_-$)} hold,  and  $Q$ has at least one fixed point in $\mathcal{K}$ for any closed, convex and positively invariant set $\mathcal{K}\subseteq C_{ \phi_1^* }$ of $Q$.   Let  $T_{-y}\circ Q_l[\varphi]\geq Q_l\circ T_{-y}[\varphi]$ for all $l\in \mathbb{Z}\setminus \{0\}$ and $(y,\varphi)\in \mathbb{R}\times C_{\phi_{1}^{*}}$ with $ly\geq 0$,  where $Q_l$ is defined as in  {\bf (ACH)} and {\bf (ACH$_-$)}.
Then $Q$ has a  nontrivial  fixed point $W$ in $C_{ \phi_{1}^* }$ such that
$W(\cdot,\pm\infty)=r^*_\pm$ provided that one of  the following conditions
is satisfied:
\begin{itemize}
\item [{\rm (i)}]
 $\min\{c_-^*,c_+^*(-\infty)\}>0$ and  for any $|l|\in \mathbb{Z}_+\backslash \{0\}$, there exists   $z:=z_l\in (-c_{l,+}^*,c_{l,-}^*)$ such that $l z_l\geq 0$, $\{( Q_l)^n [ \phi_{1}^{*} ]:n\in \mathbb{N}\}$, and $\{(T_z\circ Q_l)^n [ \phi_{1}^{*}]:n\in \mathbb{N}\}$ are precompact in $C$.

\item [{\rm (ii)}]
 $\min\{c_-^*,c_-^*(-\infty)\}>0$ and  for any $l\in \mathbb{Z}_+\backslash \{0\}$, there exists   $z:=z_l\in (-c_{l,+}^*,c_{l,-}^*)\cap \mathbb{R}_+$ such that $\{( Q_l)^n [ \phi_{1}^{*} ]:n\in \mathbb{N}\}$ and $\{(T_z\circ Q_l)^n [ \phi_{1}^{*}]:n\in \mathbb{N}\}$ are precompact in $C$.

\item [{\rm (iii)}]
 $\min\{c_+^*,c_+^*(-\infty)\}>0$ and  for any $-l\in \mathbb{Z}_+\backslash \{0\}$, there exists   $z:=z_l\in (-c_{l,+}^*,c_{l,-}^*)\cap (-\infty,0]$ such that $\{( Q_l)^n [ \phi_{1}^{*} ]:n\in \mathbb{N}\}$ and $\{(T_z\circ Q_l)^n [ \phi_{1}^{*} ]:n\in \mathbb{N}\}$ are precompact in $C$.
\end{itemize}
\end{thm}

\begin{proof}
	(i)  Clearly, there is $l_0\in \mathbb{N}$ such that $c_{l_0,-}^*>0$ and $c_{-l_0,+}^*>0$. By applying Theorem~\ref{thm3.6} to $Q_{l_0}$ and $\mathcal{S}\circ Q_{-l_0}\circ \mathcal{S}$, we obtain that $Q_{l_0}[W^*_+]=W^*_+$ and $Q_{-l_0}[W^*_-]=W^*_-$ for some $(W^*_+,W^{*}_-)\in C_{\phi_{1}^*}\times C_{\phi_{1}^*}$ with $W^*_+(\cdot,\infty)=r_{l_0}^*$ and $W^*_-(\cdot,-\infty)=r_{-l_0}^*$. Let $\mathcal{K}=[W^*_+, \phi_1^* ]_C\cap [W^*_-,\phi_1^*]_C$. Then $Q[\mathcal{K}]\subseteq \mathcal{K}\neq \emptyset$  and $\mathcal{K}$ is a closed and convex subset of $C$. It follows  that $Q$ has a  nontrivial  fixed point $W$ in $ {\mathcal{K}}$. Thus, there exists $x_0>0$ such that $\frac{\inf\{r^*_{l_0},r^*_{-l_0}\}}{2}\leq W(\cdot,x)\leq \phi_1^*$ for all $|x|\geq x_0$. Now the desired conclusion follows from Lemma~\ref{lemm3.1-999000}.

(ii) Theorem~\ref{thm3.6}  implies that $Q$ has a  nontrivial  fixed point $W$ in $C_{\phi_{1}^* }$ with  $W(\cdot,\infty)=r^*$. By applying Lemma~\ref{lemm3.1} to $\{(\mathcal{S}[W],\tilde{Q})\}$, we have $\mathcal{S}[W](\cdot,\infty)=r_-^*$, and hence, $W(\cdot,-\infty)=r_-^*$

(iii)  follows from statement (ii), as applied  to $\tilde{Q}$.
\end{proof}

\begin{rem}  \label{rem3.2-sp-monoty}  By Remark~\ref{rem3.1-sp-monoty}, it follows that in
 Theorem~\ref{thm3.6-bil-uc/uc-fixp}, we can use  the condition that
$\min\{c_+^*,c_-^*,c_+^*(-\infty),c_-^*(-\infty)\}>0$ to replace
the assumption that  $T_{-y}\circ Q_l[\varphi]\geq Q_l\circ T_{-y}[\varphi]$ for all  $(l, y,\varphi)\in \mathbb{Z}\setminus\{0\}
\times \mathbb{R}\times C_{r^*_l}$ with $ly \geq 0$.

\end{rem}

To present our last result,  we need the following assumption on the  global asymptotic stability of a positive fixed point.
\begin{enumerate}
\item [{\bf(GAS)}] There exists $W\in C_+\setminus\{0\}$ such that $\lim\limits_{n\to \infty}Q^n[\varphi]= W$ in $C$  for all $\varphi\in C_{+}\setminus\{0\}$.
\end{enumerate}

\begin{thm} \label{thm3.1-bil-gas}
Assume that $Q$ satisfies {\bf (GAS)} and $\varphi\in  C_{+}\setminus\{0\}$.  If {\bf (ACH)}, {\bf (NM)}, {\bf (SP)},   {\bf (ACH$_-$)}, {\bf (NM$_-$)}, and {\bf (SP$_-$)} hold,   and  $$\min\{c_-^*,c_+^*,c_-^*(-\infty),c_+^*(-\infty)\}>0, $$
 then
$$
\lim\limits_{n\rightarrow \infty}
\sup\left\{||Q^n[\varphi](\cdot, x)-W(\cdot,x)||:n(-c_-^*(-\infty)+\varepsilon)\leq x\leq n(c_+^*-\varepsilon)\right\}= 0
$$
for all $\varepsilon\in (0,\min\{c_+^*,c_-^*(-\infty)\})$.
\end{thm}

\begin{proof}
Given  $\gamma>0$ and $\varepsilon\in (0,\min\{c_+^*,c_-^*(-\infty)\})$.  In view of Theorem~\ref{thm3.1-bil-uc/uc}-(ii),  there exists $\alpha_0>0$ such that
$||Q^n[\varphi](\cdot, x)-r^*\cdot {\bf {1}}_{\mathbb{R}_+}(x)-r^*_-(1-{\bf {1}}_{\mathbb{R}_+}(x))||<\frac{\gamma}{2}$ and $||W(\cdot, x)-r^*\cdot {\bf {1}}_{\mathbb{R}_+}(x)-r^*_-(1-{\bf {1}}_{\mathbb{R}_+}(x))|<\frac{\gamma}{2}$ for all
$(n,x)\in \mathcal{C}_{\alpha_0,\varepsilon}$. Thus, $||Q^n[\varphi](\cdot, x)-W(\cdot, x)||<\gamma$ for all
$(n,x)\in \mathcal{C}_{\alpha_0,\varepsilon}$. It follows from {\bf (GAS)} that there exists $N_0>0$ such that $||Q^n[\varphi](\cdot, x)-W(\cdot, x)||<\gamma$ for all $x\in [-\alpha_0,\alpha_0]$ and $n>N_0$. This, together with the choices of $\mathcal{C}_{\alpha_0,\varepsilon}$, $\alpha_0$, and $N_0$, implies that $||Q^n[\varphi](\cdot, x)-W(\cdot, x)||<\gamma$ for all $x\in [n(-c_-^*(-\infty)+\varepsilon), n(c_+^*-\varepsilon)]$ and $n>\max\{N_0,\alpha_0\}$.  Now the arbitrariness of $\gamma$ completes the proof.
 \end{proof}

%{\color{blue} An example for $C_{++}$ ?}
 %Let $C_{++}$ be a given subset of $C_+\setminus \{0\}$  with $C_{++}+C_+\subseteq C_{++}$ and $\{\varphi\in C_+:\varphi(\cdot,x)\in Int(Y_+) \mbox{ for some } x\in \mathbb{R}\}\subseteq C_{++}$.  Let  $Q:  C_+\to C_+$ be a given continuous map with $\mathfrak{q}^*\in \mathbb{N}$ and %$Q^{\mathfrak{q}^*}[C_+\setminus C_{++}]=\{0\}$.
 %, and $Q[C_{++}]\subseteq C_{++}$.

%============================================
\section{Continuous-time systems}
In this section, under the bilateral {\bf (UC)} assumption we  study the existence of equilibrium points or traveling wave solutions and the asymptotic behavior  of other solutions for continuous-time autonomous and nonautonomous systems.  The unilateral  {\bf (UC)} case can be
dealt with 	in a similar  way as  in Section~\ref{sec3.1}.
\subsection{Autonomous  semiflows}

A map $Q:\mathbb{R}_{+}\times{C}_+\to{C}_+$ is said to be a
continuous-time semiflow on $C_+$ if  for any given vector $r\in Int(\mathbb{R}^N_+)$,
$Q|_{\mathbb{R}_{+}\times C_r}: \mathbb{R}_{+}\times{C}_r\to{C}_+$
is continuous,  $Q_0=Id|_{{C}_+}$,  and $Q_{t}\circ
Q_s=Q_{t+s}$ for all $t,s\in\mathbb{R}_+$, where $Q_t\triangleq
Q(t,\cdot)$ for all $t\in\mathbb{R}_+$.

Let $\{Q_t\}^{\infty}_{t=0}$ be a continuous-time semiflow on $C_+$.
We say that $W$ is a  nontrivial  equilibrium point of $\{Q_t\}^{\infty}_{t=0}$ if $W\in C_+\setminus\{0\}$  and
$Q_t[W]=W$ for all $t\in \mathbb{R}_+$. %,  and  that $W$  connects $\eta^{*}$ to $\xi^*$ if $W(\cdot,-\infty)=\eta^{*}$ and $W(\cdot,\infty)=\xi^*$, where $\eta^{*},\xi^{*}\in Y_{+}$.

 We start with  the upward convergence under the bilateral {\bf(UC)}/{\bf(UC)}-type assumption.

 %By applying Theorem~\ref{thm3.1-bil-uc/uc} to  $\{\varphi,\{Q_t\}_{t\in \mathbb{R}_+}\}$ and $\{\mathcal{S}[\varphi],\{\mathcal{S}\circ Q_t\circ \mathcal{S}\}_{t\in \mathbb{R}_+}\}$, we are able to prove the following result.

\begin{thm} \label{thm4.1-bil-uc/uc} Assume that $\{Q_t\}_{t\geq 0}$ is a continuous-time semiflow on $C_+$ such that  $Q_{t_0}$ satisfies all the assumptions of Theorem~\ref{thm3.1-bil-uc/uc} for some $t_0>0$.
%, where $Q_+$ are defined as in Lemma~\ref{lemm2.1}-(i) with $Q$ replaced by $Q_{t_0}$.
Let $\varphi\in C_{+}\setminus\{0\}$. Then the following statements are valid:
\begin{itemize}
\item [{\rm (i)}]
If $\min\{c_+^*,c_-^*(-\infty)\}>0$, then $$\lim\limits_{t\rightarrow \infty}
\max\left\{||Q_t[\varphi](\cdot, x)-r^*\cdot {\bf {1}}_{\mathbb{R}_+}(x)-r^*_-\cdot (1-{\bf {1}}_{\mathbb{R}_+}(x))||:x\in\mathcal{B}_{t,\varepsilon}\right\}= 0$$
for all  $\varepsilon\in (0,\frac{1}{2t_0}\min\{c_+^*,c_-^*(-\infty),c_+^*+c_-^*,c_+^*(-\infty)+c_-^*(-\infty)\})$, where  $$\mathcal{B}_{t,\varepsilon}=[t\max\{\varepsilon,-\frac{c_-^*}{t_0}+\varepsilon\}, t(\frac{c_+^*}{t_0}-\varepsilon)]\bigcup [t(-\frac{c_-^*(-\infty)}{t_0}+\varepsilon),t\min\{-\varepsilon,\frac{c_+^*(-\infty)}{t_0}-\varepsilon\}].$$

\item [{\rm (ii)}] If $\min\{c_-^*,c_+^*,c_-^*(-\infty),c_+^*(-\infty)\}>0$, then $$\lim\limits_{\alpha \rightarrow \infty}
\sup\left\{||Q_t[\varphi](\cdot, x)-r^*\cdot {\bf {1}}_{\mathbb{R}_+}(x)-r^*_-(1-{\bf {1}}_{\mathbb{R}_+}(x))||:(t,x)\in \mathcal{C}_{\alpha,\varepsilon}\right\}= 0$$ for all $\varepsilon\in (0,\frac{1}{t_0}\min\{c_+^*,c_-^*(-\infty)\})$, where  $$\mathcal{C}_{\alpha,\varepsilon}=\{(t,x)\in \mathbb{R}_+\times \mathbb{R}:t\geq \alpha t_0 \mbox{ and } x \in [\alpha,  t(\frac{c_+^*}{t_0}-\varepsilon)] \bigcup [t(-\frac{c_-^*(-\infty)}{t_0}+\varepsilon),-\alpha]\}.$$
\end{itemize}
\end{thm}

\begin{proof}
Since  $\{Q_t\}_{t\geq 0}$  is an autonomous semiflow, we assume that  $t_0=1$ in our proof. Otherwise, we consider the autonomous semiflow $\{\hat{Q}_t\}_{t\geq 0}:=\{Q_{t_0t}\}_{t\geq 0}$ instead of
$\{Q_t\}_{t\geq 0}$.
Let $\varphi\in C_{+}\setminus\{0\}$ and $K=Q[[0,1]\times \{\varphi\}]$. Then $K\subseteq C_+\setminus \{0\}$ and  $K$ is compact in $C$.

(i) Given any $\varepsilon\in (0,\frac{1}{2}\min\{c_+^*,c_-^*(-\infty),c_+^*+c_-^*,c_+^*(-\infty)+c_-^*(-\infty)\})$ and $\gamma>0$.
By applying Theorem~\ref{thm3.1-0}-(i) to $Q_1$ with $K$, we have
\begin{equation*}\label{4.1}
\lim\limits_{n\rightarrow \infty}
\max \left \{||Q_n[Q_t[\varphi]](\cdot, x)-r^*||:\, x\in  \left [n\max\{\frac{\varepsilon}{3},-c_-^*+\frac{\varepsilon}{3}\}, n(c_+^*-\frac{\varepsilon}{3})\right ], \, t\in [0,1]\right \}= 0.
\end{equation*}
This implies that there is $N_0\in \mathbb{N}$
such that $||Q_{t+n}[\varphi](\cdot,x)-r^*||<\gamma$ for all $n>N_0$, $x\in  \left [n\max\{\frac{\varepsilon}{3},-c_-^*+\frac{\varepsilon}{3}\}, n(c_+^*-\frac{\varepsilon}{3})\right ]$, and $t\in [0,1]$. As a result,
$||Q_{t}[\varphi](\cdot,x)-r^*||<\gamma$ for all $t>1+N_0+\frac{3c_+^*}{\varepsilon}
$ and $x\in   [t\max\{\varepsilon,-c_-^*+\varepsilon\}, t(c_+^*-\varepsilon) ]$.
By our aforementioned discussions, as applied   to  $\{\mathcal{S}[\varphi],\{\mathcal{S}\circ Q_t\circ \mathcal{S}\}_{t\in \mathbb{R}_+}\}$, there is $N_1\in \mathbb{N}$
such that
$||Q_{t}[\varphi](\cdot,x)-r^*_-||<\gamma$ for all $t>1+N_1+\frac{3c_-^*(-\infty)}{\varepsilon}
$ and $x\in   [t(-c_-^*(-\infty)+\varepsilon ),t\min\{-\varepsilon,c_+^*(-\infty)-\varepsilon\}]$. Hence, (i) follows from the arbitrariness of $\gamma$.

(ii) Given $\varepsilon\in (0,\min\{c_+^*,c_-^*(-\infty)\})$ and $\gamma>0$.
Applying Theorem~\ref{thm3.1}-(i) to $Q_1$ with $K$ and $c=0$, we have
\begin{equation*}\label{4.2}
\lim\limits_{\alpha\rightarrow \infty}
\sup \left \{||Q_n[Q_t[\varphi]](\cdot, x)-r^*||:n\geq \alpha, x\in  \left [\alpha, n(c_+^*-\frac{\varepsilon}{3})\right ],  \,  t\in [0,1]\right \}= 0.
\end{equation*}
Thus,  there is $\alpha_0>0$
such that $||Q_{t+n}[\varphi](\cdot,x)-r^*||<\gamma$ for all $n\geq \alpha_0$, $x\in  \left [\alpha_0, n(c_+^*-\frac{\varepsilon}{3})\right ]$, and $t\in [0,1]$. This implies that
$||Q_{t}[\varphi](\cdot,x)-r^*||<\gamma$ for all $t>1+\alpha_0+\frac{3c_+^*}{\varepsilon}
$ and $x\in  \left [1+\alpha_0+\frac{3c_+^*}{\varepsilon}, t(c_+^*-\varepsilon)\right ]$. In view of  our  aforementioned
discussions, as applied to  $\{\mathcal{S}[\varphi],\{\mathcal{S}\circ Q_t\circ \mathcal{S}\}_{t\in \mathbb{R}_+}\}$, it follows  that there is $\alpha_1\in \mathbb{N}$ such that
$||Q_{t}[\varphi](\cdot,x)-r^*_-||<\gamma$ for all $t>1+\alpha_1+\frac{3c_-^*(-\infty)}{\varepsilon}
$ and $x\in  [t(-c_-^*(-\infty)+\varepsilon),-1-\alpha_1-\frac{3c_-^*(-\infty)}{\varepsilon}]$. Hence, (ii) follows from the arbitrariness of $\gamma$.
\end{proof}

%%%%%%%%%

%%%%%%%%%?????????

We need the following uniform continuity,  as introduced in \cite{yz2020},
to prove the  asymptotic annihilation for $Q_t$   under the bilateral {\bf (GLC)} /{\bf (GLC)} -type assumption.

\begin{enumerate}
\item [{\bf (SC)}]
For any $s>0$ and $\phi\in Int(Y_+)$, there holds
$\lim\limits_{C_{\phi}\ni\varphi\to 0}T_{-y}\circ Q_t \circ T_{y}[\varphi](\cdot,0)=0$ in $Y$ uniformly for all $(t,y)\in [0,s] \times \mathbb{R}$.

%For any $t_0>0$ and $\phi_k\in C_+$ with $\lim\limits_{k\to \infty}\phi_k=0$ and $\sup\limits_{k\in \mathbb{N}}||\phi_k||_{L^\infty(M\times \mathbb{R}, \mathbb{R})}<\infty$, there holds $\lim\limits_{k\to \infty}T_{-y}\circ Q_t \circ T_{y}[\phi_k]=0$ in $C$ uniformly for $ (t,y)\in [0,t_0] \times \mathbb{R}$.
\end{enumerate}

\begin{thm} \label{thm4.1-bil-meal/meal} Assume that $\{Q_t\}_{t\geq 0}$ is a continuous-time semiflow on $C_+$ such that {\bf (SC)} holds and $Q_{t_0}$ satisfies all the assumptions of Theorem~\ref{thm3.1-bil-meal/meal} for some $t_0>0$.
Let  $\varphi\in { C_{+}}$ with the compact support,
$c^*_1:=\liminf\limits_{k\to \infty}c^*_{L_k}$, and
$c^*_2:=\liminf\limits_{k\to \infty}c^*_{\mathcal{S}\circ\hat{L}_k\circ \mathcal{S}}$. Then the following statements are valid:
\begin{itemize}
\item [{\rm (i)}] If $c^*_1\geq 0$ and  $c^*_2\geq 0$, then $$\lim\limits_{t\rightarrow \infty}
\sup\left\{
||Q_t[\varphi](\cdot,x)||: x\geq  t\left(\frac{c^*_1}{t_0}+\varepsilon\right) \mbox{ or } x\leq - t\left(\frac {c^*_2}{t_0}+\varepsilon\right)\right\}=0, \forall \varepsilon>0.$$

\item [{\rm (ii)}] If $c^*_1< 0$ and  $c^*_2\geq 0$, then  $$\lim\limits_{t\rightarrow \infty}\sup\left\{
||Q_t[\varphi](\cdot,x)||:x\geq t\varepsilon \mbox{ or } x\leq - t\left(\frac{c^*_2}{t_0}+\varepsilon\right)\right\}=0, \forall \varepsilon>0.$$

\item [{\rm (iii)}] If $c^*_1\geq 0$ and  $c^*_2< 0$, then  $$\lim\limits_{t\rightarrow \infty}\sup\left\{
||Q_t[\varphi](\cdot,x)||:x\leq -t\varepsilon \mbox{ or } x\geq  t\left(\frac{c^*_1}{t_0}+\varepsilon\right)\right\}=0, \forall \varepsilon>0.$$

\item [{\rm (iv)}] If $c^*_1< 0$  and   $c^*_2< 0$, then  $\lim\limits_{\alpha\rightarrow \infty}\sup\left\{||Q_t[\varphi](\cdot,x)||:|x|,t\geq \alpha \right\}=0$ and $$\lim\limits_{t\rightarrow \infty}\sup\left\{
||Q_t[\varphi](\cdot,x)||:|x|\geq t\varepsilon \right\}=0, \forall \varepsilon>0.$$

%$\lim\limits_{\sigma\rightarrow \infty}\Big[\sup\{||Q^n[\varphi](\cdot,x)||:x,n\geq \sigma \mbox{ with } n \in \mathbb{N} \}\Big]=0$
\end{itemize}
\end{thm}

\begin{proof}  We only prove (i) since we can use similar arguments to deal with (ii-iv). Since $\{Q_t\}_{t\geq 0}$  is an autonomous semiflow, we assume that  $t_0=1$ without loss of generality. Fix $\varepsilon>0$, $\gamma>0$, $k_0\in \mathbb{N}$, and $\varphi\in C_{\phi_{k_0}^*}$ with  $\varphi$ having the compact support.

It follows from  {\bf (SC)} that there
exist $\delta=\delta(\gamma)>0$ and $d=d(\gamma)>0$ such
that if $\psi\in C_{\phi_{k_0}^*}$ with $||\psi(\cdot,x)||<\delta$ for all
$x\in [-d,d] $,  then
$||T_{-y}\circ Q_{t}\circ T_{y}[\psi](\cdot,0)||<\gamma$ for all $t\in[0,1]$ and $y\in \mathbb{R}$.

By virtue of  Theorem~\ref{thm3.1-bil-meal/meal}-(i),  we obtain
\begin{equation}\label{4.3-2}
\lim\limits_{n\rightarrow \infty}
\sup\left\{
||Q_n[\varphi](\cdot,x)||:x\geq n(c^*_1+\frac{\varepsilon}{3}) \mbox{ or } x\leq -n(c^*_2+\frac{\varepsilon}{3})\right \}=0.
\end{equation}
It then follows from~\eqref{4.3-2} that there is an integer $n_1>0$
such that
$$
||Q_{n}[\varphi] (\cdot,x)||<\delta,  \,  \,  \forall
x\in [ n(c^*_1+\frac{\varepsilon}{3}), \infty)\bigcup (-\infty, -n(c^*_2+\frac{\varepsilon}{3})], \, n>n_1.
$$
Let $n_2=\max\{n_1,\frac{3d}{\varepsilon}\}$. Then for any $n>n_2$ and  $y\in  [n(c^*_1+\frac{2\varepsilon}{3}),\infty)\bigcup (-\infty, -n(c^*_2+\frac{2\varepsilon}{3})]$, we have
$$
||T_{-y}\circ Q_{n}[\varphi](\cdot,x)||<\delta, \,  \,  \forall x\in [-d,d].
$$
According to these discussions,  we know that
\begin{eqnarray*}
||Q_{t+n}[\varphi](\cdot,y)||&=&||T_{-y} \circ Q_{t}\circ T_{y} \circ T_{-y} \circ Q_{n}[\varphi](\cdot,0)||\\
&=&||T_{-y} \circ Q_{t}\circ T_{y} [ T_{-y} \circ Q_{n}[\varphi]](
\cdot,0)||<\gamma,
\end{eqnarray*}
where  $n>n_2$,   $y\in  [n(c^*_1+\frac{2\varepsilon}{3}),\infty)\bigcup (-\infty, -n(c^*_2+\frac{2\varepsilon}{3})]$, and $t\in [0,1]$.
 In particular,
$||Q_{t}[\varphi](\cdot,y)||<\gamma$ for all $t>1+n_2
$ and $y\in  [ t(c^*_1+\varepsilon),\infty)\bigcup (-\infty, - t(c^*_2+\varepsilon)]$.
Thus, (i) follows from the arbitrariness of $\gamma$.
\end{proof}

The following  result is on the existence of equilibrium points  under  the bilateral {\bf(UC)}/{\bf(UC)}-type assumption.

\begin{thm} \label{thm4.3.6-bil-uc/uc-fixp}
Suppose that  $t_0>0$, $r^{**}\in \underline {{\phi}_1^* } +Y_{+}$, and  ${Q}$ and $\underline{Q}$ are continuous-time semiflows on $C_+$ such that $Q_t[C_{r^{**}}]\subseteq C_{r^{**}}$, $Q_t[\varphi]\geq \underline{Q}_t[\varphi]\geq \underline{Q}_t[\psi]$ for all $(t,\varphi,\psi)\in \mathbb{R}_+\times C_{r^{**}}\times C_{r^{**}} $ with $\varphi\geq \psi$, and $\underline{Q}_{t_0}$ satisfies all conditions in Theorem~\ref{thm3.6-bil-uc/uc-fixp} with  parameters $r,c$ and $\phi$ being underlined. If
%$Q_{t_1}$ satisfies {\bf(NM)} and {\bf(NM$_-$)} for some $t_1>0$, and
$\{Q_t\}_{t\in\mathbb{R}_+}$ has at least one equilibrium point in $\mathcal{K}$ for any nonempty, closed, convex and positively invariant set $\mathcal{K}\subseteq C_{r^{**}}$ of $Q_t$,  then  $\{Q_t\}_{t\in\mathbb{R}_+}$ has an  equilibrium point  $W$ in $C_{r^{**}}$ with $W(\cdot,\infty)=r^*$ and $W(\cdot,-\infty)=r^*_-$.
\end{thm}

\begin{proof} By applying Theorem~\ref{thm3.6-bil-uc/uc-fixp} to  $\underline{Q}_{t_0}$, we know that   $\underline{Q}_{t_0}$ has a  nontrivial  fixed point $\underline{W}$ in $C_{r^{**}}$ such that $\underline{W}(\cdot,\infty)=\underline{r}^*$ and $\underline{W}(\cdot,-\infty)=\underline{r}^*_-$.  Let $\mathcal{K}=\bigcap\limits_{t\in [0,t_0]} [\underline{Q}_{t}[\underline{W}],r^{**}]_C$. Then $\mathcal{K}$ is a nonempty, closed, convex and positively invariant set of $Q_t$, and hence,  $\{Q_t\}_{t\in\mathbb{R}_+}$ has an  equilibrium point  $W$ in $\mathcal{K}$ with $\liminf\limits_{x\to \infty}W(\cdot,x)\geq\underline{r}^*$ and $\liminf\limits_{x\to-\infty}W(\cdot,x)\geq\underline{r}^*_-$.
By Lemma~\ref{lemm3.1-999000}, as applied  to  $\{(Q_{t_0},W)\}$ and $\{(\mathcal{S}\circ Q_{t_0} \circ \mathcal{S},\mathcal{S}[W])\}$, we have $W(\cdot,\infty)=r^*$ and $W(\cdot,-\infty)=r^*_-$.\end{proof}

\begin{rem}  \label{rem4.1-sp-monoty}
By Remark~\ref{rem3.2-sp-monoty} and the proof of Theorems~\ref{thm4.3.6-bil-uc/uc-fixp}, it follows  that
in Theorem~\ref{thm4.3.6-bil-uc/uc-fixp}, we can use
$\min\{c_+^*,c_-^*,c_+^*(-\infty),c_-^*(-\infty)\}>0$ to
replace  the assumption that  $T_{-y}\circ Q_l[\varphi]\geq Q_l\circ T_{-y}[\varphi]$ for all $l\in \mathbb{Z}\setminus\{0\}$ and $(y,\varphi)\in \mathbb{R}_+\times C_{r^*_l}$.
\end{rem}

To present our last result in this subsection, we need the following assumption.

\

 {\bf(GAS-CSF)} There exists $W\in C_+ $ such that $\lim\limits_{t\to \infty }Q_t[\varphi]=W$ in $C$ for all $\varphi\in C_+\setminus \{0\}$.

\begin{thm} \label{thm4.1-bil-gas}
Assume that $\{Q_t\}_{t\geq 0}$ is a continuous-time semiflow on $C_+$ such that $\{Q_t\}_{t\geq 0}$ satisfies {\bf (GAS-CSF)}. Let $\varphi\in C_{+}\setminus\{0\}$ and $t_0\in (0,\infty)$.
If $Q_{t_0}$ satisfies all the assumptions of Theorem~\ref{thm3.1-bil-uc/uc} with $Q$ replaced by $Q_{t_0}$, then $$\lim\limits_{t \rightarrow \infty}
\sup\left\{||Q_t[\varphi](\cdot, x)-W(\cdot,x)||:t(-\frac{c_-^*(-\infty)}{t_0}+\varepsilon)\leq x\leq t(\frac{c_+^*}{t_0}-\varepsilon)\right\}= 0$$ for all  $\varepsilon\in (0,\frac{1}{t_0}\min\{c_+^*,c_-^*(-\infty)\})$.
\end{thm}

\begin{proof}
Given  $\gamma>0$ and $\varepsilon\in (0,\infty)$.  Applying Theorem~\ref{thm4.1-bil-uc/uc}-(ii)  to $Q_{t_{0}}$,  we see that there exists $\alpha_0>0$ such that
$$||Q_t[\varphi](\cdot, x)-r^*||<\frac{\gamma}{2}, ||W(\cdot, x)-r^*||<\frac{\gamma}{2}, \forall
t\geq \alpha_0 t_0,  x \in [\alpha_0,  t(\frac{c_+^*}{t_0}-\varepsilon)] $$ and $$||Q_t[\varphi](\cdot, x)-r^*_-||<\frac{\gamma}{3}, ||W(\cdot, x)-r^*_-||<\frac{\gamma}{2}, \forall
t\geq \alpha_0 t_0, x \in [t(-\frac{c_-^*(-\infty)}{t_0}+\varepsilon),-\alpha_0].$$ Thus, we have
$$
||Q_t[\varphi](\cdot, x)-W(\cdot, x)||<\gamma,\, \,  \forall
t\geq \alpha_0 t_0, x \in [\alpha_0,  t(\frac{c_+^*}{t_0}-\varepsilon)] \bigcup [t(-\frac{c_-^*(-\infty)}{t_0}+\varepsilon),-\alpha_0].$$
 It follows from {\bf (GAS-CSF)} that there exists $T_0>0$ such that
 $$
 ||Q_t[\varphi](\cdot, x)-W(\cdot, x)||<\gamma, \, \,  \forall x\in [-\alpha_0,\alpha_0], \,  t>T_0.
 $$
 These, together with the choices of $\alpha_0$ and $T_0$, imply that $||Q_t[\varphi](\cdot, x)-W(\cdot, x)||<\gamma$ for all  $t>\max\{T_0,\alpha_0 t_0\}$ and $x\in [t(-\frac{c_-^*(-\infty)}{t_0}+\varepsilon), t(\frac{c_+^*}{t_0}-\varepsilon)]$. Now  the arbitrariness of $\gamma$ gives rise to the desired result. \end{proof}

\subsection{Nonautonomous systems}
Assume that
 $P:\mathbb{R}_{+}\times  {C}_+\to{C}_+$ is a map such that for any given $r\in Int(\mathbb{R}^N_+)$,
$P|_{\mathbb{R}_{+}\times C_r}: \mathbb{R}_{+}\times {C}_r\to{C}_+$
is continuous.  For any given $c\in \mathbb{R}$, we define a family of
mappings $Q_t:=T_{-ct}\circ P[t,\cdot]$  with parameter
$t\in \mathbb{R}_+$.

We say that $W(\cdot,x-ct)$ is a travelling wave of $P$ if
$W:M\times\mathbb{R}\to \mathbb{R}_+$ is a bounded and
nonconstant continuous function such that
$P[t,W](\theta,x)=W(\theta,x-tc)$ for all $(\theta,x)\in
M\times \mathbb{R}$ and $t\in \mathbb{R}_+$. %, and that$W$ connects $\eta^{*}$ to $\xi^*$ if $W(\cdot,-\infty)=\eta^{*}$ and$W(\cdot,\infty)=\xi^*$, where $\eta^{*}, \xi^{*}\in Y_{+}$.
By applying Theorem~\ref{thm4.1-bil-uc/uc} to  $\{\varphi,\{Q_t\}_{t\in \mathbb{R}_+}\}$, we have  the following result on the  upward convergence under the bilateral {\bf(UC)}/{\bf(UC)}-type assumption.

\begin{thm} \label{thm5.1-bil-uc/uc} Assume  that there exist $t_0>0$ and  $c\in \mathbb{R}$ such that $Q_t:=T_{-ct}\circ P[t,\cdot]$ is a continuous-time semiflow on $C_+$
 and  $Q_{t_{0}}$  satisfies all the conditions in Theorem~\ref{thm3.1-bil-uc/uc}.
Let $\varphi\in  C_{+}\setminus\{0\}$. Then the following statements are valid:
\begin{itemize}
\item [{\rm (i)}]
If $\min\{c_+^*,c_-^*(-\infty)\}>0$, then $$\lim\limits_{t\rightarrow \infty}
\max\left\{||P[t,\varphi](\cdot, x)-r^*\cdot {\bf {1}}_{\mathbb{R}_+}(x-ct)-r^*_-\cdot (1-{\bf {1}}_{\mathbb{R}_+}(x-ct))||:x\in\mathcal{B}_{t,\varepsilon,c}\right\}= 0$$
for all  $\varepsilon\in (0,\frac{1}{2t_0}\min\{c_+^*,c_-^*(-\infty),c_+^*+c_-^*,c_+^*(-\infty)+c_-^*(-\infty)\})$, where   $$\mathcal{B}_{t,\varepsilon,c}=[t\max\{\varepsilon+c,c-\frac{c_-^*}{t_0}+\varepsilon\}, t(c+\frac{c_+^*}{t_0}-\varepsilon)]\bigcup [t(c-\frac{c_-^*(-\infty)}{t_0}+\varepsilon),t\min\{c-\varepsilon,c+\frac{c_+^*(-\infty)}{t_0}-\varepsilon\}].$$

\item [{\rm (ii)}] If $\min\{c_-^*,c_+^*,c_-^*(-\infty),c_+^*(-\infty)\}>0$, then $$\lim\limits_{\alpha \rightarrow \infty}
\sup\left\{||P[t,\varphi](\cdot, x)-r^*\cdot {\bf {1}}_{\mathbb{R}_+}(x-ct)-r^*_-(1-{\bf {1}}_{\mathbb{R}_+}(x-ct))||:(t,x)\in \mathcal{C}_{\alpha,\varepsilon,c}\right\}= 0$$
for all $\varepsilon\in (0,\frac{1}{t_0}\min\{c_+^*,c_-^*(-\infty)\})$, where $$\mathcal{C}_{\alpha,\varepsilon,c}=\{(t,x)\in \mathbb{R}_+\times \mathbb{R}:t\geq \alpha t_0 \mbox{ and } x \in [\alpha+ct,  t(c+\frac{c_+^*}{t_0}-\varepsilon)] \bigcup [t(c-\frac{c_-^*(-\infty)}{t_0}+\varepsilon),ct-\alpha]\}.
$$
\end{itemize}
\end{thm}

Under the bilateral {\bf (GLC)} /{\bf (GLC)} -type assumption, we are
able to apply Theorem~\ref{thm4.1-bil-meal/meal} to $\{\varphi,\{Q_t\}_{t\in \mathbb{R}_+}\}$ to obtain the following result.

\begin{thm} \label{thm5.1-bil-meal/meal} Assume  that there exists   $c\in \mathbb{R}$ such that $Q_t:=T_{-ct}\circ P[t,\cdot]$ is a continuous-time semiflow on $C_+$, {\bf (SC)} holds, and $Q_{t_0}$ satisfies all the assumptions of Theorem~\ref{thm3.1-bil-meal/meal} for some $t_0>0$.
Let  $\varphi\in { C_{+}}$ with having the compact support,
$c^*_1:=\liminf\limits_{k\to \infty}c^*_{L_k}$, and
$c^*_2:=\liminf\limits_{k\to \infty}c^*_{\mathcal{S}\circ\hat{L}_k\circ \mathcal{S}}$. Then the following statements are valid:
\begin{itemize}
\item [{\rm (i)}] If $c^*_1\geq 0$ and  $c^*_2\geq 0$, then $$\lim\limits_{t\rightarrow \infty}
\sup\left\{
||P[t,\varphi](\cdot,x)||: x\geq  t(c+\frac{c^*_1}{t_0}+\varepsilon) \mbox{ or } x\leq - t(\frac{c^*_2}{t_0}-c+\varepsilon)\right\}=0, \forall \varepsilon>0.$$

\item [{\rm (ii)}] If $c^*_1< 0$ and  $c^*_2\geq 0$, then  $$\lim\limits_{t\rightarrow \infty}\sup\left\{
||P[t,\varphi](\cdot,x)||:x\geq t(c+\varepsilon) \mbox{ or } x\leq - t(\frac{c^*_2}{t_0}-c+\varepsilon)\right\}=0, \forall \varepsilon>0.$$

\item [{\rm (iii)}] If $c^*_1\geq 0$ and  $c^*_2< 0$, then  $$\lim\limits_{t\rightarrow \infty}\sup\left\{
||P[t,\varphi](\cdot,x)||: x\geq  t(c+\frac{c^*_1}{t_0}+\varepsilon) \mbox{ or } x\leq t(c-\varepsilon)\right\}=0, \forall \varepsilon>0.$$

\item [{\rm (iv)}] If $c^*_1< 0$ and  $c^*_2< 0$, then  $\lim\limits_{\alpha\rightarrow \infty}\sup\left\{||P[t,\varphi](\cdot,x)||:|x-ct|\geq \alpha \mbox{ and }t\geq \alpha\right \}=0$ and $$\lim\limits_{t\rightarrow \infty}\sup\left\{
||P[t,\varphi](\cdot,x)||:|x-ct|\geq t\varepsilon \right\}=0, \forall \varepsilon>0.$$

%$\lim\limits_{\sigma\rightarrow \infty}\Big[\sup\{||Q^n[\varphi](\cdot,x)||:x,n\geq \sigma \mbox{ with } n \in \mathbb{N} \}\Big]=0$
\end{itemize}
\end{thm}

As an application of Theorems \ref{thm4.3.6-bil-uc/uc-fixp},
we have the following  result on the existence of travelling waves under the bilateral {\bf(UC)}/{\bf(UC)}-type assumption.

\begin{thm} \label{thm5.3.6-bil-uc/uc-fixp}%$t_0,t_1>0$,
Suppose that $c\in \mathbb{R}$, and $Q_t:=T_{-ct}\circ P[t,\cdot]$, $\underline{Q}_t$ are continuous-time semiflows on $C_+$ such that $\{Q_t,\underline{Q}_{t}\}$ satisfies all conditions in Theorem~\ref{thm4.3.6-bil-uc/uc-fixp}. Then  $\{P[t,\cdot]\}_{t\in\mathbb{R}_+}$ has a travelling wave  $W(\cdot,x-ct)$ in $C_{r^{**}}$ with $W(\cdot,\infty)=r^*$ and $W(\cdot,-\infty)=r^*_-$.
\end{thm}

Regarding the global attracticity of the positive travelling wave, we can use
Theorem~\ref{thm4.1-bil-gas} to establish  the following result.
\begin{thm} \label{thm5.1-bil-gas}
Assume that  there exists   $c\in \mathbb{R}$ such that $Q_t:=T_{-ct}\circ P[t,\cdot]$  is a continuous-time semiflow on $C_+$ and  $\{Q_t\}_{t\geq 0}$ satisfies {\bf (GAS-CSF)}. Let $\varphi\in C_{+}\setminus\{0\}$. If there exists $t_0> 0$ such that
	$Q_{t_0}$ satisfies all the assumptions of Theorem~\ref{thm3.1-bil-uc/uc} with $Q$ replaced by $Q_{t_0}$, then $$\lim\limits_{t \rightarrow \infty}
\sup\left\{||P[t,\varphi](\cdot, x)-W(\cdot,x-ct)||: t\left(c-\frac{c_-^*(-\infty)}{t_0}+\varepsilon\right)\leq x\leq t\left(c+\frac{c_+^*}{t_0}-\varepsilon\right)\right\}= 0,$$
for all  $\varepsilon\in (0,\frac{1}{t_0}\min\{c_+^*,c_-^*(-\infty)\})$.
\end{thm}

Following  \cite[Section 3.1]{Zhaobook}, we say a map $Q:\mathbb{R}_{+}\times{C}_+\to{C}_+$ is a
 continuous-time  $\omega$-periodic semiflow on $C_+$ if  for any given  $r\in Int(\mathbb{R}_+^N)$,
 $Q|_{\mathbb{R}_{+}\times C_r}: \mathbb{R}_{+}\times{C}_r\to{C}_+$
 is continuous,  $Q_0=Id|_{{C}_+}$,  and $Q_{t}\circ
 Q_{\omega}=Q_{t+\omega}$ for  some number $\omega>0$ and all $t\in\mathbb{R}_+$,  where $Q_t:=
 Q(t,\cdot)$ for all $t\in\mathbb{R}_+$.

\begin{rem}
In the case where $Q_t:=T_{-ct}\circ P[t,\cdot]$ is a continuous-time
$\omega$-periodic semiflow on $C_+$, we can apply main results in Section~\ref{3sec} to the Poincar\'e map $Q_{\omega}$  to
establish the spreading properties and the forced time-periodic traveling waves with speed $c$ for the nonautonomous evolution system
$P[t,\cdot]$. We refer to\cite{LYZ} for the Poincar\'e map approach to
monotone periodic semiflows.
\end{rem}

%============================================
\section{Applications}

In this section, we apply our developed theory  in Sections 3 and  4 to two classes of evolution equations in the bilateral limit case, but we leave the investigation in the unilateral limit case to readers.

	\subsection {A  delayed  reaction-diffusion equation with a shifting habitat \label {Sec6.3}}
Consider the following time-delayed  reaction-diffusion equation with a shifting habitat:
\begin {equation}
\left\{
\begin{array}{ll}
	\frac{\partial u}{\partial t}(t,x)  =  d u_{xx}(t,x)-\mu u(t,x) +
	\mu f(x-ct,u(t-\tau,x)), \, (t,x)\in (0,\infty)\times \mathbb{R}, \\
	u(\theta,x) = \varphi(\theta,x),  \quad  (\theta,x) \in [-\tau,0]
	\times \mathbb{R},
\end{array} \right.
\label{6.8}
\end {equation}
where  $c\in \mathbb{R}$,
$d,\mu,\tau>0$, and the
initial data $\varphi$ belongs to $C([-\tau,0]\times \mathbb{R},\mathbb{R}_+)\cap L^\infty([-\tau,0]\times \mathbb{R},\mathbb{R})$. Note that the shifting speed $c$ in system \eqref{6.8} can be made positive via  the transformation: $-c\to c, -x\to x, u(t,-x)\to u(t,x)$ and $f(-x,u)\to f(x,u)$ when $c<0$.  Without loss of
generality, we may assume that  $c\geq 0$.

We always assume that $f\in C(\mathbb{R}\times\mathbb{R}_{+},\mathbb{R}_+)$ and $f_\pm^\infty\in C^1(\mathbb{R}_{+},\mathbb{R}_+)$ satisfy the following condition:
\begin{enumerate}
	\item[(B)] $f(s,\cdot)\in C^{1}(\mathbb{R}_{+},\mathbb{R}), f(s,0)=0, \forall
	s\in \mathbb{R}$;  $0\leq f_{\pm}^{\infty}(u)\leq \frac{{\rm d} f_{\pm}^{\infty}(0)}{\rm d u}u, \forall u\in \mathbb{R}_+$;
there exists a sequence $\{u_k^*\}_{k\in \mathbb{N}}$ in $(0,\infty)$ such that $u_k^*<u_{k+1}^*$, $\lim\limits_{k\to\infty}u_k^*=\infty$, and $f(\mathbb{R}\times [0,u_k^*])\subseteq  [0,u_k^*]$ for all $k\in \mathbb{N}$.
\end{enumerate}

To study  bilateral propagation dynamics,  we also need  the following  assumptions on $f$:
\begin{enumerate}
	\item[(B$_+$)]  $\mbox{There exists  }u^{*}_+ >0 \mbox{  such that } \{u>0:(f_{+}^{\infty})^2(u)=u\}=\{u_+^*\}$, $\frac{{\rm d} f_{+}^{\infty}(0)}{\rm d u}>1$, $f_{+}^{\infty}(0,\infty)\subseteq (0,\infty)$, {$\inf_{x>1}\frac{f_{+}^{\infty}(x)}{x}<1$,}  $f(\mathbb{R}\times (0,\infty))\subseteq (0,\infty)$,
	and $\lim\limits_{s\to \infty}f(s,\cdot)=f_{+}^{\infty}(\cdot)$
	 in  $C_{loc}^{1}(\mathbb{R}_{+},\mathbb{R})$;
	\item[(B$_-$)]  $\mbox{There exists  }u^{*}_- >0 \mbox{  such that } \{u>0:(f_{-}^{\infty})^2(u)=u\}=\{u_-^*\}$, $\frac{{\rm d} f_{-}^{\infty}(0)}{\rm d u}>1$,  $f_{-}^{\infty}(0,\infty)\subseteq (0,\infty)$, {$\inf_{x>1}\frac{f_{-}^{\infty}(x)}{x}<1$,}  $f(\mathbb{R}\times (0,\infty))\subseteq (0,\infty)$,
	and $\lim\limits_{s\to -\infty}f(s,\cdot)= f_{-}^{\infty}(\cdot)$  in   $C_{loc}^{1}(\mathbb{R}_{+},\mathbb{R})$.
	\end{enumerate}
Note that we do not
assume  that  $f(s,u)\leq \partial_u f(s,0)u$ for   all $(s,u)\in \mathbb{R}\times \mathbb{R}_+$.

The following result is adapted from  \cite[Lemma 4.2]{ycw2019} and \cite[Lemma 6.1]{yz2020}, which gives
a   monotone and subhomogeneous nonlinear lower bound for the function $f$.
\begin{lemma}\label{lem666.001} Let {\rm (B$_+$)} hold. Then for any  $u^{**}\geq u^{*}_+$ and $\gamma\in(0,\frac{{\rm d}f_+^{\infty}(0)}{{\rm d}u}-1)$ with $f(\mathbb{R}\times [0,u^{**}])\subseteq  [0,u^{**}]$, there exist $\mathfrak{s}=\mathfrak{s}(u^{**},\gamma)>0$, $K=K_{u^{**},\gamma}>0$, and $r=r_{u^{**},\gamma}\in C(\mathbb{R},\mathbb{R}_+)$ such that
	\begin{itemize}
		\item[{\rm (i)}]$r$ is a nondecreasing function having the property that $-\frac{1}{K}= r(-\infty)=r(s_1)< r(s_2)=r(\infty)=\frac{\frac{{\rm d}f_+^{\infty}(0)}{{\rm d}u}-1-\gamma}{K}$ for all  $(s_1,s_2)\in(-\infty,\mathfrak{s}] \times [1+\mathfrak{s},\infty)$;

		\item[{\rm (ii)}] $f(s,u)\geq f_{u^{**},\gamma}(s,u)$ for all $(s,u)\in \mathbb{R}\times[0,u^{**}]$,
		 where
		\[
		f_{u^{**},\gamma}(s,u)=\left\{
		\begin{array}{ll}
		\frac{(1+Kr(s))^2}{4K}, & (s,u)\in \mathbb{R}\times [\frac{1+Kr(s)}{2K},\infty),
		\\
		u+K u(r( s)-u), \qquad & (s,u)\in \mathbb{R}\times [0,\frac{1+Kr(s)}{2K}).
		\end {array}
		\right.
		\]

\item[{\rm (iii)}]$f_{u^{**},\gamma}(\infty,\cdot)$ has a unique fixed point $u_\infty$ in $(0,\infty)$ with $$u_\infty\in (0,u^{**}]\mbox{ and }
			u_\infty=\left\{
			\begin{array}{ll}
			\frac{(1+Kr(\infty))^2}{4K}, & r(\infty)K>1,
			\\
			r(\infty), \qquad & r(\infty)K\leq 1,
			\end {array}
			\right.
			$$ and hence, $u_\infty\leq u_1^*$.
				\end{itemize}
\end{lemma}

\begin{proof} Fix  $u^{**}\geq u^{*}_+$ and $\gamma\in(0,\frac{{\rm d}f_+^{\infty}(0)}{{\rm d}u}-1)$. By the choice of $\gamma$, there exists $\delta_1:=\delta_1(\gamma)>0$ such that $$\frac{{\rm d}f_+^{\infty}(u)}{{\rm d}u}>\frac{{\rm d}f_+^{\infty}(0)}{{\rm d}u}-\frac{\gamma}{3}, \, \, \forall u\in [0,\delta_1].$$ This, together with  (B$_+$), implies that there exists $\mathfrak{s}=\mathfrak{s}(u^{**},\gamma)>0$ such that
\[
f_*:=\inf f([\mathfrak{s},\infty)\times [\delta_1,u^{**}])>0,  \, \, \partial_uf(s,u)>\frac{{\rm d}f_+^{\infty}(0)}{{\rm d}u}-\frac{2\gamma}{3},\, \, \forall (s,u)\in [\mathfrak{s},\infty)\times [0,\delta_1].
\]
Take $K=\frac{1}{4f_*}[\frac{{\rm d}f_+^{\infty}(0)}{{\rm d}u}]^2$ and define $r:\mathbb{R}\to \mathbb{R}$ by
\[
r(s)=\left\{
\begin{array}{ll}
\frac{\frac{{\rm d}f_+^{\infty}(0)}{{\rm d}u}-1-\gamma}{K}, & s\geq \mathfrak{s}+1,
 \\
\frac{(\frac{{\rm d}f_+^{\infty}(0)}{{\rm d}u}-\gamma)(s-\mathfrak{s})-1}{K}, &\mathfrak{s}< s< \mathfrak{s}+1,
\\
-\frac{1}{K}, \qquad & s\leq  \mathfrak{s}.
\end {array}
\right.
\]
With  the definitions of $r(\cdot)$ and  $f_{u^{**},\gamma}$, we can easily verify  that
$r(\cdot)$ and  $f_{u^{**},\gamma}$ have all properties presented in Lemma~\ref{lem666.001}.
\end{proof}

The following observation is useful for us to verify the assumptions (GLC) and (GLC$_-$).

\begin{lemma}\label{lem666.005} Let {\rm (B$_\pm$)} hold. Then for any $u^{**}>0$ and $\gamma>0$, there exists $\overline{R}=\overline{R}_{\gamma,u^{**}}(\cdot)\in C(\mathbb{R},\mathbb{R}_+)$ such that
	\begin{itemize}
		\item[{\rm (i)}] $\overline{R}$ is a  nonincreasing function with $\infty>\overline{R}(-\infty)> \overline{R}(\infty)=\gamma+\frac{{\rm d}f_+^{\infty}(0)}{{\rm d}u}$;

		\item[{\rm (ii)}] $f(s,u)\leq \overline{R}(s)u$ for all $(s,u)\in \mathbb{R}\times[0,u^{**}]$.
	\end{itemize}
\end{lemma}

		\begin{lemma}\label{lem6.5} If $x^*\in(0,\infty)$ and $g\in C(\mathbb{R}_+,\mathbb{R}_+)$ such that $g(0)=0<1<g'(0)$, {$\inf_{x>1}\frac{g(x)}{x}<1$,}  and $\{x\in (0,\infty):g(g(x))=x\}=\{x^*\}$, then for any $b\geq a>0$ with $\{a,b\}\neq x^*$, there exist $b^*>a^*>\gamma^*>0$ such that $[a,b]\subseteq [a^*,b^*]$, $ g([a^*-\gamma^*,b^*+\gamma^*])\subseteq (a^*+\gamma^*,b^*-\gamma^*)$, and either $a+\gamma^*<\inf g([a^*-\gamma^*,b^*+\gamma^*])$ or $b-\gamma^*>\sup g([a^*-\gamma^*,b^*+\gamma^*])$.
	\end{lemma}
	
	\begin{proof}
		Note that $g(x)>x$ for all $x\in (0,x^*)$ and $g(x)<x$ for all $x\in (x^*,\infty)$. Fix $b\geq a>0$ with $\{a,b\}\neq \{x^*\}$. In view of \cite[Lemma 5.3]{ycw2013},  we see that there exist
		$a',b'\in (0,\infty)$ such that $[a,b]\subseteq [a',b']$,
		$g([a',b'])\subseteq [a',b']$, and either $a<\inf g([a',b'])$ or $b>\sup g([a',b'])$. Thus, $a'<\inf g([a',b'])$ or $b'>\sup g([a',b'])$. Next we  finish the proof according to  three cases.
		
		{\bf Case 1.} $a'<\inf g([a',b'])
			$ and $b'>\sup g([a',b'])$.
		
		In this case, we may select $\gamma^*\in (0,a')$ such that $$a'+\gamma^*<\inf g([a'-\gamma^*,b'+\gamma^*]),\quad  b'-\gamma^*>\sup g([a'-\gamma^*,b'+\gamma^*]),
		$$
		and
		$$
		\text{either}\,  \,  a+\gamma^*<\inf g([a'-\gamma^*,b'+\gamma^*]),\, \,
		 \text{or} \, \,  b-\gamma^*>\sup g([a'-\gamma^*,b'+\gamma^*]).
		$$
		These, together with $a^*=a'$ and $b^*=b'$,  yield the desired conclusions.
		
		{\bf Case 2.} $a'<\inf g([a',b'])$ and $b'=\sup g([a',b'])$.
		
		In this case, it is easy to see  $b'\geq x^*$. Take $\delta'>0$ such that $a'<\inf g([a',b'+\delta'])$ and  either $a<\inf g([a',b'+\delta'])$ or $b>\sup g([a',b'+\delta'])$. Since $b'=\sup g([a',b'])$ and $x>g(x)$ for all $x\in (b',b'+\delta']$, it follows  that $b'+\delta'>\sup g([a',b'+\delta'])$. By applying (i) to $\{a',b'+\delta'\}$, we then obtain the desired conclusions.
		
		{\bf Case 3.} $a'=\inf g([a',b'])$ and $b'>\sup g([a',b'])$.
		
		In this case, $a'\leq x^*$. Select $\delta{''}>0$ such that $b'>\sup g([a'-\delta{''},b'])$ and  either $a<\inf g([a'-\delta{''},b'])$ or $b>\sup g([a',b'+\delta{''}])$. Since $a'=\sup g([a',b'])$ and $x>g(x)$ for all $x\in [a'-\delta{''},a')$, it follows  that $a'+\delta{''}<\sup g([a'-\delta{''},b'])$. By applying (i) to $\{a'-\delta{''},b'\}$, we then get the desired conclusions.
		\end{proof}
	
	Based on Lemma~\ref{lem6.5} and the assumptions (B$_\pm$), we easily see that the real functions  may exhibit {\bf (NM)} and {\bf (NM$_-$)}
	features  when the limiting functions have unique  positive  periodic-two point, which plays a key role in our verification of   {\bf (NM)} and  {\bf (NM$_-$)}.
	\begin{lemma}\label{lem6.9} Let  $b\geq a>0$.  Then the statements are valid:
		\begin{itemize}
			\item[{\rm (i)}]
			If $\{a,b\}\neq \{u_+^*\}$ and {\rm (B$_+$)} hold, then there exist $b^*>a^*>\gamma^*>0$ and $z^*>0$ such that $[a,b]\subseteq [a^*,b^*]$, $Cl( f([z^*,\infty)\times [a^*-\gamma^*,b^*+\gamma^*]))\subseteq (a^*+\gamma^*,b^*-\gamma^*)$ and either $a+\gamma^*<\inf f([z^*,\infty)\times [a^*-\gamma^*,b^*+\gamma^*])$ or $b-\gamma^*>\sup f([z^*,\infty)\times[a^*-\gamma^*,b^*+\gamma^*])$.
			
			\item[{\rm (ii)}]
			If $\{a,b\}\neq \{u_-^*\}$ and {\rm (B$_-$)} hold, then there exist $b^*>a^*>\gamma^*>0$ and $z^*>0$ such that $[a,b]\subseteq [a^*,b^*]$, $ Cl(f((-\infty,-z^*]\times [a^*-\gamma^*,b^*+\gamma^*]))\subseteq (a^*+\gamma^*,b^*-\gamma^*)$ and either $a+\gamma^*<\inf f((-\infty,-z^*]\times [a^*-\gamma^*,b^*+\gamma^*])$ or $b-\gamma^*>\sup f((-\infty,-z^*]\times[a^*-\gamma^*,b^*+\gamma^*])$.
		\end{itemize}
	\end{lemma}
	
By the arguments similar to those for  \cite[Lemma 2.3.2]{Zhaobook}, we have  the following observation.
		\begin{lemma}\label{lem6.10-subhomgenous}
			
			Let $F^*\in Int(\mathbb{R}_+^N)$, $F\in C(\mathbb{R}\times\mathbb{R}_{+}^N,\mathbb{R}^N), F(x,\cdot)\in C^{1}(\mathbb{R}_{+}^N,\mathbb{R}^N),\mbox{ and } F(x,0)=0 \mbox{  for all }
			x\in \mathbb{R}$. If $F(x,\cdot)|_{[0,F^*]_{\mathbb{R}^N}}$ is a subhomogeneous  map for each $x\in \mathbb{R}$, then the following statements are valid:
			\begin{itemize}
				\item[{\rm (i)}] $F(x,u)\leq {{\rm D}_u F(x,0)} u$ for all $(x,u)\in \mathbb{R}\times [0,F^*]_{\mathbb{R}^N}$;
				
				\item[{\rm (ii)}]  ${\rm D}_uF(x,u)u \leq {\rm D}_u F(x,0)u$ for all $(x,u)\in \mathbb{R}\times [0,F^*]_{\mathbb{R}^N}$. In particular,  ${\rm D}_uF(x,u) \leq {\rm D}_u F(x,0)$ for all $(x,u)\in \mathbb{R}\times [0,F^*]_{\mathbb{R}^N}$ with $N=1$.
				
	%\item [{\rm (iii)}] If  $s(\limsup\limits_{s \to \infty}\frac{D F(x,0)}{D u})<1$ and  $f(x,\cdot)$ is  monotone, strictly subhomogeneous on $(0,\infty)$ for all $x\in (0,\infty)$, then $$\lim\limits_{n\rightarrow \infty}||u_n^\phi-W||_{L^\infty(\mathbb{R},\mathbb{R})}= 0$$ for all $\phi\in C_+\setminus\{0\}$.
\end{itemize}
\end{lemma}

		Let $M=[-\tau,0]$, $X=BC(\mathbb{R},\mathbb{R})$,  $X_+=BC(\mathbb{R},\mathbb{R}_+)$, $C=BC([-\tau,0]\times \mathbb{R},\mathbb{R})$, and  $C_+=BC([-\tau,0]\times \mathbb{R},\mathbb{R}_+)$. Define
		$$
		e_{\phi,\nu}(\theta,x)=\varphi(\theta,x) e^{-\nu x},\, \, \forall
		(\varphi,\nu,x)\in C\times \mathbb{R}\times \mathbb{R},
		$$
		and
	$$
		\tilde{C}:=\left\{\sum\limits_{k=1}^p  e_{\varphi_k,\nu_k}:(\varphi_k,\nu_k)\in  C\times\mathbb{R}, p\in  \mathbb{N}\right\}.
		$$
		Let $S(t)$ be the semigroup generated by the linear parabolic equation:
		\[
		\left\{
		\begin{array}{rcll}
		\frac{\partial u}{\partial t}&=& d \Delta u(t,x)+cu_x-\mu u, \qquad & t>0,
		\\
		u(0,x)& =& \phi(x), & x\in \mathbb{R},
		\end{array}
		\right.
		\]
		that is,  for $(x,\phi)\in \mathbb{R}\times X$,
		
\begin{equation}\label{eqn6.3ccc}
		\left\{
		\begin{array}{rcl}
		S(0)[\phi](x) & = & \phi(x), \\
		S(t)[\phi](x) & = & \frac{\exp(-\mu t)}{\sqrt{4d\pi t}}\int_{-\infty}^\infty
		\phi(y)\exp\left(-\frac{(x+ct-y)^2}{4
			dt}\right){\rm d}y,
		\quad t >0.
		\end{array}
		\right.
		\end{equation}
		For any $(x,\varphi)\in \mathbb{R}\times \tilde{C}$,
		we also extend the domain of  $S$ to  define
		\begin{equation}\label{eqn6.3ccc}
		\left\{
		\begin{array}{rcl}
		S(0)[\varphi(0,\cdot)](x) & = & \varphi(0,x), \\
		S(t)[\varphi(0,\cdot)](x) & = & \frac{\exp(-\mu t)}{\sqrt{4d\pi t}}\int_{-\infty}^\infty
		\varphi(0,y)\exp\left(-\frac{(x+ct-y)^2}{4
			dt}\right){\rm d}y,
		\quad t >0.
		\end{array}
		\right.
		\end{equation}
		Since $S(t)[e_{\phi,\nu}]\in X$,  it follows that
		$$
		S(t)[\varphi(0,\cdot)]\in X,	\, \, \forall  (t,\phi,\nu)\in (0,\infty)\times X\times \mathbb{R}, \, \, \varphi \in \tilde{C}.
		$$
		
		By  the standard step arguments, it follows that for any given  $\varphi\in  C_+$, equation~\eqref{6.8}  has a unique mild solution on $\mathbb{R}_+$,
		%in thesense of Lunardi~\cite{l1995}
		denoted by $u^{\varphi}(t,x;f)$ or $(u^{\varphi;f})_t$, which is
		also the classical solution of~\eqref{6.8} on  $(\tau,\infty)$
		with $\mathbb{R}_+\ni t\mapsto (u^{\varphi;f})_t\in C_+$ being
		continuous.
		Recall that $(u^{\varphi;f})_t(\theta,x)=u^{\varphi}(t+\theta,x;f),
			\, \, \forall (\theta,x)\in [-\tau,0]\times \mathbb{R}$.
		Further, we have the following result on the monotonicity and boundedness of solutions to system \eqref{6.8}.
		
		%Clearly, for any $t_1\in [t_0,\eta_{t_0,\varphi,r(\cdot)})$ and $  t\in
		%[t_1,\eta_{t_0,\varphi,r(\cdot)})$, we easily see that $$u^{t_0,\varphi}(r(\cdot);t,\cdot) = e^{\mu (t-t_1) }S(t-t_1)[u^{t_0,\varphi}(r(\cdot);t_1,\cdot)]+\int_{t_1}^t
		%e^{\mu (t-s) }S({t-s})[(r(\cdot-cs)-\mu) u(s,\cdot)-u^2(s,\cdot)] \mathrm {d} s.$$
		%For notational convenience, we write   $\eta_{0,\varphi;f} $ and $u^{0,\varphi}(t,\cdot;f)$,  respectively, as $\eta_{\varphi;f}$ and $u^{\varphi}(t,\cdot;f)$.
		%Thus,  it suffices to consider \eqref{6.1} with $t_0=0$ since $u^{0,\varphi}(t,x;f(-ct_0+\cdot,\cdot))=u^{t_0,\varphi}(t+t_0,x;f)$ whenever defined for $(t_0,t,x)\in \mathbb{R}_+\times\mathbb{R}_+\times
		%\mathbb{R}$.
		
	\begin {prop} \label{prop6.8} Assume that $M^*>0$ and $f_1,f_2\in C(\mathbb{R}\times  \mathbb{R}_+,\mathbb{R})$ satisfy   $f_2(\cdot,\cdot)\geq f_1(\cdot,\cdot)$  and  $ f_2(s,u)\leq M^*$ for all $(s,u)\in \mathbb{R}\times [0,M^*]$. Let $\psi$, $\varphi \in C_+$ with $\varphi\leq \psi\leq M^*$. If $f_1(s,u)$ is nondecreasing in $u\in [0,M^*]$ for each $s\in \mathbb{R}$ or $f_2(s,u)$ is nondecreasing in $u\in [0,M^*]$ for each $s\in \mathbb{R}$, then  $0\leq  u^\varphi(t,x;f_1)\leq u^\psi(t,x;f_2)\leq M^*$ for all $(t,x)\in [0,\infty) \times \mathbb{R}$. In particular,  $0\leq  u^\varphi(t,x;f)\leq  u_k^*$ for all $k\in \mathbb{N}$ and $\varphi\in C_{u_k^*}$.
	\end{prop}
	
	Now we introduce  the following three auxiliary equations:
	
	\begin {equation}
	\frac {\partial u}{\partial t}= cu_x+d u_{xx}(t,x)-\mu u(t,x) +
	\mu f(x,u(t-\tau,x+c\tau)), \qquad t>0, x\in\mathbb {R},
	\label {6.9}
	\end {equation}

	\begin {equation}
	\frac {\partial u}{\partial t}= cu_x+d u_{xx}(t,x)-\mu u(t,x) +
	\mu f_\pm^\infty(u(t-\tau,x+c\tau)), \qquad t>0, x\in\mathbb {R},
	\label {6.10}
	\end {equation}
	and
	\begin {equation}
	\frac {\partial u}{\partial t}= d u_{xx}(t,x)-\mu u(t,x) +
	\mu f_\pm^\infty(u(t-\tau,x)), \qquad t>0, x\in\mathbb {R}.
	\label {6.11}
	\end {equation}
	Define  $P:\mathbb{R}_+\times C_+\to C_+$ by
	$$
	P[t,\varphi;f](\theta,x)=u^{\varphi}(t+\theta,x;f),\, \, \forall
	(t,\theta,\varphi)\in \mathbb{R}_+\times  [-\tau,0] \times C_+,
	$$
and let $Q[t,\varphi;f]$, $Q_\pm[t,\varphi;f_\pm^\infty]$ and $\Phi_\pm[t,\varphi;f_\pm^\infty]$ be the mild solutions  of equations \eqref{6.9}, \eqref{6.10}  and \eqref{6.11} with the initial value $u_0 = \varphi\in C_+$, respectively.   For simplity, we denote
	$P[t,\varphi;f]$, $Q[t,\varphi;f]$, $Q_\pm[t,\varphi;f_\pm^\infty]$ and $\Phi_\pm[t,\varphi;f_\pm^\infty]$ by $P[t,\varphi]$, $Q[t,\varphi]$, $Q_\pm[t,\varphi]$ and $\Phi_\pm[t,\varphi]$, respectively.
	
	In the following, we also use  $Q[t,\varphi;\zeta(\cdot) Id_{\mathbb{R}}]$ to represent the mild solution of the following linear reaction-diffusion equation with time delay:
	\begin {equation}
	\left\{
	\begin{array}{ll}
		\frac{\partial u}{\partial t}(t,x)  =cu_x+ d u_{xx}(t,x)-\mu u(t,x) +
		\mu \zeta(x)u(t-\tau,x+c \tau), \, (t,x)\in (0,\infty)\times \mathbb{R}, \\
		u(\theta,x) = \varphi(\theta,x),  \quad  (\theta,x) \in [-\tau,0]
		\times \mathbb{R},
	\end{array} \right.
	\label{6.8-lin-lin-lin}
	\end {equation}
	where $\zeta\in C(\mathbb{R},\mathbb{R})\bigcap L^\infty(\mathbb{R},\mathbb{R})$ and $\varphi\in \tilde{C}$. In other words,   $Q[t,\varphi;\zeta(\cdot)Id_{\mathbb{R}}](x)$  satisfies the integral form of equation \eqref{6.8-lin-lin-lin}:
	\begin {equation}
	\left\{
	\begin{array}{ll}
		u(t,x)  =S(t)[\varphi](x)+\int_0^t
		S(t-s)[\mu \zeta(\cdot)u(s-\tau,\cdot+c \tau)](x){\rm d} s, \, (t,x)\in [0,\infty)\times \mathbb{R},
		\\
		u_0=\varphi\in \tilde{C}.
	\end{array} \right.
	\label{6.8-lin-lin-lin-integ}
	\end {equation}
	Using   the standard step arguments, one can easily obtain the global existence on $\mathbb{R}_+$ and positivity of solutions to~\eqref{6.8-lin-lin-lin-integ}.
	
	By the definitions of $P,Q,Q_\pm$ and $\Phi_\pm$, it is easy to verify the following relations.
	\begin{prop} \label{prop6.9} Let $t\in \mathbb{R}_+$ and $\varphi\in C_+$. Then the following statements are valid:
		\begin{itemize}
			\item [{\rm (i)}]  $Q[t,\varphi](\theta,x)=P[t,\varphi](\theta,x+ct)$ for all $(\theta,x)\in [-\tau,0]\times \mathbb{R}$.

			\item [{\rm (ii)}]  If $\lim\limits_{s\to \pm \infty}
			f(s,\cdot)= f_\pm^\infty$ in $L^\infty_{loc}(\mathbb{R}_+,\mathbb{R})$, then $ Q_\pm[t,\varphi]=\lim\limits_{y\to \pm \infty} Q[t,T_{y}[\varphi]](\cdot,\cdot+y)$ in $C$.

			\item [{\rm (iii)}]   $\Phi_\pm[t,\varphi]=Q_\pm[t,\varphi](\cdot,\cdot-ct).$
		\end{itemize}

	\end{prop}
	
	%??š¬2šŠš°?šŠ?š®3šŠ?1??ŠÌ
	
	%$v(t,x)=u(t,x+ct)$
	
	%$v_t=u_t+cu_x=cu_x+du_{xx}-\mu u+\mu f(x+ct-ct,u(t-\tau,x+ct))$
	
	%$v_t=cv_x+dv_{xx}-\mu v+\mu f(x,u(t-\tau,x+ct))$
	
	%$v(t,x)=u(t,x-ct)$$v_t=u_t-cu_x=-cu_x+cu_x+du_{xx}-\mu u+\mu f_\pm^\infty(u(t-\tau,x-ct+c\tau))$$v_t=du_{xx}-\mu u+\mu f_\pm^\infty(u(t-\tau,x-ct+c\tau))$$v_t=dv_{xx}-\mu v+\mu f_\pm^\infty(u(t-\tau,x-ct+c\tau))$?
	
	We can easily verify the following result on the compactness,
	monotonicity  and positivity.
	\begin{prop} \label{prop6.10} The following statements are valid:
		\begin{itemize}
			\item [{\rm (i)}]  $Q[t,C_+\setminus \{0\}]\subseteq C_+\setminus \{0\}$ and $Q[t, C_{u_k^*}]\subseteq C_{u_k^*}$ for all $k \in \mathbb{N}$ and $t\in \mathbb{R}_+$.
			Moreover, $Q[t,\varphi]\in C_+^\circ$ for all $(t,\varphi)\in(\tau,\infty)\times C_+$ with $\varphi(0,\cdot)>0$ or $f(\cdot,\varphi(\cdot,\cdot))>0$, and hence, $Q[(\tau,\infty)\times (C_+ \setminus\{0\})]\subseteq C_+^\circ$
			provided that $f(\mathbb{R}\times (0,\infty))\subseteq (0,\infty)$.
			
			\item [{\rm (ii)}]  If $f(s,u)$ is nondecreasing in $u\in \mathbb{R}_+$ for each $s\in \mathbb{R}$, then  $Q[t,\varphi]\geq Q[t,\psi]$ for all $t\in \mathbb{R}_+$ and $\varphi,\psi\in C_+$ with $\varphi\geq \psi$.
			
			\item [{\rm (iii)}]  If $f(s,u)$ is nondecreasing in $s\in \mathbb{R}$ for each $u\in \mathbb{R}_+$, then  $T_{-y}[Q[t,T_y[\varphi]]]\geq Q[t,\varphi]$ for all   $ (t,y,\varphi)\in \mathbb{R}_+^2\times C_+$.
			
			\item [{\rm (iv)}]  If $f(s,u)$ is nonincreasing in $s\in \mathbb{R}$ for each $u\in \mathbb{R}_+$, then  $T_{-y}[Q[t,T_y[\varphi]]]\leq Q[t,\varphi]$ for all   $(t,y,\varphi)\in \mathbb{R}_+^2\times C_+$.	
				
			\item [{\rm (v)}]  If  $r>0$, then  $Q[t,C_r]$ is precompact in $C$  for all  $t>\tau$ and $Q[t,C_r](0,\cdot)$ is precompact in $X$ for all $t>0$.

				\item [{\rm (vi)}]  If $f(s,\alpha u)\geq \alpha f(s,u)$ for all $(s,u,\alpha)\in \mathbb{R}\times \mathbb{R}_+\times [0,1]$, then  $Q[t,\alpha \varphi)]\geq \alpha Q[t,\varphi]$ for all $(t,\varphi,\alpha)\in \mathbb{R}\times C_+\times[0,1]$.
		\end{itemize}
		\end{prop}

	The following three lemmas are useful in our verification of  {\bf (NM)} and {\bf (NM$_-$)}.
	
	\begin{lemma}\label{lem666.0010}
		Assume that there exists $x^*\in(0,\infty)$ such that  $f(\mathbb{R}_+\times [0,x^*])\subseteq [0,x^*]$. Let $T\in [\tau,2\tau]$, $\beta\in (0,x^*)$, and $x^*>b>a>0$. Then the following statements are
		valid:
		\begin{itemize}
			\item[{\rm (i)}] If $\beta<\inf f_\pm^\infty([a,b])$ and $\lim\limits_{s\to \pm \infty}
			f(s,\cdot)= f_\pm^\infty$ in $L^\infty_{loc}(\mathbb{R}_+,\mathbb{R})$, then for any $\gamma^*>0$, there exist $\sigma^*=\sigma^{*}(x^*,\gamma^*)>0$ and $z^*=z^{*}(x^*,\gamma^*,a,b,\beta)>0$ such that  $$T_{-z}\circ Q[t,T_z[\varphi]](0,0)>\beta+(\inf\{\varphi(0,y):y\in [-\sigma^*,\sigma^*]\}-\beta)e^{-\mu t}-\gamma^*$$  for all $\hat{z}\geq z^*$ and $(\pm z, t,\varphi)\in   [\hat{z},\infty)\times(0,T]\times  \mathcal{H}_{\sigma^*,\hat{z}}^\pm$, where {\[ \mathcal{H}_{\sigma_0,z_0}^\pm:= \{\psi\in C_+:T_{-z-c\tau}\circ Q[[0,T-\tau],T_z[\psi]]\subseteq [a\xi_{\sigma_0},b\tilde{\xi}_{\sigma_0,\frac{x^*}{b}}]_C, \, \, \forall
				\pm z\geq z_0\}\]} for all positive numbers $\sigma_0$ and $z_0$.
			
			\item[{\rm (ii)}] If $\beta>\sup f_\pm^\infty([a,b])$ and $\lim\limits_{s\to \pm \infty}
			f(s,\cdot)= f_\pm^\infty$ in $L^\infty_{loc}(\mathbb{R}_+,\mathbb{R})$, then for any $\gamma^*>0$, there exist $ \sigma^{**}=\sigma^{**}(x^*,\gamma^*)>0$ and $z^{**}=z^{**}(x^*,\gamma^*,a,b,\beta)>0$ such that  $$T_{-z}\circ Q[t,T_z[\varphi]](0,0)<\beta+(\sup\{\varphi(0,y):y\in [-\sigma^{**},\sigma^{**}]\}-\beta)e^{-\mu t}+\gamma^*$$ for all $\hat{z}\geq z^{**}$ and $(\pm z, t,\varphi)\in   [\hat{z},\infty)\times(0,T]\times  \mathcal{H}_{\sigma^{**},\hat{z}}^\pm$.
			
		\end{itemize}
	\end{lemma}

	\begin{proof}  We only consider the cases of $+$ since the cases of $-$ can be derived by applying the results of $+$ to $\mathcal{S}\circ Q[t,\cdot;f]\circ \mathcal{S}$.
		
		(i) Fix $t\in (0,T]$ and $\gamma^*>0$. Selete $\delta\in (0,1)$ such that $2 x^*\delta<\gamma^*$. Take  $\sigma^*,z^*\in (0,\infty)$ such that
		$$
		\frac{1}{\sqrt{\pi}}\int_{|y|\geq \frac{\sigma^*-2\tau c}{2\sqrt{2d\tau}}}
		e^{-y^2}{\rm d}y<\delta, \quad f([z^*-\sigma^*,\infty)\times [a,b])\subseteq [\beta,\infty).
		$$
		It follows  from \eqref{6.9} that for any  $\hat{z}\geq z^*$, $(z,\varphi)\in [\hat{z},\infty)\times \mathcal{H}_{\sigma^*,\hat{z}}^+$ with $\underline{a}=\inf\{\varphi(0,y):y\in [-\sigma^*,\sigma^*]\}$,  there holds
		\begin{eqnarray*}
			&&T_{-z}\circ Q[t,\cdot] \circ T_{z}[\varphi](0,0)
			\\
			&&= \frac{e^{-\mu t}}{\sqrt{4d\pi t}}\int_{-\infty}^\infty
			\varphi(0,y)e^{-\frac{(ct-y)^2}{4
					dt}}{\rm d}y
			\\
			&&\quad +\mu \int_0^t \frac{e^{-\mu (t-s)}}{\sqrt{4d\pi (t-s)}}\int_{-\infty}^\infty
			f(y+z,T_{{-z-c\tau}}[Q[s,T_{z}[\varphi]] ](-\tau,y))e^{-\frac{(c(t-s)-y)^2}{4
					d(t-s)}}{\rm d}y {\rm d}s
					\\
					&&\geq  \frac{\underline{a} e^{-\mu t}}{\sqrt{4d\pi t}}\int_{-\sigma^*}^{\sigma^*}
			e^{-\frac{(ct-y)^2}{4
					dt}}{\rm d}y
			\end{eqnarray*}

\begin{eqnarray*}
			&&\quad +\mu \int_0^t \frac{e^{-\mu (t-s)}}{\sqrt{4d\pi (t-s)}}\int_{-\sigma^*}^{\sigma^*}
			\inf f(y+z,[a,b])e^{-\frac{(c(t-s)-y)^2}{4
					d(t-s)}}{\rm d}y {\rm d}s
			\\
			&&\geq  \frac{\underline{a} e^{-\mu t}}{\sqrt{4d\pi t}}\int_{-\sigma^*}^{\sigma^*}
			e^{-\frac{(ct-y)^2}{4
					dt}}{\rm d}y
			+\mu \beta\int_0^t \frac{e^{-\mu (t-s)}}{\sqrt{4d\pi (t-s)}}\int_{-\sigma^*}^{\sigma^*}
			e^{-\frac{(c(t-s)-y)^2}{4
					d(t-s)}}{\rm d}y {\rm d}s
			%\\
			%&&\geq  \frac{\underline{a} e^{-\mu t}}{\sqrt{4d\pi t}}\int_{-\sigma^*-ct}^{\sigma^*-ct}e^{-\frac{y^2}{4dt}}{\rm d}y+\mu \beta\int_0^t \frac{e^{-\mu s}}{\sqrt{4d\pi s}}\int_{-\sigma^*-cs}^{\sigma^*-cs}e^{-\frac{y^2}{4ds}}{\rm d}y {\rm d}s
			\\
			&&=\frac{\underline{a} e^{-\mu t}}{\sqrt{\pi}}\int_{\frac{-\sigma^*-ct}{2\sqrt{dt}}}^{\frac{\sigma^*-ct}{2\sqrt{dt}}}
			e^{-y^2}{\rm d}y
			+\mu \beta\int_0^t \frac{e^{-\mu s}}{\sqrt{\pi }}\int_{\frac{-\sigma^*-cs}{2\sqrt{ds}}}^{\frac{\sigma^*-cs}{2\sqrt{ds}}}
			e^{-y^2}{\rm d}y{\rm d}s
			\\
			&&\geq  \frac{\underline{a} e^{-\mu t}}{\sqrt{\pi}}\int_{\frac{-\sigma^*}{2\sqrt{2d\tau}}}^{\frac{\sigma^*-2\tau c}{2\sqrt{2d\tau}}}
			e^{-y^2}{{\rm d}y}
			+\mu \beta\int_0^t \frac{e^{-\mu s}}{\sqrt{\pi }}\int_{\frac{-\sigma^*}{2\sqrt{2d\tau}}}^{\frac{\sigma^*-2\tau c}{2\sqrt{2d\tau}}}
			e^{-y^2}{\rm d}y {\rm d}s
			\\
			&&\geq  \underline{a}(1-\delta) e^{-\mu t}
			+\beta (1-\delta)(1-e^{-\mu t})
			%\\
			%&&\geq \beta+(\underline{a}-\beta) e^{-\mu t}- \underline{a}\delta-\beta \delta
			\\
			&&\geq  \beta+(\underline{a}-\beta) e^{-\mu t}-\gamma^*.
		\end{eqnarray*}
		Thus, (i) follows from the definition of $\underline{a}$.
		
		(ii) Fix $t\in (0,T]$ and $\gamma^*>0$. Selete $\delta\in (0,1)$ such that $ x^*\delta<\gamma^*$. Take  $\sigma^{**},z^{**}\in (0,\infty)$ such that $$
		\frac{1}{\sqrt{\pi}}\int_{|y|\geq \frac{\sigma^{**}-2\tau c}{2\sqrt{2d\tau}}}
		e^{-y^2}{\rm d}y<\delta, \quad
		f([z^{**}-\sigma^{**},\infty)\times [a,b])\subseteq [0,\beta].
		$$
		It follows   from \eqref{6.9} that for any  $\hat{z}\geq z^{**}$, $(z,\varphi)\in [\hat{z},\infty) \times \mathcal{H}_{\sigma^{**},\hat{z}}^+$ with $\overline{b}:=\sup\{\varphi(0,y):y\in [-\sigma^{**},\sigma^{**}]\}$, there holds
		\begin{eqnarray*}
			&&T_{-z}\circ Q[t,\cdot] \circ T_{z}[\varphi](0,0)
			\\
			%&&= T_{-z}\circ S(t)[T_{z}[\varphi](0,\cdot)](0) +\mu \int_0^t T_{-z}\circ S(t-s)\circ T_z
			%[T_{-z}[f(\cdot,Q[s,T_{z}[\varphi]] (-\tau,\cdot+c\tau))]](0){\rm d}s
			%\\
			%&&=  S(t)[\varphi(0,\cdot)](0) +\mu \int_0^t  S(t-s)[T_{-z}[f(\cdot,Q[s,T_{z}[\varphi]] (-\tau,\cdot+c\tau))]](0){\rm d}s
			%\\
			%&&= S(t)[\varphi(0,\cdot)](0) +\mu \int_0^t S(t-s)
			%[f(\cdot+z,T_{-z}[Q[s,T_{z}[\varphi]] (-\tau,\cdot+c\tau)])](0){\rm d}s
			%\\
			%&&= \frac{e^{-\mu t}}{\sqrt{4d\pi t}}\int_{-\infty}^\infty\varphi(0,y)e^{-\frac{(ct-y)^2}{4dt}}{\rm d}y
			%\\
			%&&\quad  +\mu \int_0^t \frac{e^{-\mu (t-s)}}{\sqrt{4d\pi (t-s)}}\int_{-\infty}^\infty	f(y+z,T_{-z}[Q[s,T_{z}[\varphi]] ](-\tau,y+c\tau))e^{-\frac{(c(t-s)-y)^2}{4d(t-s)}}{\rm d}y {\rm d}s\\
			&&=\frac{e^{-\mu t}}{\sqrt{4d\pi t}}\int_{-\infty}^\infty
			\varphi(0,y)e^{-\frac{(ct-y)^2}{4
					dt}}{\rm d}y
			\\
			&&\quad  +\mu \int_0^t \frac{e^{-\mu (t-s)}}{\sqrt{4d\pi (t-s)}}\int_{-\infty}^\infty
			f(y+z,T_{ {-z-c\tau}}[Q[s,T_{z}[\varphi]] ](-\tau,y))e^{-\frac{(c(t-s)-y)^2}{4
					d(t-s)}}{\rm d}y {\rm d}s
			\\
			&&\leq  \frac{\bar{b} e^{-\mu t}}{\sqrt{4d\pi t}}\int_{-\sigma^{**}}^{\sigma^{**}}
			e^{-\frac{(ct-y)^2}{4
					dt}}{\rm d}y+\frac{x^*e^{-\mu t}}{\sqrt{4d\pi t}}\int_{|y|\geq \sigma^{**}}
			e^{-\frac{(ct-y)^2}{4
					dt}}{\rm d}y
			\\
			&&\quad  +\mu \int_0^t \frac{\beta e^{-\mu (t-s)}}{\sqrt{4d\pi (t-s)}}\int_{-\sigma^{**}}^{\sigma^{**}}
			e^{-\frac{(c(t-s)-y)^2}{4
					d(t-s)}}{\rm d}y {\rm d}s+\mu \int_0^t \frac{x^*e^{-\mu (t-s)}}{\sqrt{4d\pi (t-s)}}\int_{|y|\geq  \sigma^{**}}
			e^{-\frac{(c(t-s)-y)^2}{4
					d(t-s)}}{\rm d}y {\rm d}s
			\\
			&&= \frac{\bar{b} e^{-\mu t}}{\sqrt{\pi }}\int_{\frac{-\sigma^{**}-ct}{\sqrt{4dt}}}^{\frac{\sigma^{**}-ct}{\sqrt{4dt}}}
			e^{-y^2}{\rm d}y+\frac{x^*e^{-\mu t}}{\sqrt{\pi }}\int_{|y\sqrt{4dt}+ct|\geq \sigma^{**}}
			e^{-y^2}{\rm d}y
			\\
			&&\quad +\mu \int_0^t \frac{\beta e^{-\mu s}}{\sqrt{\pi }}\int_{\frac{-\sigma^{**}-cs}{\sqrt{4d s}}}^{\frac{\sigma^{**}-cs}{\sqrt{4d s}}}
			e^{-y^2}{\rm d}y {\rm d}s+\mu \int_0^t \frac{x^*e^{-\mu s}}{\sqrt{\pi }}\int_{|y\sqrt{4d s}+cs|\geq  \sigma^{**}}
			e^{-y^2}{\rm d}y {\rm d}s
			%\\
			%&&\leq \underline{b} e^{-\mu t}+\frac{x^*e^{-\mu t}}{\sqrt{\pi }}\int_{|y\sqrt{4dt}+ct|\geq \sigma^{**}}
			%e^{-y^2}{\rm d}y\\&&\quad  + \beta(1- e^{-\mu t})+\mu \int_0^t \frac{x^*e^{-\mu s}}{\sqrt{\pi }}\int_{|y\sqrt{4d s}+cs|\geq  \sigma^{**}}e^{-y^2}{\rm d}y {\rm d}s
			\\
			&&\leq \bar{b} e^{-\mu t}+\frac{x^*e^{-\mu t}}{\sqrt{\pi }}\int_{|y| \geq \frac{\sigma^{**}-2c\tau}{2\sqrt{2d\tau}}}
			e^{-y^2}{\rm d}y
			\\
			&&\quad  + \beta(1- e^{-\mu t})+\mu \int_0^t \frac{x^*e^{-\mu s}}{\sqrt{\pi }}\int_{|y| \geq \frac{\sigma^{**}-2c\tau}{2\sqrt{2d\tau}}}
			e^{-y^2}{\rm d}y {\rm d}s
			\\
			&&<  \beta+(\bar{b}-\beta) e^{-\mu t}+x^*e^{-\mu t}\delta
			+x^*(1-e^{-\mu t})\delta
			\\
			&&< \beta+(\bar{b}-\beta) e^{-\mu t}+\gamma^*.
		\end{eqnarray*}
		This, together with the definition of $\bar{b}$,  implies  (ii).
	\end{proof}

For any $\mu\in\mathbb{R}$, let $\Psi_{\pm,\mu}(t,\cdot)$ be the  solution maps of the following linear reaction-diffusion equations with time delay:
	\begin {equation}
	\left\{
	\begin{array}{ll}
		\frac{\partial v}{\partial t}(t,x)  = d v_{xx}(t,x)-\mu v(t,x) +
		\mu \frac{ {\rm d}f_\pm^\infty(0)}{du}v(t-\tau,x), \, (t,x)\in (0,\infty)\times \mathbb{R}, \\
		v(\theta,x) = \varphi(\theta,x),  \quad  (\theta,x) \in [-\tau,0]
		\times \mathbb{R},
	\end{array} \right.
	\label{6.8-lin-lin-lin}
	\end {equation}
	where $\varphi\in C(\mathbb{R},\mathbb{R})\bigcap L^\infty(\mathbb{R},\mathbb{R})$. Here and after, we  extend the domain of $\Psi_{\pm,\mu}(t,\cdot)$ to $\tilde{C}$ by the step argument.  Now we define $L_{\pm,\mu}:Y\to Y$ by
	$$
	L_{\pm,\mu}[\phi]=\Psi_{\pm,\mu}(1,e_{\phi,\mu})(\cdot,0), \,  \,  \forall \phi\in Y,
	$$
	where $e_{\phi,\mu}(\theta,x)=\phi(\theta)e^{-\mu x}$. Clearly, $L_{\pm,\mu}$ is  continuous and positive. Moreover, $(L_{\pm,\mu})^{n_0}$ is a compact and strongly positive operator on $Y$ for some $n_0\in \mathbb{N}$.
	According to \cite[Lemma 3.2]{lz2007}, we know that $L_{\pm,\mu}$ has the unique principal eigenvalue $\lambda_\pm(\mu)$ and  a unique strongly
	positive eigenfunction $\zeta_\mu^\pm$ associated with  $\lambda_\pm(\mu)>0$ such that $||\zeta_\mu^\pm||=1$.
	Let
	$$c^*(\pm\infty):=\inf\limits_{\mu>0}\frac{1}{\mu}\log\lambda_\pm(\mu). $$
	Then  $c^*(\pm\infty)>0$ provided that $\frac{ {\rm d}f_\pm^\infty(0)}{{\rm d}u}>1$.
	By  \cite[Theorem 5.1]{lz2007}, it follows  that $c^*(\pm\infty)$ is the  spreading speeds  of  the right and left limiting systems \eqref{6.11}, respectively, in the case where $f$  is nondecreasing. 
	In such a case, \cite[Theorem 5.1 ]{lz2007} and Proposition~\ref{prop6.9}-(iii) imply that $t(c^*(\pm\infty)-c)$ and $t(c^*(\pm \infty)+c)$ are just 
the right and left  spreading speeds of $Q_\pm[t,\cdot]$ for any given $t>0$. 
	%Further,\cite[Theorem 5.7]{ycw2013} implies  that $c^*(\pm\infty)$ is the spreading speeds  of   the right and left limiting systems \eqref{6.11}, respectively, in the case where $f$ %satisfies (B$_\pm$).
	
	\begin{prop} \label{prop6.11} % Let $c^*_+(\pm\infty)=2\tau(c^*(\pm\infty)-c)$ and $c_-^*(\pm\infty)=2\tau(c+c^*(\pm \infty))$ with $c^*(\pm \infty)>0$ being defined as above.  
	Then the following statements are valid:
		\begin{itemize}
			
			\item [{\rm (i)}] $Q[2\tau,\cdot]$ satisfies {\bf (NM)}  if {\rm (B$_+$)} holds.

			\item [{\rm (ii)}] $Q[2\tau,\cdot]$ satisfies {\bf (NM$_-$)}  if {\rm (B$_-$)} holds.
			
			%\item [{\rm (iii)}]  $\lim\limits_{t\to \infty}\{||Q[t,\varphi]||_{L^\infty(\mathbb{R},\mathbb{R})}:\varphi\in C_{u_k^*}\}=0$ for all $k \in  \mathbb{N}$ if (B2$_+$) holds.

			%\item [{\rm (iv)}]  $\lim\limits_{t\to \infty}\{||Q[t,\varphi]||_{L^\infty(\mathbb{R},\mathbb{R})}:\varphi\in C_{u_k^*}\}=0$ for all $k \in \mathbb{N}$ if (B2$_-$) holds.
			
		\end{itemize}

	\end{prop}
	
	\noindent
	\begin{proof}  
		It then suffices to prove {\bf (NM)} since {\bf (NM$_-$)} can be verified in a similar way.
		
Fix $r^*=u^*_+$ and $\tilde{t}\geq s\geq r>0$ with $\{r,s\}\neq \{1\}$.   In view of  Lemma~\ref{lem6.9}-(i), there exist $b^*>a^*>\delta^*>0$ and $z^*>0$ such that
$$[rr^*,sr^*] \subseteq [a^*,b^*],\quad  Cl(f([z^*,\infty)\times [a^*-\delta^*,{b^*+\delta^*]}))\subseteq (a^*+\delta^*,b^*-\delta^*),$$
and one of the following two cases happens:
	\begin{itemize}
		\item [{\rm (a)}] $rr^*+\delta^*<\inf f([z^*,\infty)\times [a^*-\delta^*,b^*+\delta^*])$;
			\item [{\rm (b)}] $sr^*-\delta^*>\sup f([z^*,\infty)\times[a^*-\delta^*,{b^*+\delta^*}])$.
		\end{itemize}	
		
		In the following, we only consider the case  (a) since  the case  (b) can be dealt with in a similar way.
		Note that  in the  case  (a), we easily see that $$f_+^\infty([a^*-\frac{\delta^*}{3},b^*+\frac{\delta^*}{3}])\subseteq (a^*+\delta^*,b^*-\delta^*)$$ and  $$ f_+^\infty([a^*-\delta^*,b^*+\delta^*])\subseteq (rr^*+\delta^*,\infty).$$
		Take $\gamma^*=\frac{\delta^*(1-e^{-\mu \tau})}{3}$, $\gamma=\frac{\gamma^*}{r^*}$, and $x^*=u_{k_0}^*\geq (1+\gamma+\tilde{t})r^* {+2b^*+}1$ for some $k_0\in \mathbb{N}$. Let
		%\Gamma^*=\{(a^*-\frac{\delta^*}{3},b^*+\frac{\delta^*}{3},a^*+\delta^*),(a^*-\frac{\delta^*}{3},b^*+\frac{\delta^*}{3},b^*-\delta^*),(a^*-\delta^*,b^*+\delta^*,rr^*+\delta^*)\},$$
		$\sigma^{*}=c\tau+\max \{ \sigma^{*}(x^*,\gamma^*),\sigma^{**}(x^*,\gamma^*)\}$ and $z^*=\max\{z^*(x^*,\gamma^*,a^*-\frac{\delta^*}{3},b^*+\frac{\delta^*}{3},a^*+\delta^*), z^{*}(x^*,\gamma^*,a^*-\delta^*,b^*+\delta^*,rr^*+\delta^*),z^{**}(x^*,\gamma^*,a^*-\frac{\delta^*}{3},b^*+\frac{\delta^*}{3},b^*-\delta^*)\}$,
		%:(a,b,\beta)\in \Gamma^*\},
		  where $\sigma^{*}(\cdot),\sigma^{**}(\cdot),z^{*}(\cdot),z^{**}(\cdot)$ are defined as in Lemma~\ref{lem666.0010}.
		
		By applying  Lemma~\ref{lem666.0010} with $a=a^*-\frac{\delta^*}{3}$, $b=b^*+\frac{\delta^*}{3}$, $\beta\in \{a^*+\delta^*,b^*-\delta^*\}$, and $T=\tau$,  we see that for any $t\in (0,\tau]$, $z\geq  z^*$, and $\varphi\in [a\xi_{\sigma^{*}},b\tilde{\xi}_{\sigma^{*},\frac{x^*}{b}}]_C$,
		there holds
		\begin{eqnarray*}
			b^*+\delta^*&\geq & b^*-\delta^*+(b^*+\frac{\delta^*}{3}-b^*+\delta^*)e^{-\mu t}+\gamma^*
			\\
			&\geq & b^*-\delta^*+(\sup\{\varphi(0,y):y\in [-\sigma^{*},\sigma^{*}]\}-b^*+\delta^*)e^{-\mu t}+\gamma^*
			\\
						&\geq & b^*-\delta^*+(\sup\{\varphi(0,y):y\in [-\sigma^{*}+c\tau,\sigma^{*}-c\tau]\}-b^*+\delta^*)e^{-\mu t}+\gamma^*
			\\
			&>&T_{-z}\circ Q[t,T_z[\varphi]](0,0)
			\\
			&>&a^*+\delta^*+(\inf\{\varphi(0,y):y\in [-\sigma^{*}+c\tau,\sigma^{*}+c\tau]\}-a^*-\delta^*)e^{-\mu t}-\gamma^*
			\\
			&\geq &a^*+\delta^*+(\inf\{\varphi(0,y):y\in [-\sigma^{*},\sigma^{*}]\}-a^*-\delta^*)e^{-\mu t}-\gamma^*
			\\
			&\geq &a^*+\delta^*+(a^*-\frac{\delta^*}{3}-a^*-\delta^*)e^{-\mu t}-\gamma^*
			\\
			&\geq &a^*-\delta^*.
		\end{eqnarray*}
		This, together with the choices of $\sigma^*$ and $z^*$, implies that  $$
		T_{-z-c\tau}\circ Q[t,T_z[\varphi]]\in [(a^*-\delta^*)\xi_{\sigma^*},(b^*+\delta^*)\tilde{\xi}_{\sigma^*,\frac{x^*}{b^*+\delta^*}}]_C
		$$
		for all $(z,t,\varphi)\in [z^*+\sigma^*+1,\infty)\times [0,\tau]\times [a\xi_{2+2\sigma^*+c\tau},b\tilde{\xi}_{2+2\sigma^*+c\tau},\frac{x^*}{b}]_C$.  In particular,   $$T_{-z-c\tau}\circ Q[t,T_z[\varphi]]\in [(a^*-\delta^*)\xi_{\sigma^*},(b^*+\delta^*)\tilde{\xi}_{\sigma^*,\frac{x^*}{b^*+\delta^*}}]_C$$ for all $(z,t,\varphi)\in [z^*+\sigma^*+1,\infty)\times [0,\tau]\times   [(rr^*-\gamma r^*)\xi_{2+2\sigma^*+c\tau},(s r^*+\gamma r^*)\tilde{\xi}_{2+2\sigma^*+c\tau},\frac{x^*}{s r^*+\gamma r^*}]_C$.
		
		By  Lemma~\ref{lem666.0010}-(i) with  $a=a^*-\delta^*$, $b=b^*+\delta^*$, $\beta=rr^*+\delta^*$, and $T=2\tau$, it then follows  that for any $t\in [\tau,2\tau]$ and $\varphi\in [(rr^*-\gamma r^*)\xi_{\sigma_0},(s r^*+\gamma r^*)\tilde{\xi}_{\sigma_0,\frac{x^*}{s r^*+\gamma r^*}}]_C$ with  $z\geq z_0:=z^*+\sigma^*+1$ and $\sigma_0:=2+2\sigma^*+c\tau$,  there holds
		\begin{eqnarray*}
			T_{-z}\circ Q[t,T_z[\varphi]](0,0)
			&>&rr^*+\delta^*+(rr^*-\gamma r^*-rr^*-\delta^*)e^{-\mu t}-\gamma^*
			\\
			&=&rr^*+\delta^*-(\gamma r^*+\delta^*)e^{-\mu \tau}-\gamma^*
			\\
			&> &rr^*+\gamma r^*.
		\end{eqnarray*}
		Therefore, $Q[2\tau,\cdot;f]$ satisfies {\bf (NM)}.
	\end{proof}
	
	Regarding  assumptions {\bf (ACH)} and {\bf (ACH$_-$)}, we have the following observation.

	\begin{prop} \label{prop6.1-10002} The following statements are valid:
			\begin{itemize}
				\item [{\rm (i)}] If {\rm (B$_+$)} holds, then  $Q[2\tau,\cdot]$ satisfies {\bf (ACH)}  with   $T_{-y}\circ Q_l[\varphi]\geq Q_l\circ T_{-y}[\varphi]$ for all $l\in \mathbb{N}$ and $(y,\varphi)\in \mathbb{R}_+\times C_+$,  where $Q_l$ is defined as in  {\bf (ACH)}.
				
				\item [{\rm (ii)}] If {\rm (B$_-$)} holds, then $Q[2\tau,\cdot]$ satisfies {\bf (ACH$_-$)}  with  $T_{y}\circ Q_{-l}[\varphi]\geq Q_{-l}\circ T_{y}[\varphi]$ for all $l\in \mathbb{N}$ and $(y,\varphi)\in \mathbb{R}_+\times C_+$,  where $Q_l$ is defined as in  {\bf (ACH$_-$)}.
			\end{itemize}

		\end{prop}
		
		\begin{proof}  We only prove (i) since (ii) follows from  (i), as applied  to $\mathcal{S}\circ Q[2\tau,\cdot]\circ \mathcal{S}$. For any  $l\in \mathbb{N}$, let us denote $\gamma_l=\frac{1}{2l}(\frac{{\rm d}f_+^{\infty}(0)}{{\rm d}u}-1)$, $r_{l}=r_{u_l^*,\gamma_l}$, $K_l=K_{u_l^*,\gamma_l}$, $f_l=f_{u_l^*,\gamma_l}$, and $f_l^\pm(u)=\lim\limits_{s\to\pm \infty}f_l(s,u)$ for all $u\in \mathbb{R}_+$, where $r_{u_l^*,\gamma_l}$, $K_{u_l^*,\gamma_l}$,  and  $f_{u_l^*,\gamma_l,}$ are  defined as in Lemma~\ref{lem666.001}.  Then $(f_l,f_l^\pm)$ satisfies  (B)  for all $l\in \mathbb{N}$.  Moreover,  $0\leq f_l(s,u)\leq f(s,u)$, $f_l(s,u)$ is nondecreasing in $s,u$,  and $f_l(s,\alpha u)\geq \alpha f_l(s,u)$ for all $(s,u,\alpha)\in \mathbb{R}\times (0,\infty)\times [0,1]$ and $l\in \mathbb{N}$.
			Let
			$$
			r_l^*:=\left\{
			\begin{array}{ll}
			\frac{(1+K_lr_l(\infty))^2}{4K_l}, & r_l(\infty)K_l>1,
			\\
			r_l(\infty), \qquad & r_l(\infty)K_l\leq 1,
			\end {array}
			\right.
			$$
			$Q_l=Q[2\tau,\cdot;f_l]$,  and $Q_{l,+}=Q[2\tau,\cdot;f_l^+]$ for all $l\in \mathbb{N}$. By Theorem 5.1-(ii) in \cite{lz2007}, and Propositions~\ref{prop6.9}-(ii),  \ref{prop6.10}, we easily verify that $Q[2\tau,\cdot]$ satisfies {\bf (ACH)}.
		\end{proof}
		
	\begin{prop} \label{prop6.3-3-16}   $Q[\tau,\cdot]$ satisfies {\bf (GLC)} with $\tau(c^*(\infty)-c)$  and  {\bf (GLC$_-$)} with $\tau(c^*(-\infty)+c)$ provided that {\rm (B$_+$)} and {\rm(B$_-$)} hold.
	\end{prop}
	
	\begin{proof} It suffices to verify {\bf (GLC)}. For any  $l\in \mathbb{N}$, let $R_{l}=\overline{R}_{\frac{1}{l},u_l^*}$, where $\overline{R}_{\frac{1}{l},u_l^*}$ is defined as in Lemma~\ref{lem666.005}. Then $R_{l}$ is a bounded and nonincreasing function such that
		$$
		R_{l}(-\infty)>R_{l}(\infty)=\frac{{\rm d} f_+^\infty(0)}{{\rm d} u}+\frac{1}{l},
		\quad f(s,u)\leq R_{l}(s)u, \, \,  \forall (s,u)\in \mathbb{R}\times [0,u_l^*].
		$$
		Define  $$L_0[\varphi](\theta,x)= \frac{e^{-\mu (\theta+\tau)}}{\sqrt{4d\pi (\theta+\tau)}}\int_{-\infty}^\infty
		\varphi(0,y)e^{-\frac{(x+c[\theta+\tau]-y)^2}{4
				d(\theta+\tau)}}{\rm d}y,$$
		$$L_l[\varphi](\theta,x)= L_0[\varphi](\theta,x)
		+\mu \int_0^{\theta+\tau}\frac{e^{-\mu (\theta+\tau-s)}}{\sqrt{4d\pi (\theta+\tau-s)}}\int_{-\infty}^\infty
		R_l(y)\varphi (s-\tau,y+c\tau)e^{-\frac{(x+c(\theta+\tau-s)-y)^2}{4
				d(\theta+\tau-s)}}{\rm d}y {\rm d}s,$$
		and $$L_l^+[\varphi](\theta,x)=L_0[\varphi](\theta,x)
		+\mu \int_0^{\theta+\tau}\frac{e^{-\mu (\theta+\tau-s)}}{\sqrt{4d\pi (\theta+\tau-s)}}\int_{-\infty}^\infty
		R_l(\infty)\varphi (s-\tau,y+c\tau)e^{-\frac{(x+c(\theta+\tau-s)-y)^2}{4
				d(\theta+\tau-s)}}{\rm d}y {\rm d}s$$
		for all $(\theta,x,\varphi)\in [-\tau,0]\times \mathbb{R}\times \tilde{C}$.
		According to the definitions of $R_l,L_l,L_l^+$,  we easily see that  $$\tau(c^*(\infty)-c)=\lim\limits_{l\to\infty}c^*_{L_l^+}, \quad Q[\tau,\cdot;f]|_{C_{u_l^*}}\leq L_l|_{C_{u_l^*}}, \forall l\in\mathbb{N}.$$
		To finish this proof, we only need to  verify that $(L_l,L_l^+)$ satisfies {\bf (LC)} for all $l\in\mathbb{N}$. Fix $(l,\epsilon,\nu,\zeta)\in\mathbb{N}\times (0,\infty)^2\times BC([-\tau,0],[1,\infty))$.
		In view of  the positivity  of $L_0$, it suffices to prove that $(L_l-L_0,L_l^+-L_0)$ satisfies {\bf (LC)}.
		By letting
		$$P_l(s,x)=\int_{-\infty}^\infty
		R_l(2\sqrt{ds} y+ x+cs-2\nu ds)  e^{-y^2}  {\rm d}y$$
		and
		$$P_l^+(s,x)=\int_{-\infty}^\infty
		R_l(\infty) e^{-y^2}  {\rm d}y=\sqrt{\pi}  R_l(\infty)$$
		for all $(s,x)\in[0,\tau]\times \mathbb{R}$, we easily check that  there exists $\mathfrak{x}_{\epsilon, \nu}>0$ such that  $$P_l(s,x)\leq (1+\nu  \epsilon) P_l^+(s,x), \forall (s,x)\in [0,\tau]\times [\mathfrak{x}_{\epsilon,\nu},\infty).$$
		By  the  definitions of $L_0$ and $L_l$,  it then follows that  for any $(\theta,x)\in [-\tau,0]\times \mathbb{R}$, there holds
		\begin{eqnarray*}
		&&L_l[e_{\zeta,\nu}](\theta,x)-L_0[e_{\zeta,\nu}](\theta,x)\\
		&&=\mu \int_0^{\theta+\tau}\frac{e^{-\mu (\theta+\tau-s)}}{\sqrt{4d\pi (\theta+\tau-s)}}\int_{-\infty}^\infty
			R_l(y) e_{\zeta,\nu}(s-\tau,y+c\tau)e^{-\frac{(x+c(\theta+\tau-s)-y)^2}{4
					d(\theta+\tau-s)}}{\rm d}y {\rm d}s
			\\
			&&=\mu \int_0^{\theta+\tau}\frac{e^{-\mu (\theta+\tau-s)}}{\sqrt{4d\pi (\theta+\tau-s)}}\int_{-\infty}^\infty
			R_l(y) \zeta(s-\tau) e^{-\nu y-\nu c\tau}e^{-\frac{(x+c(\theta+\tau-s)-y)^2}{4
					d(\theta+\tau-s)}}{\rm d}y {\rm d}s
			\\
				&&=\frac{\mu 	e^{-\nu c\tau-\nu x } }{\sqrt{\pi }} \int_0^{\theta+\tau} \zeta(\theta-s) e^{(d \nu^2-c\nu-\mu) s} \int_{-\infty}^\infty
			R_l(2y\sqrt{ds} + x+cs-2\nu ds)  e^{-y^2}{\rm d}y {\rm d}s
			\\
				&&=\frac{\mu 	e^{-\nu c\tau-\nu x } }{\sqrt{\pi }} \int_0^{\theta+\tau} \zeta(\theta-s) e^{(d \nu^2-c\nu-\mu) s} P_l(s,x) {\rm d}s.
		\end{eqnarray*}
		From the  definitions of $L_0,L_l^+$, we similarly check that  for any $(\theta,x)\in [-\tau,0]\times \mathbb{R}$, there holds
		\begin{eqnarray*}
			&&L_l^+[e_{\zeta,\nu}](\theta,x)-L_0[e_{\zeta,\nu}](\theta,x)
			\\
			&&=\frac{\mu 	e^{-\nu c\tau-\nu x } }{\sqrt{\pi }} \int_0^{\theta+\tau} \zeta(\theta-s) e^{(d \nu^2-c\nu-\mu) s} P_l^+(s,x) {\rm d}s.
		\end{eqnarray*}
		Thus, we see from the choice of $\mathfrak{x}_{\epsilon,\nu}$ that
		$$
		L_l[e_{\zeta,\nu}](\theta,x)-L_0[e_{\zeta,\nu}](\theta,x)\leq (1+\nu \epsilon) (L_l^+[e_{\zeta,\nu}](\theta,x)-L_0[e_{\zeta,\nu}](\theta,x)),
		\, \,  \forall \theta\in [-\tau,0], \, \,  x\geq \mathfrak{x}_{\epsilon,\nu}.
		$$
		In particular, $(L_l,L_l^+)$ satisfies {\bf (LC)}. Consequently, $Q[\tau,\cdot]$ satisfies {\bf (GLC)}.
	\end{proof}
	
	Under assumptions (B$_+$) and (B$_-$), we have the following result for system  ~\eqref{6.8}.

	\begin{thm}\label{thm6.9}
		Assume that  $f$ satisfies {\rm (B$_+$) }and {\rm (B$_-$)}.  Then the following statements are valid:
		\begin{itemize}
			\item [{\rm (i)}]  If $\varphi\in C_+\setminus\{0\}$, $c<c^*(\infty)$, and $\varepsilon\in (0,\frac{1}{2}\min\{c^*(-\infty),c^*(\infty)-c\})$,  then
			$$\lim\limits_{t \rightarrow \infty}
			\Big[\sup\{||P[t,\varphi](\cdot, x)-u^*_+\cdot {\bf {1}}_{\mathbb{R}_+}(x-ct)-u^*_-\cdot (1-{\bf {1}}_{\mathbb{R}_+}(x-ct)||: x\in t\mathcal{E}_{\varepsilon,c}\}\Big]= 0.
			$$
		 Moreover, if   $c<c^*(-\infty)$ and $\varepsilon\in (0,\min\{c^*(-\infty)-c,c^*(\infty)-c\})$, then
			$$\lim\limits_{\alpha \rightarrow \infty}
			\Big[\sup\{||P[t,\varphi](\cdot, x)-u^*_+\cdot {\bf {1}}_{\mathbb{R}_+}(x-ct)-u^*_-\cdot (1-{\bf {1}}_{\mathbb{R}_+}(x-ct)||: (t,x)\in \mathcal{E}_{\alpha,\varepsilon,c}\}\Big]= 0.
			$$
			Here,   $\mathcal{E}_{\varepsilon,c}= [-c^*(-\infty)+\varepsilon,\min\{c,c^*(-\infty)\}-\varepsilon] \bigcup [c+\varepsilon,c^*(\infty)-\varepsilon]$ and $\mathcal{E}_{\alpha,\varepsilon,c}=\{(t,x)\in \mathbb{R}_+\times \mathbb{R}: x \in [ t(-c^*(-\infty)+\varepsilon),-\alpha+ct ] \bigcup  [\alpha+ct,  t(c^*(\infty)-\varepsilon)] \}$.
			
			\item [{\rm (ii)}] If $\varphi\in  C_{+}$ has the compact support, then $$\lim\limits_{t\rightarrow \infty}
			\Big[\sup\{
			||P[t,\varphi](\cdot,x)||: x\geq  t(\max\{c,c^*(\infty)\}+\varepsilon) \mbox{ or } x\leq - t(c^*(-\infty)+\varepsilon)\}\Big]=0$$
			 for all $\varepsilon>0$.
			
			\item [{\rm (iii)}] $P$ has a  nontrivial   travelling wave $W(x-ct)$  such that
			$W(\infty)=u^*_+$ and $W(-\infty)=u^*_-$.
			
			\item [{\rm (iv)}] If $c<\min\{c^*(\infty),c^*(-\infty)\}$ and $f(s,u)$ is nondecreasing and subhomogeneous in $u\in \mathbb{R}_+$ for each $s\in \mathbb{R}$, then
			$$\lim\limits_{t\rightarrow \infty}\Big[\sup\{||P[t,\varphi](\cdot,x)-{ W(x-ct)}||: t(-c^*(-\infty)+\varepsilon)\leq x \leq t(c^*(\infty)-\varepsilon)\}\Big]= 0$$ for all $(\varepsilon,\varphi)\in (0,\min\{c^*(\infty)-c,c^*(-\infty)+c\})\times C_+\setminus\{0\}$, where $W$ is defined as in (iii).
		\end{itemize}
		\end{thm}
	
	\begin{proof}  (i)  follows  from  Theorem~\ref{thm5.1-bil-uc/uc}.
		
		(ii) follows from Theorem~\ref{thm5.1-bil-meal/meal}-(i) and (ii).
		
		(iii) follows from Theorem~\ref{thm5.3.6-bil-uc/uc-fixp}.
		
		(iv) To finish the proof, we only need to prove {\bf (GAS-CSF)} in $C_{+}\setminus\{0\}$ due to Theorem~\ref{thm5.1-bil-gas}.
		More precisely, it suffices to prove
		$$\omega(\varphi):=\bigcap\limits_{s\in \mathbb{R}_+} Cl(\{Q[t,\varphi]:t\geq s\})=\{W\},\, \,  \forall \varphi\in C_{+}\setminus\{0\}.
		$$
		Fix $\varphi\in C_{+}\setminus\{0\}$.
		Let $a^*_-=\sup \{a\in \mathbb{R}:\omega(\varphi)\geq a W\}$ and $a_+^*=\inf \{a\in \mathbb{R}:\omega(\varphi)\leq a W\}$. Then $0<a^*_-\leq 1\leq a^*_+<\infty$ and $a^*_- W\leq \omega(\varphi)\leq a^*_+ W$. We claim that $a^*_+=a^*_-=1$; otherwise, $\{a^*_+,a^*_-\}\neq \{1\}$. By the conditions of $f$, we have
		 $$
		 a^*_+ f(x,W(x))\geq f(x,\psi(\theta,x))\geq a^*_- f(x,W(x)),
		 \, \,  \forall (\theta,x,\psi)\in [-\tau,0]\times \mathbb{R}\times \omega(\varphi).
		 $$
		 We only consider the case of $a_-^*<1$ since the case of $a_+^*>1$ can be dealt with in a similar way. Note that there exists $x_0>0$ such that $$
		 f(x,\psi(\theta,x))> a^*_- f(x,W(x)), \, \,  \forall
		 (\theta,x,\psi)\in [-\tau,0]\times \mathbb{R}\times \omega(\varphi)
		  \, \,  \text{with} \, \,  |x|\geq x_0.
		 $$
		 By the definition of $Q[t,\cdot]$ and the invariance of $\omega(\varphi)$,
		 it follows  that for any $(t,x,\psi)\in (0,\infty)\times \mathbb{R} \times \omega(\varphi)$, there holds
		\begin{eqnarray*}
		Q[t,\psi](0,x)&=& S(t)[\psi(0,\cdot)](x)+ \mu\int^t_0 S(t-s)[f(\cdot,Q[s,\psi](-\tau,c\tau+\cdot))](x)
			{\rm d} s
			\\
			&>& a^*_-S(t)[W](x)+ \mu\int^t_0 S(t-s)[a^*_-f(\cdot,W(c\tau+\cdot))](x){\rm d} s
			\\
			&=& a^*_-W(x).
		\end{eqnarray*}
		As a result, we have
		$$
		\psi(\theta,x)>a^*_-W(x),\, \,  \forall
		(\theta,x,\psi)\in [-\tau,0]\times \mathbb{R} \times \omega(\varphi),
		$$
		which, together with the limiting properties of  $W$ and $\omega(\varphi)$, gives rise to  $\omega(\varphi)\geq (a^*_-+\delta_0) W$ for some $\delta_0>0$,  a contradiction. This shows that  (iv) holds true.
	\end{proof}
	
	We should point out that the first two conclusions in Theorem
	\ref{thm6.9} were obtained in \cite{hsl2020} and 
	\cite{ycw2019}, respectively, in  the special case where $\tau=0$ and $f(x,u)$ is monotone in $x$.  Further,  the third conclusion 
	on the existence of forced waves for system  \eqref{6.8} with or without time delay  seems to appear for the first time.
	
		%===============

	%======================

	{\subsection{Asymptotically homogeneous  reaction-diffusion systems}}
	
	In this section, we consider  a class of  cooperative reaction-diffusion systems:
	\begin {equation}
	\left\{
	\begin{array}{ll}
		\frac{\partial {\bf u}}{\partial t}(t,x)  =  D\frac{\partial^2 {\bf u}(t,x)}{\partial x^2}+{\bf f }(x,{\bf u}(t,x)), \, (t,x)\in (0,\infty)\times \mathbb{R}, \\
		{\bf u}(0,x) = \varphi(x),  \quad  x\in \mathbb{R},
	\end{array} \right.
	\label{6.3-coop-subh}
	\end {equation}
	where  ${\bf u}(t,x)=(u_1(t,x),u_2(t,x),\cdots,u_N(t,x))^{T}$, $D=diag(d_1,d_2,\cdots,d_N)$, and
	the   reaction term ${\bf f}=(f_1,f_2,\cdot,f_N)^{T} \in C(\mathbb{R}\times \mathbb{R}_+^N,\mathbb{R}^N)$.
	%{\color{blue} is asymptotic to ${\bf f}_\pm\in C(\mathbb{R}_+^N,\mathbb{R}^N)$, respectively, as $x\to \pm \infty$.-------delete}
	
An  $N\times N$ matrix  $A=(A_{ij})$ is said to be cooperative
if $A_{ij}\geq 0$ for all $1\leq i\ne j\leq N$. And the stability modulus
of $A$  is defined as $s(A):=\max\{{\rm Re} \lambda: \, \,  \det (A-\lambda I)=0\}$.
We always  assume that
	\begin{enumerate}
		\item[(C1)] $d_k>0$ for all $k\in \{1,2,\cdots,N\}$.
		
		\item[(C2)] ${\bf f}_\pm\in C(\mathbb{R}_+^N,\mathbb{R}^N)$, ${\bf f}(s,\cdot) \in C^1(\mathbb{R},\mathbb{R}^N)$, and ${\bf f} (s,0)=0$ for each $s\in \mathbb{R}$.
		
		\item[(C3)]  ${\rm D}_{\bf u}f(x,{\bf u}):=\left(\frac{\partial f_i(x,{\bf u})}{\partial u_j}\right)$ is a cooperative matrix for each $(x,{\bf u})\in \mathbb{R}\times \mathbb{R}_+^N$.
		
		\item[(C4)]  There exists ${\bf M}\in Int(\mathbb{R}_+^N)$ such that
		${\bf f }_\pm ( \alpha{\bf  M}) \ll  0$ and ${\bf f } (x, \alpha{\bf  M}) \ll  0$ for any $\alpha> 1$,  and
		${\bf f}(x,\cdot)$ is subhomogeneous on $\mathbb{R}^N_+$ for each $x\in \mathbb{R}$.
		\end{enumerate}
		
		For two limiting functions ${\bf f}_\pm\in C(\mathbb{R}_+^N,\mathbb{R}^N)$,
we need the following additional assumptions:
\begin{enumerate}
		\item[(C$_+$)] (i)   There exists ${\bf u}_+^*\in Int(\mathbb{R}_+^N)$ such that $\{{\bf u}\in\mathbb{R}_+^N\setminus\{0\}: {\bf f}_+({\bf u})=0\}=\{{\bf u}_+^*\}$;
(ii) $\lim\limits_{x\to \infty}{\bf f}(x,\cdot)= {\bf f}_+$ in $C^1_{loc}(\mathbb{R}_+^{N},\mathbb{R}^N)$;
  (iii) The matrix
			${\rm D}_{\bf u}{\bf f} _+({\bf u})$ is irreducible for each ${\bf u}\in \mathbb{R}_+^N$ and $s({\rm D}_{\bf u}{\bf f} _+(0))>0$.
		
		\item[(C$_-$)] (i) There exists ${\bf u}_-^*\in Int(\mathbb{R}_+^N)$ such that $\{{\bf u}\in\mathbb{R}_+^N\setminus\{0\}: {\bf f}_-({\bf u})=0\}=\{{\bf u}_-^*\}$;  (ii) $\lim\limits_{x\to -\infty}{\bf f}(x,\cdot)= {\bf f}_-$ in $C^1_{loc}(\mathbb{R}_+^{N},\mathbb{R}^N)$;
 (iii) The matrix
		${\rm D}_{\bf u}{\bf f} _-({\bf u})$ is irreducible for each ${\bf u}\in \mathbb{R}_+^N$ and $s({\rm D}_{\bf u}{\bf f} _-(0))>0$.	
	
\end{enumerate}
	
	Let  $C=BC(\mathbb{R},\mathbb{R}^N)$, and  $C_+=BC(\mathbb{R},\mathbb{R}_+^N)$.
	We consider the mild solutions of~\eqref{6.3-coop-subh} with
	initial values $\varphi\in C_+$. Here a mild solution solves the following
	integral equation with the given initial function:
	\begin {equation}\label{6.3-coop-subh-integal-eq}
	\left\{
	\begin{array}{rcl}
		{\bf u}(t,\cdot) & =& S_{\alpha_0^*}(t)[\varphi]+\int^t_0
		S_{\alpha_0^*}({t-s})[\alpha_0^* {\bf f}_{\alpha_0^*}({\bf u}(s,\cdot))]{ \mathrm {d} s}, \qquad  t\in
		\mathbb{R}_+,
		\\
		{\bf u}(0,\cdot) & = & \varphi\in C_+,
	\end{array}
	\right.
	\end {equation}
	where $\alpha_0^*>0$ is fixed such that
	$$
	\alpha_0^*E+{\rm D}_{\bf u}{\bf f}(x,{\bf u})\geq 0, \,  \, \forall (x,{\bf u})\in \mathbb{R}\times [0,{\bf M}]_{\mathbb{R}^N},
	$$
${\bf f}_{\alpha_0^*}$ is defined as
	$$
	{\bf f}_{\alpha_0^*}(x,{\bf u})={\bf u}+\frac{{\bf f}(x,{\bf u})}{\alpha_0^*},
	 \,  \,   \forall (x,{\bf u})\in \mathbb{R}\times \mathbb{R}^N_+,
	$$
	and
	$S_{\alpha_0^*}(t)$ is the semigroup generated by the following  linear reaction-diffusion system:
	\[
	\left\{
	\begin{array}{rcll}
	\frac{\partial {\bf u}}{\partial t}&=& D\frac{\partial^2 {\bf u}}{\partial x^2}-\alpha_0^* {\bf u}, \qquad & t>0,
	\\
	{\bf u}(0,\cdot)& =& \varphi\in C, &
	\end{array}
	\right.
	\]
	that is,  for $(x,\varphi)\in \mathbb{R}\times C$, we have
	\[
	\left\{
	\begin{array}{rcl}
	S_{\alpha_0^*,k}(0)[\varphi] & = & \varphi_k, \\
	S_{\alpha_0^*,k}(t)[\varphi](x) & = & \frac{\exp(-\alpha_0^* t)}{\sqrt{4\pi d_k t}}\int_{\mathbb{R}}
	\varphi_k(y)\exp \left(-\frac{(x-y)^2}{4d_k
		t}\right){\rm d }y,
	\quad t >0.
	\end{array}
	\right.
	\]
	
	It is well-known that for any given  $\varphi\in  C_+$, equation~\eqref{6.3-coop-subh}  has a
	unique mild solution on a maximal interval $[0, \eta_{\varphi;{\bf f}})$,
	%in thesense of Lunardi~\cite{l1995}
	denoted by ${\bf u}^{\varphi}(t,x;{\bf f})$, which is
	also the classical solution of~\eqref{6.3-coop-subh} on  $(0,\eta_{\varphi;{\bf f}})$
	with
	$[0,\eta_{\varphi;{\bf f}})\ni t\mapsto {\bf u}^{\varphi}(t,\cdot;{\bf f})\in C_+$ being
	continuous and $\limsup\limits_{t\to \eta_{\varphi;{\bf f}}^-
	}||{\bf u}^{\varphi}(t,\cdot;{\bf f})||=\infty$ whenever $\eta_{\varphi;{\bf f}}<\infty$.
	By the Phragm\'en-Lindel\"of type maximum principle~\cite{pw1967} and the standard comparison arguments, it easily follows that
	$$
	0\leq  {\bf u}^\varphi(t,x;{\bf f})\leq  \alpha{\bf M}, \,  \, \forall
	\alpha\in  [1,\infty),  \,  \, (t,x,\varphi)\in   [0,\eta_{\varphi;{\bf f}}) \times \mathbb{R} \times [0,\alpha  {\bf M}]_C,
	$$
	and hence, $\eta_{\varphi;{\bf f}}=\infty$.
	
	Now we introduce  the following auxiliary reaction-diffusion systems:
	\begin {equation}
	\left\{
	\begin{array}{ll}
		\frac{\partial {\bf u}}{\partial t}(t,x)  =  D\frac{\partial^2 {\bf u}(t,x)}{\partial x^2}+{\bf f }(x+z,{\bf u}(t,x)), \, (t,x,z)\in (0,\infty)\times \mathbb{R}^2, \\
		{\bf u}(0,x) = \varphi(x),  \quad  x\in \mathbb{R},
	\end{array} \right.
	\label{6.3-coop-subh-shift}
	\end {equation}

	\begin {equation}
	\left\{
	\begin{array}{ll}
		\frac{\partial {\bf u}}{\partial t}(t,x)  =  D\frac{\partial^2 {\bf u}(t,x)}{\partial x^2}+{\bf f }_+({\bf u}(t,x)), \, (t,x)\in (0,\infty)\times \mathbb{R}, \\
		{\bf u}(0,x) = \varphi(x),  \quad  x\in \mathbb{R},
	\end{array} \right.
	\label{6.3-coop-subh-+infty}
	\end {equation}
	and
	\begin {equation}
	\left\{
	\begin{array}{ll}
		\frac{\partial {\bf u}}{\partial t}(t,x)  =  D\frac{\partial^2 {\bf u}(t,x)}{\partial x^2}+{\bf f }_-({\bf u}(t,x)), \, (t,x)\in (0,\infty)\times \mathbb{R}, \\
		{\bf u}(0,x) = \varphi(x),  \quad  x\in \mathbb{R}.
	\end{array} \right.
	\label{6.3-coop-subh--infty}
	\end {equation}

	Let $Q[t,\varphi;{\bf f }]$, $Q[t,\varphi;{\bf f }_\pm]$ be the mild solutions  of~\eqref{6.3-coop-subh}, \eqref{6.3-coop-subh-+infty}  and \eqref{6.3-coop-subh--infty} with the initial value $u(0,\cdot) = \varphi\in C_+$, respectively.   For simplity, we denote
	$Q[t,\varphi;{\bf f }]$ and $Q[t,\varphi;{\bf f }_\pm]$ by  $Q[t,\varphi]$ and $Q_\pm[t,\varphi]$, respectively.
	By the standard arguments, we can
	verify the following properties  for $Q$ and $Q_\pm$.
	
	\begin{prop} \label{prop6.5} The following statements are valid:
		\begin{itemize}
			\item [{\rm (i)}]If $(t,\varphi)\in (0,\infty)\times C_+$ with $\varphi_k>0$  for all $k\in \{1,2,\cdots,N\}$, then  $Q[t,\varphi]\in C_+^\circ$, and hence, $Q[(0,\infty)\times C_+^\circ]\subseteq C_+^\circ$.   Moreover,  $Q[(0,\infty)\times (C_+\setminus\{0\})]\subseteq C_+^\circ$ if {\rm (C$_+$) } or {\rm (C$_-$) } holds true.
			
			\item [{\rm (ii)}]   $Q[\mathbb{R}_+\times C_{\alpha {\bf M}}]\subseteq C_{\alpha {\bf M}}$ and $Q_\pm[\mathbb{R}_+\times C_{\alpha {\bf M}}]\subseteq C_{\alpha {\bf M}}$ for all $\alpha \in [1,\infty)$. Moreover, $$\limsup\limits_{t\to\infty}\Big[\sup\{{\bf M }-Q[t,\varphi](x):x\in \mathbb{R}\}\Big],\quad
			\limsup\limits_{t\to\infty}\Big[\sup\{{\bf M }-Q_\pm[t,\varphi](x):x\in \mathbb{R}\}\Big]\in \mathbb{R}^N_+ ,$$  and hence, $\omega(\varphi;Q)\bigcup \omega(\varphi;Q_\pm)\subseteq C_{{\bf M}}$ for all  $\varphi\in C_{+}$, where $\omega(\varphi;Q)$ and  $\omega(\varphi;Q_\pm)$  represent the positive limiting sets of $Q$ and $Q_\pm$, respectively.
			
			\item [{\rm (iii)}] $Q$ and $Q_\pm:\mathbb{R}\times C_{r\check{1}}\to C_+$ are continuous. Moreover,   $Q[t,C_{r\check{1}}]$ and $Q_\pm[t,C_{r\check{1}}]$  are precompact in $C$ for each $(t,r)\in (0,\infty)^2$.
			
			\item [{\rm (iv)}]  $T_{-z}\circ Q[t,\cdot]\circ T_z[\varphi](\theta,x)$ is the mild solution  of~\eqref{6.3-coop-subh-shift}  for each $z\in \mathbb{R}$.
			
			\item [{\rm (v)}]  If {\rm (C$_\pm$) } holds, then $\lim\limits_{y\to \pm \infty}
			Q[t,T_{y}[\varphi]](\cdot,\cdot+y)= Q_\pm[t,\varphi]$ in $C$
			for all $(t,\varphi)\in \mathbb{R}_+\times C_+$, and hence, $Q_\pm[t,T_{y}[\varphi]](\cdot,\cdot+y)=Q_\pm[t,\varphi]$ for all $(y,t,\varphi)\in \mathbb{R}\times\mathbb{R}_+\times C_+$.

			\item [{\rm (vi)}]   $Q[t,\varphi]\geq Q[t,\psi]$ and $Q_\pm[t,\varphi]\geq Q_\pm[t,\psi]$ for all $(t,\varphi,\psi)\in \mathbb{R}_+\times C_+\times C_+$ with $\varphi\geq \psi$.

			\item [{\rm (vii)}]   $Q[t,\alpha \varphi]\geq \alpha Q[t,\varphi]$ and $Q_\pm[t,\alpha\varphi]\geq \alpha Q_\pm[t,\psi]$ for all $(\alpha,t,\varphi)\in [0,1]\times  \mathbb{R}_+\times C_{\bf M}$.
			
			%\item [{\rm (viii)}]  If  ${\bf f}(\pm s,\cdot)$ is nondecreasing in $s\in \mathbb{R}$, then  $T_{-z}[Q[t,T_z[\varphi]]]\geq Q[t,\varphi]$ for all $(\pm z,t,\varphi)\in \mathbb{R}_+^2\times C_+$.

			\item [{\rm (viii)}]  $Q_\pm[t,\varphi(-\cdot)]=Q_\pm[t,\varphi](-\cdot)$ for all $(t,\varphi)\in \mathbb{R}_+\times C_+$.
			
			\item [{\rm (ix)}]   $\lim\limits_{t\to \infty}Q_\pm[t,{\bf w}_k^\pm]\to {\bf u}_\pm^*$ in $C$ for all $k\in \mathbb{N}$ for some sequence  $\{{\bf w}_k^\pm\}_{k\in \mathbb{N}}$ in $Int(\mathbb{R}_+^N)$ with
		$\lim\limits_{k\to \infty}{\bf w}_k^\pm=0$ provided that  {\rm (C$_\pm$) } holds.
		\end{itemize}
		\end{prop}
	
	For any $\mu \in\mathbb{R}$, we define $L_{\pm,\mu}: \mathbb{R}^N\to  \mathbb{R}^N$ by
	$$
	L_{\pm,\mu}[\phi]=Q[1,e_{\phi,\mu};{\rm D}_{\bf u}{\bf f}_\pm (0)\cdot Id_{\mathbb{R}^N}](0), \,  \,  \forall \phi\in \mathbb{R}^N,
$$
where $e_{\phi,\mu}(x)=\phi e^{-\mu x}$ for all $x\in\mathbb{R}$.
	Clearly, $L_{\pm,\mu}$ is  a strongly positive linear operator  on $\mathbb{R}^N$  provided that  {\rm (C$_\pm$) } holds.
By  the Perron-Frobenius theorem, it follows that $L_{\pm,\mu}$ has the unique principal eigenvalue $\lambda_\pm(\mu)$ and  an unique  eigenfunction $\zeta_\mu^\pm\in Int(\mathbb{R}_+^N)$ associated with  $\lambda_\pm(\mu)>0$ such that  $||\zeta_\mu^\pm||=1$.
Following  \cite{w1982, lz2007}, we define
\begin{equation}\label{spreadRDS}
c^*(\pm\infty):=\inf\limits_{\mu>0}\frac{1}{\mu}\log\lambda_\pm(\mu).
\end{equation}	
Under the assumptions (C$_+$) and (C$_-$), it follows from \cite{lz2007} that  $c^*(\pm\infty)$ are the spreading speeds  of the right and left limiting  systems \eqref{6.3-coop-subh-+infty} and  \eqref{6.3-coop-subh--infty}, respectively. We can easily verify the following properties of $c^*(\pm \infty)$.
	
	%and $\underline{L}[\varphi;\alpha,\mu](x)=\min\{\alpha \zeta_\mu,L[\varphi](x)\}$ for all $(\varphi,x)\in \tilde{C}\times \mathbb{R}$.
	%By appealing to some arguments similar to \cite{lz2010}, we may verify that $c^*(\infty), c^*(-\infty)$ be the right/left spreading speeds  of  systems \eqref{6.3-coop-subh-+infty}, \eqref{6.3-coop-subh--infty}, respectively.
	%$\frac{mu^2 d+f'(0)}{\mu}=\mu d+\frac{f'(0)}{\mu}$
	\begin{prop} \label{prop6.7} {Let  $c^*(\pm \infty)$ be defined as above. Then the following statements are valid:
			\begin{itemize}
				\item [{\rm (i)}]  If $s({\rm D}_{\bf u}{\bf f}_+ (0))>0$, then
				$c^*( \infty)>0$.
				\item [{\rm (ii)}] If $s({\rm D}_{\bf u}{\bf f}_-(0))>0$, 
				then $c^*( -\infty)>0$.
			\end{itemize}}
			\end{prop}
		
	By using  \cite[Theorem 2.17-(ii)]{lz2007}, Proposition~\ref{prop6.5} and Lemma~\ref{lemm4.1}, we can easily  verify the following result.	
		
		\begin{prop} \label{prop6.5-11} The following statements are valid:
			\begin{itemize}
				\item [{\rm (i)}]    If  {\rm (C$_+$)} holds, then for each
				$t>0$, the map  $Q_+[t,\cdot]$ satisfies   {\bf (UC)} with $t c^*(\infty)$, where $c^*(\infty)$ is defined  in \eqref{spreadRDS}. Hence,  $Q[t,\cdot]$ satisfies  {\bf (NM)}.
				% Here $\#$ represents the power of a given set.
				
				\item [{\rm (ii)}]  If  {\rm (C$_-$)} holds, then for each
				$t>0$, the map $Q_-[t,\cdot]$ satisfies  {\bf (UC$_-$)} with
				$tc^*(-\infty)$, where $c^*(-\infty)$  is defined
				 in \eqref{spreadRDS}. Hence,  $Q[t,\cdot]$ satisfies  {\bf (NM$_-$)}.
		
					\end{itemize}
			\end{prop}

			Let $\mathbb{L}[t,\varphi;\zeta]$ be the mild solution of  the following linear system:
		\begin{equation}
		\left\{
		\begin{array}{rcll}
		\frac{\partial {\bf u}}{\partial t}&=& D\frac{\partial^2 {\bf u}}{\partial x^2}+\zeta(x) {\bf u}, \qquad & t>0,
		\\
		{\bf u}(0,x)& =& \varphi(x), & x\in \mathbb{R},
		\end{array}
		\right.
		\label{6.3-coop-subh-liz}
		\end{equation}
		where $\zeta\in C(\mathbb{R},\mathbb{R}^{N\times N})$ and $\varphi\in C$.  By the definitions of $\mathbb{L}$ and $\mathbb{L}_\pm$,  we can easily prove the following lemma.
		
		\begin{lemma} \label{lemma-6.3.sub}  Assume that $\zeta \in  C(\mathbb{R},\mathbb{R}^{N\times N})$ such that $\zeta(x)$ is cooperative and nonincreasing in $x\in \mathbb{R}$, and both  $\zeta(\infty):=\lim\limits_{x\to \infty}\zeta(x)$ and $\zeta(-\infty):=\lim\limits_{x\to -\infty}\zeta(x)$ are  irreducible $N\times N$ matrices. Then the  following statements are valid:
			\begin{itemize}
				\item [{\rm (i)}]
				$\mathbb{L}[t,C_+\setminus\{0\};\zeta]\subseteq C_+^\circ$ for all $t\in (0,\infty)$.
				
				\item [{\rm (ii)}]   $\mathbb{L}[t,\cdot;\zeta]|_{[-r\check{1},r\check{1}]_C}:[-r\check{1},r\check{1}]_C\to C$ is a continuous and compact map for each $(t,r)\in(0,\infty)^2$.
				% Here $\#$ represents the power of a given set.
				
				\item [{\rm (iii)}]   $\mathbb{L}[t,\cdot;\zeta]:(C,C_+,||\cdot||_{L^\infty(\mathbb{R},\mathbb{R}^N)})\to (C,C_+,||\cdot||_{L^\infty(\mathbb{R},\mathbb{R}^N)})$ is a bounded and positive operator for each $t\in \mathbb{R}_+$.

				\item [{\rm (iv)}]    $T_{-z}\circ \mathbb{L}[t,\cdot;\zeta]\circ T_z$  is nonincreasing in $z\in \mathbb{R}$.
				
				\item [{\rm (v)}]  $\lim\limits_{z\to \pm \infty}||T_{-z}\circ \mathbb{L}[t,\varphi;\zeta]\circ T_z-\mathbb{L}[t,\varphi;\zeta(\pm \infty)]||=0$ for each $(t,\varphi)\in \mathbb{R}_+\times C$.
				
				\item [{\rm (vi)}]  $\mathbb{L}[t,\varphi;{\rm D}_{\bf u}{\bf f}(x,{\bf 0})]\geq Q[t,\varphi]$ for all  $(t,\varphi)\in \mathbb{R}_+\times C_+$.
				\end{itemize}
			\end{lemma}
		
		A straightforward  computation gives rise to the following result.
		\begin{lemma} \label{lemma-6.3.sub-glc-1}  Let $a,\mu\in (0,\infty)$, ${\bf b}\in \mathbb{R}_+^N$,  and let $A\in \mathbb{R}^{N\times N}$ be a cooperative matrix. Then    $$l_{\mu,a}:\mathbb{R}\ni z\mapsto  \frac{1}{\sqrt{4\pi a }}\int_{-z}^\infty
			e^{-\mu y}\exp \left(-\frac{y^2}{4a
			}\right){\rm d }y\in \mathbb{R}$$ is a bounded and continuous function  on $\mathbb{R}$  with    $l_{\mu,a}(-\infty)=0$ and  $l_{\mu,a}(\infty)=e^{a\mu^2}$.
			Hence,  $$L_{\mu,{\bf b},A}:=e^A({\bf b}_1l_{\mu,d_k},{\bf b}_2l_{\mu,d_2},\cdots, {\bf b}_Nl_{\mu,d_N})^{T}$$ is a bounded and continuous vector-valued function  on $\mathbb{R}$  with    $L_{\mu, {\bf b},A}(-\infty)=0$ and
			$$L_{\mu,{\bf b},A}(\infty)=e^A({\bf b}_1e^{d_1 \mu^2},{\bf b}_2e^{d_2 \mu^2},\cdots, {\bf b}_Ne^{d_N \mu^2})^{T}=e^{A+\mu^2 D}{\bf b}.$$
		\end{lemma}
		
	For any given $a,x\in \mathbb{R}$, we define
		\[
		\chi_a(x)=\left\{
		\begin{array}{ll}
		1 & x>a,
		\\
		x-a+1 & a-1\leq x\leq a,
		\\
		0, \qquad & x<a-1.
		\end{array}
		\right.
		\]
		
		\begin{prop} \label{prop6.3-2-11}   $Q[1,\cdot]$ satisfies {\bf (GLC)} with $c^*(\infty)$ whence  {\rm (C$_+$)} holds and  {\bf (GLC$_-$)} with  $c^*(-\infty)$ whence  {\rm (C$_-$)} holds, where $c^*(\infty)$ and $c^*(-\infty)$ are defined as in
		 \eqref{spreadRDS}.
		\end{prop}

		\begin{proof} It suffices to prove that $Q$ satisfies  {\bf (GLC)}.  Since Lemma~\ref{lem6.10-subhomgenous}-(i) implies that
				$${\bf f}(x,{\bf u}) \leq {\rm D}_{\bf u}{\bf f}(x,{\bf 0}){\bf u},
				\quad \forall (x,{\bf u}) \in\mathbb{R}\times \mathbb{R}_+^N,
				$$
		%By Lemma~\ref{lem666.005}, as applied to $({\rm D}_{\bf u}{\bf f}(x,{\bf 0})))_{ij}$,
	it follows  that for any $l\in \mathbb{N}$,  there exists  $\zeta_l \in  C(\mathbb{R},\mathbb{R}^{N\times N})$
			such that $$\zeta_l\geq {\rm D}_{\bf u}{\bf f}(\cdot,{\bf 0}), \zeta_l(\infty)=\frac{1}{l}\check{\bf 1}+\limsup\limits_{x\to \infty}{\rm D}_{\bf u}{\bf f}(x,{\bf 0}),\zeta_l(-\infty)\in \mathbb{R}^{N\times N}, $$ and $\zeta_l(x)$ is cooperative, irreducible, and nonincreasing in $x\in \mathbb{R}$.
			%Without loss of generality, we may assume that  $\zeta(\infty)$ is strongly positive since we may choose $\zeta(\cdot)+\frac{1}{l}\check{\bf 1}$ to replace $\zeta(\cdot)$ in the process of verifying assumptions  {\bf (GLC)}.
			
			Let $(l,\mu)\in \mathbb{N}\times  (0,\infty)$
				be given. By the Riesz representation theorem and  Lemma~\ref{lemma-6.3.sub}-(iii),    there exist bounded
			nonnegative measure matrices $(m_{ij,l}(x,{\rm d}y))_{N\times N}$ and $(m_{ij,l}^\pm({\rm d}y))_{N\times N}$  such that
			for any $(x,\varphi)\in \mathbb{R}\times C$, there hold $$\mathbb{L}[1,\varphi;\zeta_l](x)=\int_{\mathbb{R}}(m_{ij,l}(x,{\rm d}y))_{N\times N}\varphi(x-y),
			$$
			and
			$$
			\mathbb{L}[1,\varphi;\zeta_l(\pm\infty)](x)=\int_{\mathbb{R}}(m_{ij,l}^\pm({\rm d}y))_{N\times N}\varphi(x-y).$$
		Define $L_l$ and $L_l^\pm:\tilde{C}\to \tilde{C}$ by
			$$
			L_l[\varphi](x)=\int_{\mathbb{R}}(m_{ij,l}(x,{\rm d}y))_{N\times N}\varphi(x-y),
			$$
			  and
			$$
			L_l^\pm[\varphi](x)=\int_{\mathbb{R}}(m_{ij,l}^\pm({\rm d}y))_{N\times N}\varphi(x-y), \forall (x,\varphi)\in \mathbb{R}\times \tilde{C}.$$
			It then follows that $Q[\varphi]\leq  L_l[\varphi]$ for all $\varphi\in  C_+\setminus\{0\}$ and
			$$
			0\ll L_l^+[\varphi]\leq T_z\circ L_l\circ T_z[\varphi]\leq L_l[\varphi]\leq L_l^-[\varphi],\,  \,  \forall  (z,\varphi)\in \mathbb{R}_+\times \tilde{C}_+\setminus\{0\}.
			$$
			The latter implies that $L_l$ and $L_l^\pm$ satisfy  {\bf (SLC)}-(iii).
			%, together with Lemma~\ref{lemma-6.3.sub-glc-1}, implies that $L_l$ and $L_l^\pm$ satisfy  {\bf (SLC)}-(iii,iv).
			Since
			$$L_l^\pm[e_{{\bf b}\cdot {\bf 1}_{[-z,\infty)},\mu}](0)=L_{\mu,{\bf b},\zeta_{l}(\pm \infty)}(z), \forall z\in \mathbb{R},
			$$
	 {\bf (SLC)}-(i) and {\bf (SLC)}-(iv) follow from Lemma~\ref{lemma-6.3.sub-glc-1}.
			
			To verify  {\bf (SLC)}-(ii), by  {\bf (SLC)}-(iii), it suffices to prove $$
			\lim\limits_{z\to \pm\infty}T_{-z}\circ L_l\circ T_{z}[e_{\check{1},\mu}]= L_l^\pm[e_{\check{1},\mu}] \, \mbox{  in }L^\infty_{loc}(\mathbb{R},\mathbb{R}), \forall \mu>0.
			$$
			Fix $\gamma,\epsilon,\mu>0$.  Note that
			$$L_l^\pm[e_{\check{1}\cdot {\bf 1}_{[-z,\infty)},\mu}](x)=e^{-\mu x}L_l^\pm[e_{\check{1}\cdot {\bf 1}_{[-z-x,\infty)},\mu}](0),
			\,  \,  \forall x,z\in \mathbb{R},
			$$
			 and
			 $$
			 L_l^\pm[e_{\check{1}\cdot {\bf 1}_{[-z,\infty)},\mu}](x)\leq L_l^\pm[e_{\check{1}\cdot {\chi}_{-z}(\cdot),\mu}](x)\leq L_l^\pm[e_{\check{1}\cdot {\bf 1}_{[-z-1,\infty)},\mu}](x),
			 	\,  \,  \forall x,z\in \mathbb{R}.
			 $$
			
			 This, together with {\bf (SLC)}-(i,iii), implies that  there exists $z_0>0$ such that
			\begin{eqnarray*}
			||L_l^+[e_{\check{1}\cdot {\chi}_{-z_0}(\cdot),\mu}](x)-L_l^+[e_{\check{1},\mu}](x) ||
		&\leq&||L_l^-[e_{\check{1}\cdot {\chi}_{-z_0}(\cdot),\mu}](x)-L_l^-[e_{\check{1},\mu}](x) ||
			      \\
	        &\leq&||L_l^-[e_{\check{1}\cdot {\bf 1}_{[-z_{0},\infty)},\mu}](x)-L_l^-[e_{\check{1},\mu}](x) ||
	                        \\
	        &=&||L_l^-[e_{\check{1}\cdot {\bf 1}_{[-z_{0}-x,\infty)},\mu}](0)-L_l^-[e_{\check{1},\mu}](0) ||e^{-\mu x}
	                        \\
		&<&\frac{\epsilon}{3},\quad  \forall x\in[-\gamma,\gamma].
		       \end{eqnarray*}
					  In view of  {\bf (SLC)}-(iii) and the choice of $z_{0}$,
					  % and {\bf (SLC)}-(iii),
					  we easily see that  	for any $(x,z)\in[-\gamma,\gamma]\times \mathbb{R}$, there holds
		     \begin{eqnarray*}
	&& || T_{-z}\circ L_l\circ T_z[e_{\check{1}\cdot {\chi}_{-z_0}(\cdot),\mu}](x)-T_{-z}\circ L_l\circ T_z[e_{\check{1},\mu}](x) ||
	\\
	&&=|| T_{-z}\circ L_l\circ T_z[e_{\check{1},\mu}-e_{\check{1}\cdot {\chi}_{-z_0}(\cdot),\mu}](x) ||
	\\
	&&\leq || L_l^-[e_{\check{1},\mu}-e_{\check{1}\cdot {\chi}_{-z_0}(\cdot),\mu}](x) ||
	<\frac{\epsilon}{3}.
			  \end{eqnarray*}
			Since $e_{\check{1}\cdot {\chi}_{-z_0}(\cdot),\mu}\in C$,
			Lemma~\ref{lemma-6.3.sub}-(v) implies that there exists $z^*>0$ such that   $$||T_{-z}\circ L_l\circ T_z[e_{\check{1}\cdot {\chi}_{-z_0}(\cdot),\mu}](x)-L_l^\pm[e_{\check{1}\cdot {\chi}_{-z_0}(\cdot),\mu}](x)||<\frac{\epsilon}{3}, \forall (x,\pm z)\in[-\gamma,\gamma]\times [z^*,\infty).$$
			Thus, we can verify that for any $x\in [-\gamma,\gamma]$, there holds
			%$$||T_{-z}\circ L_l\circ T_z[e_{\check{1}\cdot {\chi}_{-z_0}(\cdot),\mu}](x)-[e_{\check{1},\mu}](x)||<\epsilon, \forall z\in [z^*,\infty) $$ and $$ ||T_{-z}\circ L_l\circ T_z[e_{\check{1},\mu}](x)-L_l^-[e_{\check{1}\cdot {\chi}_{-z_0}(\cdot),\mu}](x)||<\epsilon, \forall z\in (-\infty,-z^*].$$This, together with  {\bf (SLC)}-(iii), implies that
			$$ ||T_{-z}\circ L_l\circ T_z[e_{\check{1},\mu}](x)-L_l^\pm[e_{\check{1},\mu}](x)||<\epsilon,  \forall  \pm z\in [z^{*},\infty).$$
			Consequently,  the  arbitrariness of $\gamma$ and $\epsilon$  completes the proof.
		\end{proof}
		
		Now we are in a position to present the result on the propagation dynamics
		for system  ~\eqref{6.3-coop-subh} in  the bilateral {\bf (UC)}/{\bf (UC)} case.
		\begin{thm}\label{thm6.5}
			Assume that (C$_+$) and (C$_-$) hold and let  $c^*(\pm \infty)$ be defined as in  \eqref{spreadRDS}. Then the following statements are valid:
			\begin{itemize}
				\item [{\rm (i)}]
				If $\varphi\in C_+\setminus \{0\}$ and $\varepsilon\in (0,\min\{c^*(\infty),c^*(-\infty)\})$, then $$\lim\limits_{\alpha\rightarrow \infty}
				\max\{||Q[t,\varphi](x)-{\bf u}_+^*\cdot {\bf {1}}_{\mathbb{R}_+}(x)-{\bf u}^*_-\cdot (1-{\bf {1}}_{\mathbb{R}_+}(x))||:(t,x)\in \mathcal{B}_{\alpha,\varepsilon}\}= 0,$$ where $\mathcal{B}_{\alpha,\varepsilon}=\{(t,x)\in \mathbb{R}_+\times \mathbb{R}: \, \,  t\geq \alpha, \,  x\in [\alpha,t(c^*(\infty)-\varepsilon)]\bigcup [t(-c^*(-\infty)+\varepsilon),-\alpha]\}$.

				\item [{\rm (ii)}]
				If $\varepsilon\in (0,\infty)$ and $\phi\in C_+$ has the compact support, then $$\lim\limits_{t\rightarrow \infty}
				\sup\left\{||Q[t,\varphi](x)||:x\in [t(c^*(\infty)+\varepsilon), \infty)\bigcup (-\infty,-t(c^*(-\infty)+\varepsilon)]\right\}= 0.$$
				%Moreover, if $\max\{c^*(\infty),c^*(-\infty)\}<0$, then $\lim\limits_{\alpha\rightarrow \infty}
				%\Big[\sup\{||Q[t,\varphi](\cdot,x)||:(t,x)\in \mathbb{R}\times \mathbb{R} \mbox{ with }t \geq \alpha
				%\mbox{ and } |x| \geq \alpha  \}\Big]$.
				
				\item [{\rm (iii)}]  $Q$ has a  nontrivial  equilibrium points $W$ in $C_{\bf M}^\circ$ such that
				$W(\infty)={\bf u}^*_+$ and $W(-\infty)={\bf u}^*_-$.
				
				\item [{\rm (iv)}] If $(\varepsilon,\varphi)\in (0,\min\{c^*(\infty),c^*(-\infty)\})\times (C_{+}\setminus \{0\})$, then  $$\lim\limits_{t\rightarrow \infty}\sup\left\{||Q[t,\varphi](x)-W(x)||: t(-c^*(-\infty)+\varepsilon)\leq x \leq t(c^*(\infty)-\varepsilon)\right\}= 0,$$ where $W$ is defined as in (iii).
				
			\end{itemize}
			
		\end{thm}

		\begin{proof}
			(i) follows from Theorem~\ref{thm4.1-bil-uc/uc}-(ii).
			
			(ii) follows from Theorem~\ref{thm4.1-bil-meal/meal}-(i).
			
			(iii)  follows from Remark~\ref{rem4.1-sp-monoty}.
			
			%{\color{blue}By applying the former statement in Proposition~\ref{prop3.5}-(i) to $Q[1,\cdot]$ and $\mathcal{S}\circ Q[1,\cdot]  \circ\mathcal{S} $,  there exist $W_-$ and $W_+$ in $C_+$ such that $W_\pm(\pm \infty)=\bf{u}_\pm$ and  $Q[n,W_\pm]=W_\pm, \forall n\in \mathbb{N}$. Let $\mathcal{K}=\bigcap\limits_{t\in [0,1]} ([Q[t,W_+],{\bf M}]_C\bigcap  [Q[t,W_-],{\bf M}]_C)$. Then $\mathcal{K}\subseteq [0,{\bf M]}_C\setminus\{0\}$ is a nonempty, closed, convex, and positively invariant set of $Q_t$, and hence by Proposition~\ref{prop6.5}-(ii), we know that  $\{Q_t\}_{t\in\mathbb{R}_+}$ has an  equilibrium point  $W$ in $\mathcal{K}$ with $\liminf\limits_{x\to \infty}W(\cdot,x)\geq {\bf u}_+^*$ and $\liminf\limits_{x\to-\infty}W(\cdot,x)\geq {\bf u}^*_-$.By applying Lemma~\ref{lemm3.1-999000} to  $\{(Q[1,\cdot],W)\}$ and $\{(\mathcal{S}\circ Q[1,\cdot] \circ \mathcal{S},\mathcal{S}[W])\}$, we have $W(\cdot,\infty)={\bf u}_+^*$ and $W(\cdot,-\infty)={\bf u}^*_-$}
				
			(iv) We only need to prove {\bf (GAS-CSF)} in $C_{\bf M}\setminus \{0\}$ due to Theorem~\ref{thm4.1-bil-gas}-(i) and Proposition~\ref{prop6.5}-(i). More precisely, it suffices to show that
			$$
			\omega(\varphi):=\bigcap\limits_{s\in \mathbb{R}_+}Cl(\{Q[t,\varphi]:t\geq s\})=\{W\},
				\,  \,  \forall  \varphi\in C_{\bf M}\setminus \{0\}.
			$$
			Fix $\varphi\in C_{\bf M}\setminus \{0\}$.
			Let
			$$a^*_-=\sup \{a\in \mathbb{R}:\omega(\varphi)\geq a W\}$$
			 and
			 $$a_+^*=\inf \{a\in \mathbb{R}:\omega(\varphi)\leq a W\}.$$
			 By  (i) and (iii), we then have
			 $$0<a^*_-\leq 1\leq a^*_+<\infty,	\quad
			 a^*_- W\leq \omega(\varphi)\leq a^*_+ W.
			 $$
			 This, together with Proposition~\ref{prop6.5}-(vii), implies  that
			$$
			Q[t,a^*_- W]\geq a^*_- W, \quad Q[t,a^*_+ W]\leq a^*_+ W,
			\quad  \forall  t\in \mathbb{R}_+.
			$$
			Thus, there exist two  equilibrium points $\overline{W},\underline{W}$ in $C_+^\circ$ such that
			$$\underline{W}:=\lim\limits_{t\to \infty}Q[t,a^*_-W]\leq W\leq \overline{W}:=\lim\limits_{t\to \infty}Q[t,a^*_+W]$$
			 and $\underline{W}\leq\omega(\varphi)\leq \overline{W}$. Then  (i)
			 implies that
			 $$\overline{W}(\infty)=\underline{W}(\infty)={\bf u}_+^*, \quad \overline{W}(-\infty)=\underline{W}(-\infty)={\bf u}_-^*.$$
			
			To finish the proof,  it suffices to prove $a^*:=\sup\{a>0:\underline{W}\geq a\overline{W}\}=1$. Otherwise, $a^*\in (0,1)$ and $\overline{W}\geq \underline{W}\geq a^*\overline{W}$. Thus, by the subhomogeneity of $f$ and the choice of $\alpha_{0}^{*}$, we have ${\bf f }(\cdot, a^{*}\overline{W})\geq a^{*}{\bf f} (\cdot, \overline{W})$ and ${\bf f }_{\alpha_{0}^{*}}(\cdot,\underline{W})\geq a^{*}{\bf f}_{\alpha_{0}^{*}} (\cdot, \overline{W})$.
		In view of  (C$_+$) and (C$_-$),
			we easily see that
			$$u^*_+ +\frac{{\bf f }_+(u^*_+)}{\alpha^*_0}\gg a^*[u^*_+ + \frac{{\bf f }_+(u^*_+)}{\alpha^*_0}]
			$$
		and hence there exists $x_{0}>0$ such that
		${\bf f }_{\alpha_{0}^{*}}(x,\underline{W}(x))\gg a^{*}{\bf f}_{\alpha_{0}^{*}} (x, \overline{W}(x)), \forall |x|\geq x_{0}$.

		 It follows from \eqref{6.3-coop-subh-integal-eq} and the definition of $Q[t,\cdot]$ that  for any $x\in \mathbb{R}$, there holds
			\begin{eqnarray*}
				\underline{W}(x)&=& Q[1,\underline{W}](x)
				\\
				&\geq &S_{\alpha_0^*}(1)[\underline{W}](x)+\int^1_0 S_{\alpha_0^*}({1-s})[\alpha_0^* {\bf f}_{\alpha_0^*}(\cdot, \underline{W}(\cdot))](x){ \mathrm {d} s}
				\\
				&\gg & a^*S_{\alpha_0^*}(1)[\overline{W}](x)+a^*\int^1_0 S_{\alpha_0^*}({1-s})[\alpha_0^* {\bf f}_{\alpha_0^*}(\cdot, \overline{W}(\cdot))](x){ \mathrm {d} s}
				\\
				&=& a^*\overline{W}(x).
			\end{eqnarray*}
			This, together with the limits of $\underline{W}$ and  $\overline{W}$,
			gives rise to  $\underline{W}\geq (a^*+\epsilon_0)\overline{W}$ for some $\epsilon_0>0$, a contradiction. As a result, (iv) holds true.
			\end{proof}

Finally, we remark that the first two conclusions
in Theorem \ref{thm6.5} were established in 
\cite{hsl2020} and \cite{ycw2019}, respectively, in the special case where $N=1$ and ${\bf f}(x,u)$ is monotone in $x$.
Here we proved that  these two conclusions hold true no matter 
whether ${\bf f}(x,u) $ is monotone in $x$. Moreover,  the third conclusion on the existence of the nontrivial steady states for system \eqref{6.3-coop-subh} seems to appear for the first time.
		
\

\noindent
{\bf Acknowledgements.} T.  Yi's research is supported by the National Natural Science Foundation of  China (NSFC 11971494),  and X.-Q. Zhao's research is supported in part by the NSERC of Canada (RGPIN-2019-05648).

\end{document}